\documentclass[twoside,a4paper,reqno,11pt]{amsart} 
\usepackage[top=28mm,right=28mm,bottom=28mm,left=28mm]{geometry}
\usepackage{microtype}
\usepackage{amsfonts, amsmath, amssymb, amsbsy, mathrsfs, bm, latexsym, stmaryrd, hyperref, array, enumitem, textcomp, url}
\usepackage[table]{xcolor}

\newcommand{\br}{\genfrac{[}{]}{0pt}{}}

\renewcommand{\a}{\alpha}
\renewcommand{\b}{\beta}

\newcommand{\e}{\epsilon}
\renewcommand{\l}{\lambda}

\renewcommand{\O}{\Omega}

\newcommand{\la}{\langle}
\newcommand{\ra}{\rangle}
\newcommand{\C}{\mathcal{C}}
\renewcommand{\L}{\Lambda}
\newcommand{\leqs}{\leqslant}
\newcommand{\geqs}{\geqslant}

\newcommand{\normeq}{\trianglelefteqslant}

\newcommand{\vs}{\vspace{3mm}}

\makeatletter
\newcommand{\imod}[1]{\allowbreak\mkern4mu({\operator@font mod}\,\,#1)}
\makeatother

\theoremstyle{plain}
\newtheorem{theorem}{Theorem} 
 
\newtheorem{corol}[theorem]{Corollary}
\newtheorem{thm}{Theorem}[section] 
\newtheorem{lem}[thm]{Lemma}
\newtheorem{prop}[thm]{Proposition}

\newtheorem*{theorem*}{Theorem} 
\newtheorem*{conj*}{Conjecture}

\theoremstyle{definition}
\newtheorem{rem}[thm]{Remark}
\newtheorem{defn}[thm]{Definition}

\newtheorem{remk}{Remark}

\begin{document}

\title[Fixed point ratios for finite primitive groups]{Fixed point ratios for finite primitive groups \\ and applications}

\author{Timothy C. Burness}
\address{T.C. Burness, School of Mathematics, University of Bristol, Bristol BS8 1UG, UK}
\email{t.burness@bristol.ac.uk}

\author{Robert M. Guralnick}
\address{R.M. Guralnick, Department of Mathematics, University of Southern California, Los Angeles, CA 90089-2532, USA}
\email{guralnic@usc.edu}

\date{\today} 

\begin{abstract}
Let $G$ be a finite primitive permutation group on a set $\O$ and recall that the fixed point ratio of an element $x \in G$, denoted ${\rm fpr}(x)$, is the proportion of points in $\O$ fixed by $x$. Fixed point ratios in this setting have been studied for many decades, finding a wide range of applications. In this paper, we are interested in comparing ${\rm fpr}(x)$ with the order of $x$. Our main theorem classifies the triples $(G,\O,x)$ as above with the property that $x$ has prime order $r$ and ${\rm fpr}(x) > 1/(r+1)$. There are several applications. Firstly, we extend earlier work of Guralnick and Magaard by determining the primitive permutation groups of degree $m$ with minimal degree at most $2m/3$. Secondly, our main result plays a key role in recent work of the authors (together with Moret\'{o} and Navarro) on the commuting probability of $p$-elements in finite groups. Finally, we use our main theorem to investigate the minimal index of a primitive permutation group, which allows us to answer a question of Bhargava.
\end{abstract}

\maketitle

\setcounter{tocdepth}{1}
\tableofcontents

\section{Introduction}\label{s:intro}

Let $G \leqs {\rm Sym}(\O)$ be a finite transitive permutation group with point stabilizer $H$. For $x \in G$, we write
\[
{\rm fpr}(x) = \frac{|C_{\O}(x)|}{|\O|} = \frac{|x^G \cap H|}{|x^G|}
\]
for the \emph{fixed point ratio} of $x$, where $C_{\O}(x) = \{\a \in \O \,:\, \a^x = \a\}$ is the set of fixed points of $x$ and $x^G$ denotes the conjugacy class of $x$ in $G$. Sometimes we will write ${\rm fpr}(x,\O)$ if we wish to highlight the permutation domain $\O$.

Fixed point ratios have been extensively studied for many decades, finding a wide range of applications. In one direction, we can view ${\rm fpr}(x)$ as the probability that a random element in $\O$ is fixed by $x$ and this explains why fixed point ratios often arise naturally in a probabilistic setting. For example, upper bounds on fixed point ratios are a key ingredient in a powerful probabilistic approach for bounding the base size of a  finite permutation group. This method was originally introduced by Liebeck and Shalev in \cite{LSh} and it has played a major role in recent proofs of influential base size conjectures of Cameron, Kantor and Pyber. Fixed point ratios have also turned out to be very useful for studying the generation and random generation properties of finite groups. For example, bounds on fixed point ratios are applied extensively in \cite{BGH}, which provides the final step in the proof of a conjecture of Breuer, Guralnick and Kantor \cite{BGK} on $\frac{3}{2}$-generated finite groups. In a different direction, fixed point ratios have also been used to study the structure of monodromy groups of coverings of the Riemann sphere, playing a prominent role in the proof of the Guralnick-Thompson genus conjecture \cite{FM0}. We refer the reader to the survey article \cite{Bur18} for a more detailed discussion of these applications.

In this paper, we study fixed point ratios in the setting where $G \leqs {\rm Sym}(\O)$ is a primitive permutation group. Recall that a transitive group $G$ is primitive if $\O$ has no nontrivial $G$-invariant partition (equivalently, the point stabilizer $H$ is a maximal subgroup of $G$). The structure and action of a primitive group is described by the  Aschbacher-O'Nan-Scott theorem, which divides the finite primitive groups into several families. The almost simple primitive groups form one of these families and there is an extensive literature on the corresponding fixed point ratios, stretching back several decades. 

First recall that $G$ is \emph{almost simple} if there exists a nonabelian finite simple group $G_0$ (the \emph{socle} of $G$) such that $G_0 \normeq G \leqs {\rm Aut}(G_0)$. The possibilities for $G_0$ are determined by the classification of finite simple groups; $G_0$ is either an alternating group, a sporadic group or a group of Lie type (classical or exceptional). When studying fixed point ratios in this setting, it is natural to partition the primitive almost simple classical groups into two collections (the \emph{subspace} and \emph{non-subspace} actions), according to the action of $H \cap G_0$ on the natural module $V$ for $G_0$. Roughly speaking, the action of $G$ on $\O$ is a subspace action if $H \cap G_0$ acts reducibly on $V$, which allows us to identify $\O$ with a set of subspaces (or pairs of subspaces) of $V$, otherwise the action is non-subspace. We refer the reader to Definition \ref{d:sub} for the formal definition of a subspace action that we will work with in this paper. 

Before stating our main results, let us briefly highlight some of the earlier work on fixed point ratios for almost simple primitive groups of Lie type. So let $G \leqs {\rm Sym}(\O)$ be such a group, where $G_0$ is a group of Lie type over the finite field $\mathbb{F}_q$ of order $q$. One of the main results in this setting is due to Liebeck and Saxl \cite{LS91}, which states that 
\[
{\rm fpr}(x) \leqs \frac{4}{3q}
\]
for all nontrivial elements $x \in G$, with a small list of known exceptions, mainly involving groups with socle $G_0 = {\rm L}_{2}(q)$. This bound is essentially best possible. For example, if $G = {\rm L}_n(q)$, $x$ is a transvection and $\O$ is the set of $1$-dimensional subspaces of $V$, then it is easy to show that ${\rm fpr}(x) = (q^{n-1}-1)/(q^n-1)$, which is roughly $1/q$. However, stronger bounds can be obtained by imposing some additional conditions on $G$. For instance, an in-depth analysis of fixed point ratios for primitive actions of exceptional groups of Lie type is presented in \cite{LLS}. For classical groups, we refer the reader to \cite{FM, GK} for a more detailed treatment of fixed point ratios for subspace actions. 

For non-subspace actions of classical groups, a key theorem is due to Liebeck and Shalev \cite{LSh}, which states that there exists a universal constant $\e>0$ (independent of $G$) such that 
\[
{\rm fpr}(x) < |x^G|^{-\e}
\]
for all $x \in G$ of prime order. An effective version of this result is established in the series of papers \cite{Bur1,Bur2,Bur3,Bur4}, which shows that a constant $\e \sim 1/2$ is essentially best possible (see Theorem \ref{t:fpr}). The proofs of these results rely heavily on Aschbacher's celebrated subgroup structure theorem \cite{asch} for finite classical groups, which divides the possibilities for the point stabilizer into several subgroup collections.

With a view towards new applications, in this paper we seek an upper bound on ${\rm fpr}(x)$ that is given in terms of the order of $x$. Moreover, we want a bound that applies to all finite primitive groups. With this aim in mind, we present the following result, which is the main theorem of this paper. (Note that Table \ref{tab:class} is presented in Section \ref{ss:res}.)

\begin{theorem}\label{t:main}
Let $G \leqs {\rm Sym}(\O)$ be a finite primitive permutation group with point stabilizer $H$ and let $x \in G$ be an element of prime order $r$. Then either
\begin{equation}\label{e:bound}
{\rm fpr}(x) \leqs \frac{1}{r+1}
\end{equation}
or one of the following holds (up to permutation isomorphism):
\begin{itemize}\addtolength{\itemsep}{0.2\baselineskip}
\item[{\rm (i)}] $G$ is almost simple and one of the following holds: 

\vspace{1mm}

\begin{itemize}\addtolength{\itemsep}{0.2\baselineskip}
\item[{\rm (a)}] $G = S_n$ or $A_n$ acting on $\ell$-element subsets of $\{1, \ldots, n\}$ with $1 \leqs \ell < n/2$.
\item[{\rm (b)}] $G = S_n$, $H = S_{n/2} \wr S_2$, $x$ is a transposition and 
\[
{\rm fpr}(x) = \frac{1}{3} + \frac{n-4}{6(n-1)}.
\]
\item[{\rm (c)}] $G = {\rm M}_{22}{:}2$, $H = {\rm L}_{3}(4).2_2$, $x \in \emph{\texttt{2B}}$ and ${\rm fpr}(x) = 4/11$. 
\item[{\rm (d)}] $G$ is classical in a subspace action and $(G,H,x,{\rm fpr}(x))$ is listed in Table \ref{tab:class}.
\end{itemize}

\item[{\rm (ii)}] $G=V{:}H$ is an affine group with socle $V=(C_p)^d$ and point stabilizer 
$H \leqs {\rm GL}_{d}(p)$, $r=p$, $x$ is conjugate to a transvection in $H$ and ${\rm fpr}(x) = 1/r$.
\item[{\rm (iii)}] $G \leqs L \wr S_k$ is a product type primitive group with its product action on $\O = \Gamma^k$, where $k \geqs 2$ and $L \leqs {\rm Sym}(\Gamma)$ is one of the almost simple primitive groups in part (i).
\end{itemize}
\end{theorem}

\begin{remk}\label{r:main}
Some remarks on the statement of Theorem \ref{t:main} are in order.
\begin{itemize}\addtolength{\itemsep}{0.2\baselineskip}
\item[{\rm (a)}] In part (i)(a), it is plain to see that there are many exceptions to the bound in \eqref{e:bound}. For example, if $G = S_n$ and $\ell=1$, then ${\rm fpr}(x) = 1-2/n$ when $x$ is a transposition. More generally, it is straightforward to show that ${\rm fpr}(x)$ is maximal when $x$ is an $r$-cycle (or a double transposition if $r=2$ and $G = A_n$) and it is easy to compute ${\rm fpr}(x)$ in this case (see Proposition \ref{p:alt2} and Remark \ref{r:alt2}).
\item[{\rm (b)}] In part (i)(c), we use the standard Atlas \cite{ATLAS} notation. As noted above, Table \ref{tab:class} in (i)(d) is presented in Section \ref{ss:res} and we refer the reader to Remark \ref{r:tab} for information on the notation adopted in this table. It is worth noting that most of the special cases in Table \ref{tab:class} correspond to the action of $G$ on a set of $1$-dimensional subspaces (or hyperplanes) of the natural module $V$ and the relevant elements $x \in G$ with ${\rm fpr}(x) > (r+1)^{-1}$ typically have an eigenspace on $V$ of codimension $1$.
\item[{\rm (c)}] Let $x = (x_1, \ldots, x_k)\pi \in G$ be an element of prime order $r$, where $G \leqs L \wr S_k$ is a product type group as in part (iii). Let $J$ be a point stabilizer in the action of $L$ on $\Gamma$. In Section \ref{s:product} we will show that ${\rm fpr}(x) > (r+1)^{-1}$ only if $\pi=1$, in which case 
${\rm fpr}(x) = \prod_i {\rm fpr}(x_i,\Gamma)$. Moreover, Proposition \ref{p:pt2} states that  either $L$ is permutation isomorphic to $S_n$ or $A_n$ acting on $\ell$-element subsets of $\{1, \ldots, n\}$, or $x$ is conjugate to  $(x_1, 1, \ldots, 1)$ and $(L,J,x_1)$ is one of the special cases arising in part (b), (c) or (d) of Theorem \ref{t:main}(i).
\item[{\rm (d)}] The special cases arising in Theorem \ref{t:main} are described up to permutation isomorphism in order to avoid unnecessary repetition. For instance, if $G = A_8$ and $H = {\rm AGL}_{3}(2)$, then either ${\rm fpr}(x) \leqs (r+1)^{-1}$, or $x$ is an involution with cycle-shape $(2^4)$ and ${\rm fpr}(x) = 7/15$. But here $G$ is permutation isomorphic to ${\rm L}_{4}(2)$ acting on the set of $1$-dimensional subspaces of the natural module (with $x$ corresponding to a transvection), so this case is included in part (i)(d). Similarly, consider the case where $G = {\rm Sp}_{4}(2)$ and $H = {\rm O}_{4}^{\e}(2)$ is a subspace subgroup. If $\e=+$ then $G$ is permutation isomorphic to $S_6$ acting on the set of partitions of $\{1, \ldots, 6\}$ into two subsets of size $3$ (as in part (i)(b) of Theorem \ref{t:main}), and it is permutation isomorphic to $S_6$ in its natural action on $\{1, \ldots, 6\}$ when $\e=-$ (and therefore included in part (i)(a)). 
\end{itemize}
\end{remk}

The following result is an immediate corollary of Theorem \ref{t:main}.

\begin{corol}\label{c:main}
Let $G$ be a finite primitive permutation group and let $x \in G$ be an element of prime order $r$. Then either 
\[
{\rm fpr}(x) \leqs \frac{1}{\sqrt{r+1}}, 
\]
or $G$ is a subgroup of $S_n \wr S_k$ containing $(A_n)^k$ with $k \geqs 1$, where the action of $S_n$ is on $\ell$-element subsets of $\{1, \ldots, n\}$ and the wreath product has the product action of degree $\binom{n}{\ell}^k$.
\end{corol}

\begin{remk}
Consider the special case $G \leqs S_n \wr S_k$ arising in the statement of Corollary \ref{c:main}. Let $\Gamma$ be the set of $\ell$-element subsets of $\{1, \ldots, n\}$ and note that we may assume $1 \leqs \ell < n/2$. By combining Proposition \ref{p:alt2} with the proof of Proposition \ref{p:pt2}, we deduce that ${\rm fpr}(x) \leqs (r+1)^{-1}$ if $r > n-\ell$, so we may assume $r \leqs n-\ell$. Then ${\rm fpr}(x)$ is maximal when $x$ is conjugate to an element in $(S_n)^k$ of the form $(y, 1, \ldots, 1)$ with $y \in S_n$ an $r$-cycle, in which case ${\rm fpr}(x) = {\rm fpr}(y,\Gamma)$. An expression for ${\rm fpr}(y,\Gamma)$ is given in part (ii) of Proposition \ref{p:alt2} and we deduce that ${\rm fpr}(x) \leqs 1 - r/n$ (see Remark \ref{r:alt2}).
\end{remk}

We also obtain the following result on almost simple primitive groups.

\begin{corol}\label{c:main0}
Let $G \leqs {\rm Sym}(\O)$ be a finite almost simple primitive permutation group with point stabilizer $H$ and let $x \in G$ be an element of prime order $r$. Then either
\[
{\rm fpr}(x) \leqs \frac{1}{r}
\]
or one of the following holds (up to permutation isomorphism):
\begin{itemize}\addtolength{\itemsep}{0.2\baselineskip}
\item[{\rm (i)}] $G = S_n$ or $A_n$ acting on $\ell$-element subsets of $\{1, \ldots, n\}$ with $1 \leqs \ell < n/2$.
\item[{\rm (ii)}] $G$ is a classical group in a subspace action and $(G,H,x,{\rm fpr}(x))$ is listed in Table \ref{tab:subb2}.
\end{itemize}
\end{corol}

\renewcommand{\arraystretch}{1.2}
{\small \begin{table}
\[
\begin{array}{lcccll} \hline
G_0 & H & x & r & {\rm fpr}(x) & \mbox{Conditions} \\ \hline
{\rm L}_{2}(q) & P_1 & (\omega, \omega^{-1}) & q-1 & \frac{1}{q-1} +\frac{q-3}{q^2-1} & q \geqs 8 \\ 
{\rm U}_{4}(2) & P_2 & \tau & 2 & \frac{5}{9} & G = {\rm U}_4(2).2  \\
{\rm Sp}_{n}(2) & {\rm O}_{n}^{-}(2) & (J_2,J_1^{n-2}) & 2 & \frac{1}{2}+\frac{1}{2(2^{n/2}-1)} & n \geqs 6 \\
& & (\Lambda, I_{n-2}) & 3 & \frac{5}{14} & n=6 \\
\O_{n}^{-}(2) & P_1 & (J_2,J_1^{n-2}) & 2 & \frac{1}{2}+\frac{1}{2(2^{n/2}+1)} & G = {\rm O}_{n}^{-}(2) \\
\O_n^{+}(2) & N_1 & (J_2,J_1^{n-2}) & 2 & \frac{1}{2}+\frac{1}{2(2^{n/2}-1)} & G = {\rm O}_{n}^{+}(2) \\ \hline
\end{array}
\]
\caption{The subspace actions in part (ii) of Corollary \ref{c:main0}}
\label{tab:subb2}
\end{table}}
\renewcommand{\arraystretch}{1}

\begin{remk}
Let us briefly comment on the notation used in Table \ref{tab:subb2}. Firstly, $G_0$ denotes the socle of $G$. In the first row, $r = q-1 \geqs 7$ is a Mersenne prime, $H = P_1$ is a Borel subgroup and $x$ is any element of order $r$. In the second row, $H=P_2$ is the stabilizer of a $2$-dimensional totally singular subspace of the natural module and $x$ is an involutory graph automorphism with $C_{G_0}(x) = {\rm Sp}_4(2)$. Next, in the third row $x$ is a transvection (and similarly in rows 5 and 6), while $x$ is an element of order $3$ with an $(n-2)$-dimensional $1$-eigenspace on the natural module in the fourth row. In the final row, $H$ is the stabilizer of a nonsingular $1$-space.
\end{remk}

We now turn to some of the applications of Theorem \ref{t:main}. One of our motivations for seeking a bound as in Theorem \ref{t:main} stems from a widely applied theorem of Guralnick and Magaard \cite{GM} on the minimal degree of a finite primitive permutation group $G \leqs {\rm Sym}(\O)$. Recall that the \emph{minimal degree} of $G$, denoted $\mu(G)$, is the minimal number of points moved by a nonidentity element of $G$. This is a classical invariant in permutation group theory, which has been investigated by many authors for more than a century (for example, see Babai \cite{Babai}, Bochert \cite{Bochert}, Jordan \cite{Jordan} and
Manning \cite{Manning}). The main theorem of \cite{GM} determines the primitive groups $G$ of degree $m$ with $\mu(G) \leqs m/2$, extending an earlier result of Liebeck and Saxl \cite{LS91}, which describes the groups with $\mu(G) \leqs m/3$. In order to do this, Guralnick and Magaard determine the finite primitive groups $G$ with the property that
\[
{\rm fpr}(x) > \frac{1}{2}
\]
for some nonidentity element $x \in G$. 

We can use Theorem \ref{t:main} to determine all the primitive groups that contain a nonidentity element $x$ with ${\rm fpr}(x) > 1/3$, which allows us to establish the following result on the minimal degree of a finite primitive permutation group. This can be viewed as a natural extension of the earlier work of Liebeck and Saxl \cite{LS91} and Guralnick and Magaard \cite{GM}.

\begin{theorem}\label{t:mindeg}
Let $G \leqs {\rm Sym}(\O)$ be a finite primitive permutation group of degree $m$ with point stabilizer $H$ and minimal degree $\mu(G)$. Then either $\mu(G) \geqs 2m/3$, or one of the following holds (up to permutation isomorphism):
\begin{itemize}\addtolength{\itemsep}{0.2\baselineskip}
\item[{\rm (i)}] $G = S_n$ or $A_n$ acting on $\ell$-element subsets of $\{1, \ldots, n\}$ with $1 \leqs \ell < n/2$.
\item[{\rm (ii)}] $G = S_n$, $H = S_{n/2} \wr S_2$ and 
\[
\mu(G) = \frac{1}{4}\left(1+\frac{1}{n-1}\right)\frac{n!}{(n/2)!^2}.
\]
\item[{\rm (iii)}] $G = {\rm M}_{22}{:}2$, $H = {\rm L}_{3}(4).2_2$, $m=22$ and $\mu(G)=14$.
\item[{\rm (iv)}] $G$ is an almost simple classical group in a subspace action and $(G,H,m,\mu(G))$ is listed in Table \ref{tab:md}, where $G_0$ is the socle of $G$.
\item[{\rm (v)}] $G = V{:}H$ is an affine group with socle $V=(C_2)^d$, $H \leqs {\rm GL}_d(2)$ contains a transvection and $\mu(G) = 2^{d-1} = m/2$.
\item[{\rm (vi)}] $G \leqs L \wr S_k$ is a product type primitive group with its product action on $\O = \Gamma^k$, where $k \geqs 2$ and $L \leqs {\rm Sym}(\Gamma)$ is one of the almost simple primitive groups in parts (i)-(iv).
\end{itemize}
\end{theorem}

\renewcommand{\arraystretch}{1.2}
{\small \begin{table}
\[
\begin{array}{lllll} \hline
G_0 & H & m & \mu(G) & \mbox{Conditions} \\ \hline
{\rm L}_{n}(2) & P_1 & 2^n-1 & 2^{n-1} & n \geqs 3 \\ 
{\rm L}_{n}(3) & P_1 & \frac{1}{2}(3^n-1) & 3^{n-1}-1 & \mbox{$n \geqs 3$, $r \in G$} \\  
{\rm U}_4(q) & P_2 & (q^3+1)(q+1) & q^2(q^2-1) & \mbox{$q \in \{2,3\}$, $\tau \in G$} \\
{\rm Sp}_n(2) & P_1 & 2^n-1 & 2^{n-1} & n \geqs 6 \\
 & {\rm O}_n^{\e}(2) & 2^{n/2-1}(2^{n/2}+\e) & 2^{n/2-1}(2^{n/2-1}+\e) & n \geqs 6 \\ 
\O_n(3) & P_1 & \frac{1}{2}(3^{n-1}-1) & 3^{(n-3)/2}(3^{(n-1)/2}-1) & r^{+} \in G \\
& N_1^{-} & \frac{1}{2}3^{(n-1)/2}(3^{(n-1)/2}-1) & 3^{n-2}-2.3^{(n-3)/2}-1 & r^{-} \in G \\
{\rm P\O}_{n}^{\e}(q) & P_1 & (2^{n/2}-\e)(2^{n/2-1}+\e) & 2^{n/2-1}(2^{n/2-1}+\e) & G = {\rm O}_{n}^{\e}(2) \\
& & \frac{1}{2}(3^{n/2}+1)(3^{n/2-1}-1) & 3^{n/2-1}(3^{n/2-1}-1) & \mbox{$(q,\e) = (3,-)$, $r \in G$} \\
 & N_1 & 2^{n/2-1}(2^{n/2}-\e) & 2^{n/2-1}(2^{n/2-1}-\e) & G = {\rm O}_{n}^{\e}(2) \\ 
  & & \frac{1}{2}3^{n/2-1}(3^{n/2}-1) & 3^{n/2-1}(3^{n/2-1}-1) & \mbox{$(q,\e)=(3,+)$, $r_{\boxtimes} \in G$} \\
 & & \frac{1}{2}3^{n/2-1}(3^{n/2}+1) & 3^{n-2}-1 & \mbox{$(q,\e)=(3,-)$, $r_{\square} \in G$} \\ \hline
\end{array}
\]
\caption{The subspace actions in part (iv) of Theorem \ref{t:mindeg}}
\label{tab:md}
\end{table}}
\renewcommand{\arraystretch}{1}

\begin{remk}
In Table \ref{tab:md}, we adopt the standard $P_m$ notation for maximal parabolic subgroups (in which case, we can identify $\O$ with the set of totally singular $m$-dimensional subspaces of the natural module for $G_0$). In the second row, $r \in {\rm PGL}_n(3)$ is the image of a reflection $(-I_{n-1},I_1)$ and we note that $r \in G_0$ if and only if $n$ is odd. Similarly, in the third row, $\tau$ is an involutory graph automorphism with $C_{G_0}(\tau) = {\rm PSp}_4(q)$. For $G_0 = \O_n(3)$, we write $N_1^{-}$ for the stabilizer of a nondegenerate $1$-space $U$ of the natural module such that $U^{\perp}$ is a minus-type orthogonal space. We also write $r^{\e}$ with $\e = \pm$ for a reflection $(-I_{n-1},I_1)$ with an $\e$-type $(-1)$-eigenspace (we note that $r^{\e} \in G$ if and only if $G = {\rm SO}_n(3)$ or $n \equiv \e \imod{4}$). Similarly, if $G_0 = {\rm P\O}_n^{\e}(3)$ with $n$ even, then $r \in {\rm PGO}_{n}^{\e}(3)$ is the image of any reflection of the form $(-I_{n-1},I_1)$, while we write $r_{\delta}$ with $\delta \in \{\square, \boxtimes\}$ if we need to specify the discriminant of the $1$-dimensional $1$-eigenspace of $r$ (which is either a square or nonsquare). In addition, $N_1$ denotes the stabilizer of a nonsingular $1$-space (respectively, a nondegenerate $1$-space with square discriminant) when $q=2$ (respectively, $q=3$).
\end{remk}

In order to describe our next application, let $G$ be a finite group and recall that the \emph{commuting probability} of $G$ is the probability that two random elements of $G$ commute. In \cite{BGMN}, Moret\'{o}, Navarro and the authors introduce a natural analogue of this widely studied notion, which is defined in terms of a prime $r$. Let ${\rm Pr}_r(G)$ be the probability that two random $r$-elements in $G$ commute. Then \cite[Theorem A]{BGMN} is the following.

\begin{theorem}\label{t:bgmn}
Let $G$ be a finite group and let $r$ be a prime. Then  
\[
{\rm Pr}_r(G) > \frac{r^2+r-1}{r^3}
\]
if and only if $G$ has a normal and abelian Sylow $r$-subgroup.
\end{theorem}

The proof of Theorem \ref{t:bgmn} relies on bounding the ratio $|C_G(x)_r|/|G_r|$ for every nontrivial $r$-element $x \in G$, where $K_r$ denotes the set of $r$-elements in the subgroup $K$ of $G$. By embedding $C_G(x)$ in a maximal subgroup of $G$, we can bring bounds on fixed point ratios for primitive groups into play and Theorem \ref{t:main} turns out to be a key ingredient in the proof. We refer the reader to \cite{BGMN} for further details.

Our final application concerns the minimal index of a permutation group. Let $G \leqs {\rm Sym}(\O)$ be a primitive permutation group of degree $m$. For $x \in G$ we define   
\[
{\rm ind}(x) = m - {\rm orb}(x) = m\left(1 - \frac{1}{|x|}\sum_{y \in \la x \ra}{\rm fpr}(y)\right)
\]
to be the \emph{index} of $x$, where ${\rm orb}(x)$ is the number of orbits of $x$ on $\O$. Note that ${\rm ind}(x)$ is also the minimal number $t$ such that $x$ is a product of $t$ transpositions in the symmetric group $S_m$. Let us also observe that if $x$ has order $r$, then ${\rm ind}(x) \leqs m(1-1/r)$.

This quantity arises naturally in various number theoretic estimates, including the Riemann-Hurwitz formula for the genus of a branched covering of a smooth projective curve. Similarly, 
\[
{\rm Ind}(G) = \min\{ {\rm ind}(x) \,:\, 1 \ne x \in G\},
\]
which we call the \emph{minimal index} of $G$, also appears in various number theoretic settings (see the work of Malle \cite{Malle1, Malle2}, for example). In particular, it plays a crucial role in a recent beautiful paper of Bhargava \cite{Bha}, where he estimates the number of number fields of given discriminant with a given Galois group. 

In response to a question from Manjul Bhargava, we can use Theorem \ref{t:main} to investigate ${\rm Ind}(G)$ for an arbitrary primitive permutation group $G$. Our main results are Theorems \ref{t:mind} and \ref{t:mind2} below, which will be proved in  Section \ref{s:appl} as an application of Theorem \ref{t:main}.

\begin{theorem}\label{t:mind}
Let $G \leqs {\rm Sym}(\O)$ be a primitive permutation group of degree $m$ with point stabilizer $H$ and assume $|G|$ is even. Let $x \in G$ be an element with ${\rm Ind}(G) = {\rm ind}(x)$. Then the following hold:
\begin{itemize}\addtolength{\itemsep}{0.2\baselineskip}
\item[{\rm (i)}] $|x| \in \{2,3\}$ and there exists an involution $x \in G$ with ${\rm ind}(x) = {\rm Ind}(G)$.
\item[{\rm (ii)}] If $|x|=3$ then one of the following holds (up to permutation isomorphism):

\vspace{1mm}

\begin{itemize}\addtolength{\itemsep}{0.2\baselineskip}
\item[{\rm (a)}] ${\rm Ind}(G) = m/2$ and $|H|$ is odd.
\item[{\rm (b)}] ${\rm Ind}(G) = 4m/9$, $G = V{:}H$ is an affine group with socle $(C_3)^d$, $x \in H \leqs {\rm GL}_d(3)$ is a transvection and $H$ does not contain an involution of the form $(-I_1,I_{d-1})$.
\item[{\rm (c)}] ${\rm Ind}(G) = 4m/9$, $G = L \wr P$ with its product action on $\O = \Gamma^k$, where $k \geqs 1$, $P \leqs S_k$ is transitive, $L = {\rm L}_2(8){:}3$ in its standard action of degree $9$ and $x \in L^k$ is conjugate to $(x_1, 1, \ldots, 1)$ with $x_1$ a field automorphism of ${\rm L}_2(8)$ of order $3$. 
\item[{\rm (d)}] ${\rm Ind}(G) = 2m/n$, $(A_n)^k \normeq G \leqs S_n \wr S_k$, $n \geqs 5$, $k \geqs 1$, $S_n$ has its natural action on $\{1, \ldots, n\}$, the wreath product has the product action of degree $n^k$, $x$ is conjugate to $(x_1, 1, \ldots, 1) \in (A_n)^k$ with $x_1$ a $3$-cycle and $G \cap (S_n)^k$ does not contain elements of the form $(y_1, 1, \ldots, 1)$, $(y_1,y_2,1, \ldots, 1)$ or $(y_1,y_2,y_3, 1, \ldots, 1)$ ($n=5$ only), up to conjugacy, where each $y_i$ is a transposition.
\end{itemize}
\end{itemize}
\end{theorem}

\begin{theorem}\label{t:mind2}
Let $G \leqs {\rm Sym}(\O)$ be a primitive permutation group of degree $m$ with point stabilizer $H$ and assume $|G|$ is even. Then either
\[
\frac{m}{4} \leqs {\rm Ind}(G) \leqs \frac{m}{2},
\]
or one of the following holds (up to permutation isomorphism):
\begin{itemize}\addtolength{\itemsep}{0.2\baselineskip}
\item[{\rm (i)}] $G = L \wr P$ with its product action on $\O = \Gamma^k$, where $k \geqs 1$, $P \leqs S_k$ is transitive, $L \leqs {\rm Sym}(\Gamma)$ is an almost simple primitive classical group in a subspace action with point stabilizer $J$ and 
\[
\frac{3m}{14} \leqs {\rm Ind}(G) = |\Gamma|^{k-1}{\rm Ind}(L,\Gamma)<\frac{m}{4},
\]
where $(L,J,|\Gamma|,{\rm Ind}(L,\Gamma))$ is one of the cases in Table \ref{tab:md2}. 
\item[{\rm (ii)}] $G$ is a subgroup of $S_n \wr S_k$ containing $(A_n)^k$ with $k \geqs 1$, where the action of $S_n$ is on $\ell$-element subsets of $\{1, \ldots, n\}$ and the wreath product has the product action of degree $\binom{n}{\ell}^k$.
\end{itemize}
\end{theorem}

\renewcommand{\arraystretch}{1.2}
{\small \begin{table}
\[
\begin{array}{lllll} \hline
L & J & |\Gamma| & {\rm Ind}(L,\Gamma) & {\rm Conditions} \\ \hline
{\rm U}_4(2).2 & P_2 & 27 & 6 \\
{\rm Sp}_n(2) & {\rm O}_{n}^{-}(2) & 2^{n/2-1}(2^{n/2}-1) & 2^{n/2-2}(2^{n/2-1}-1) & n \geqs 6 \\
{\rm O}_{n}^{-}(2) & P_1 & (2^{n/2-1}-1)(2^{n/2}+1) & 2^{n/2-2}(2^{n/2-1}-1) & n \geqs 8 \\
{\rm O}_{n}^{+}(2) & N_1 & 2^{n/2-1}(2^{n/2}-1) & 2^{n/2-2}(2^{n/2-1}-1) & n \geqs 8 \\ \hline
\end{array}
\]
\caption{The cases arising in part (i) of Theorem \ref{t:mind2}}
\label{tab:md2}
\end{table}}
\renewcommand{\arraystretch}{1}

If $G$ is a primitive group of odd order, then $G$ is solvable by the Feit-Thompson theorem and thus $G$ is an affine group of degree $m=p^d$ for some odd prime $p$. In Theorem \ref{t:odd} we will show that  
\[
\min\left\{ m\left(1-\frac{3}{2r+1}\right), m\left(1-\frac{1}{p}\right)^2 \right\} \leqs {\rm Ind}(G) \leqs m\left(1-\frac{1}{r}\right),
\]
where $r$ is the smallest prime divisor of $|G|$.

\vs

A framework for our proof of Theorem \ref{t:main} is provided by the Aschbacher-O'Nan-Scott theorem, which divides the finite primitive permutation groups into several families (see Table \ref{tab:prim} for a rough description). We proceed by considering each family in turn. As one might expect, most of the work involves the almost simple groups, with a long and delicate analysis required for the subspace actions of classical groups (this is carried out in Section \ref{s:class2}). Our proof for almost simple groups relies heavily on some of the earlier results on fixed point ratios referred to above (in particular, the main theorem of \cite{LS91}, combined with \cite{Bur1, Bur2, Bur3, Bur4, GK} for classical groups and \cite{LLS} for exceptional groups of Lie type). 

\renewcommand{\arraystretch}{1.2}
{\small \begin{table}
\[
\begin{array}{ll} \hline
\mbox{Type} & \mbox{Description} \\ \hline
\mbox{I} & \mbox{Affine: $G = V{:}H \leqs {\rm AGL}(V)$, $H \leqs{\rm GL}(V)$ irreducible} \\
\mbox{II} & \mbox{Almost simple: $T \leqs G \leqs {\rm Aut}(T)$} \\
\mbox{III(a)(i)} & \mbox{Diagonal type: $T^k \leqs G \leqs T^k.({\rm Out}(T) \times P)$, $P \leqs S_k$ primitive} \\
\mbox{III(a)(ii)} & \mbox{Diagonal type: $T^2 \leqs G \leqs T^2.{\rm Out}(T)$} \\
\mbox{III(b)(i)} & \mbox{Product type: $G \leqs L \wr P$, $L$ primitive of type II, $P \leqs S_k$ transitive} \\
\mbox{III(b)(ii)} & \mbox{Product type: $G \leqs L \wr P$, $L$ primitive of type III(a), $P \leqs S_k$ transitive} \\
\mbox{III(c)} & \mbox{Twisted wreath product} \\ \hline
\end{array}
\]
\caption{The finite primitive permutation groups}
\label{tab:prim}
\end{table}}
\renewcommand{\arraystretch}{1}

\section{Affine groups, diagonal groups and twisted wreath products}\label{s:adtw}

In this section we prove Theorem \ref{t:main} when $G$ is either a primitive group of affine type, diagonal type or a twisted wreath product.

\subsection{Affine groups}\label{ss:affine}

\begin{prop}\label{p:affine}
Let $G \leqs {\rm Sym}(\O)$ be a finite primitive permutation group of affine type with socle $(C_p)^d$ and point stabilizer $H \leqs {\rm GL}_d(p)$, where $p$ is a prime. If $x \in G$ has  prime order $r$, then either 
\begin{itemize}\addtolength{\itemsep}{0.2\baselineskip}
\item[{\rm (i)}] ${\rm fpr}(x) \leqs (r+1)^{-1}$; or
\item[{\rm (ii)}] $r=p$, $x$ is conjugate to a transvection in $H$ and ${\rm fpr}(x) = r^{-1}$.  
\end{itemize}
\end{prop}

\begin{proof}
Write $G = V{:}H$, where $V = (\mathbb{F}_p)^d$ and $H \leqs {\rm GL}(V)$ is irreducible. By replacing $x$ by a suitable conjugate, we may assume that $x \in H$ (otherwise ${\rm fpr}(x) = 0$). Set $e = \dim C_V(x)$ and note that ${\rm fpr}(x) = p^{e-d}$.

If $r=p$ then either $e \leqs d-2$ and ${\rm fpr}(x) \leqs p^{-2} < (r+1)^{-1}$, or $e = d-1$, $x$ is a transvection and ${\rm fpr}(x) = r^{-1}$. Now assume $r \ne p$. Here $r$ divides $|{\rm GL}_{d-e}(p)|$, so $r \leqs p^{d-e}-1$ and thus ${\rm fpr}(x) \leqs (r+1)^{-1}$ as required.
\end{proof}

\subsection{Diagonal groups}\label{ss:diag} 

Next we turn to the primitive groups of diagonal type. We will need the following lemma on finite simple groups.

\begin{lem}\label{l:simple}
Let $T$ be a nonabelian finite simple group. Then the following hold:
\begin{itemize}\addtolength{\itemsep}{0.2\baselineskip}
\item[{\rm (i)}] $|{\rm Out}(T)|^3 < |T|$.
\item[{\rm (ii)}] $|\{ t \in T \,:\, t^{\a} = t^{-1}\}| \leqs 4|T|/15$ for all $\a \in {\rm Aut}(T)$.
\item[{\rm (iii)}] $|T| \geqs (|\a|+1)|C_{{\rm Inn}(T)}(\a)|$ for all $\a \in {\rm Aut}(T)$ of prime order.
\end{itemize}
\end{lem}

\begin{proof}
Part (i) is \cite[Lemma 4.8]{faw1} and part (ii) follows from \cite[Theorem 3.1]{Potter}. 

Now consider part (iii) and let $\a \in {\rm Aut}(T)$ be an automorphism of prime order $r$. First assume $\a \in {\rm Inn}(T)$, in which case it suffices to show that $|\a^{{\rm Inn}(T)}| \geqs r+1$. If $|\a^{{\rm Inn}(T)}| \leqs r$ then $|\a^{{\rm Inn}(T)}| \leqs r-1$ since no simple group has a nontrivial conjugacy class of 
prime-power length (this is a classical result of Burnside), whence ${\rm Inn}(T)$ is isomorphic to a subgroup of $S_{r-1}$ and this  contradicts the fact that ${\rm Inn}(T)$ contains an element of order $r$. Now assume $\a \in {\rm Aut}(T) \setminus {\rm Inn}(T)$. Since $T$ is simple, the index of $C_{{\rm Inn}(T)}(\a)$ in ${\rm Inn}(T)$ is at least $5$ and so we may assume $|\a| = r \geqs 5$. This implies that $T$ is a group of Lie type over $\mathbb{F}_q$ and we write ${\rm Inndiag}(T)$ for the subgroup of ${\rm Aut}(T)$ generated by the inner and diagonal automorphisms of $T$. Then either
\begin{itemize}\addtolength{\itemsep}{0.2\baselineskip}
\item[{\rm (a)}] $T = {\rm L}_{n}^{\e}(q)$, $\a \in {\rm Inndiag}(T) \setminus {\rm Inn}(T)$ and $r$ divides $(n,q-\e)$; or
\item[{\rm (b)}] $q=q_0^r$ for some prime power $q_0$ and $\a$ is a field automorphism. 
\end{itemize}

If (a) holds, then 
\[
\frac{|T|}{|C_{{\rm Inn}(T)}(\a)|} \geqs \frac{|{\rm GL}_{n}^{\e}(q)|}{|{\rm GL}_{n-1}^{\e}(q)||{\rm GL}_{1}^{\e}(q)|} = \frac{q^{n-1}(q^n-\e)}{q-\e} \geqs r+1
\]
as required. Similarly, if (b) holds then $|C_{{\rm Inn}(T)}(\a)| \leqs |{\rm Inndiag}(S)|$, where $S$ is a group of the same type as $T$, but defined over the subfield $\mathbb{F}_{q_0}$. Once again, it is easy to verify the desired bound.
\end{proof}

\begin{prop}\label{p:diagonal}
Let $G \leqs {\rm Sym}(\O)$ be a finite primitive permutation group of diagonal type and let $x \in G$ be an element of prime order $r$. Then ${\rm fpr}(x) \leqs (r+1)^{-1}$.
\end{prop}

\begin{proof}
Let $T^k$ be the socle of $G$, where $T$ is a nonabelian simple group and $k \geqs 2$. If $\a \in {\rm Aut}(T)$ then let $\bar{\a}$ denote the coset $\a {\rm Inn}(T)$ in ${\rm Out}(T)$. Adopting Fawcett's notation for diagonal groups presented in \cite{faw1}, we may assume that $G = A(k,T){:}S_k \leqs {\rm Aut}(T) \wr S_k$ and $H = D(k,T)$, where 
\begin{align*}
A(k,T) & = \{(\a_1, \ldots,\a_k) \in {\rm Aut}(T)^k \,:\, \mbox{$\bar{\a}_1 = \bar{\a}_i$ for all $i$} \} \\
D(k,T) & = \{ (\a, \ldots, \a)\pi \,:\, \a \in {\rm Aut}(T),\, \pi \in S_k\}.
\end{align*}
Let $R(G)$ be a set of representatives of the $G$-classes of elements of prime order in $H$. Following \cite[Section 4]{faw1}, we partition $R(G)$ into three collections (here we write $[k]$ for $\{1, \ldots, k\}$):
\begin{align*}
R_1(G) & = \{ (\a, \ldots, \a)\pi \in R(G) \,:\, \mbox{$\pi$ is fixed-point-free on $[k]$} \} \\
R_2(G) & = \{ (\a, \ldots, \a)\pi \in R(G) \,:\, \pi = 1 \} \\
R_3(G) & = \{ (\a, \ldots, \a)\pi \in R(G) \,:\, \mbox{$\pi \ne 1$ and $i\pi = i$ for some $i \in [k]$} \}. 
\end{align*}

First assume $x = (\a, \ldots, \a)\pi \in R_1(G)$ has order $r$. Then $r$ divides $k$ and \cite[Lemma 4.6]{faw1} gives $|C_G(x)| \leqs |C_{S_k}(\pi)||{\rm Out}(T)||T|^{k/r}$. Since $|x^G \cap H| \leqs |{\rm Aut}(T)||\pi^{S_k}|$, it follows that 
\[
{\rm fpr}(x) \leqs \frac{|{\rm Out}(T)|}{|T|^{k-k/r-1}}
\]
and it suffices to show that 
\begin{equation}\label{e:diag}
(r+1)|{\rm Out}(T)| \leqs |T|^{k- k/r-1}.
\end{equation}

Suppose $r<k$, so $k \geqs 2r$. By setting $k=2r$ and $T = A_5$, it is easy to verify the bound in \eqref{e:diag}. Similarly, the desired result follows if $r=k \geqs 3$. 

Finally, suppose $r=k=2$. Here we claim that $|C_{\O}(x)| = |\{ t \in T \,:\, t^{\a} = t^{-1}\}|$.
To see this, we identify $\O$ with the set of cosets $\{D(1,t)\,:\, t \in T\}$ of $D = \{(s,s) \,:\, s \in T\}$ in $T^2$. Then
\[
D(1,t)^{x} = D(1,t^{\a})^{\pi} = D(t^{\a},1)
\]
and thus $D(1,t)^x = D(1,t)$ if and only if $(s,st) = (t^{\a},1)$ for some $s \in T$. The latter equality holds if and only if $s=t^{\a}$ and $t^{\a} = t^{-1}$, which justifies the claim. By applying Lemma \ref{l:simple}(ii) we deduce that $|C_{\O}(x)| \leqs 4|T|/15$ and thus 
\[
{\rm fpr}(x) = \frac{|C_{\O}(x)|}{|T|} \leqs \frac{4}{15} < \frac{1}{3} = \frac{1}{r+1}
\]
as required. 

Next let us assume $x = (\a, \ldots, \a) \in R_2(G)$, so $|\a| = r$. By applying \cite[Lemmas 4.5, 4.6]{faw1} we see that $|C_{\O}(x)| = |C_{{\rm Inn}(T)}(\a)|^{k-1}$ 
and thus Lemma \ref{l:simple}(iii) yields 
\[
{\rm fpr}(x) = \left(\frac{|C_{{\rm Inn}(T)}(\a)|}{|T|}\right)^{k-1} \leqs \left(\frac{1}{r+1}\right)^{k-1}.
\]
The result follows.

Finally let us assume $x = (\a, \ldots, \a)\pi \in R_3(G)$, so $k \geqs 3$ and $\pi$ is a nontrivial permutation of order $r$ with $f$ fixed points on $[k]$, where $1 \leqs f \leqs k-r$. Note that $|\a|=r$. Using \cite[Lemmas 4.5, 4.6]{faw1} we deduce that
\[
{\rm fpr}(x) = \frac{|{\rm Out}(T)||C_{{\rm Inn}(T)}(\a)|^{f-1}}{|T|^{k-1-(k-f)/r}} = \frac{|{\rm Out}(T)|}{|T|^{(k-f)(1-1/r)}} \cdot \left(\frac{|C_{{\rm Inn}(T)}(\a)|}{|T|}\right)^{f-1} \leqs \frac{|{\rm Out}(T)|}{|T|^{r-1}}
\]
and the result follows by applying the bound $|{\rm Out}(T)|<|T|^{1/3}$ in Lemma \ref{l:simple}(i).
\end{proof}

\subsection{Twisted wreath groups}\label{ss:tw}

\begin{prop}\label{p:tw}
Let $G \leqs {\rm Sym}(\O)$ be a finite primitive permutation group of twisted wreath type and let $x \in G$ be an element of prime order $r$. Then ${\rm fpr}(x) \leqs (r+1)^{-1}$.
\end{prop}

\begin{proof}
Write $G = T^k{:}H$, where $T$ is a nonabelian finite simple group and the point stabilizer $H$ is a transitive subgroup of $S_k$. Let $x \in H$ be an element of prime order $r$. Then by applying \cite[Lemmas 5.3, 5.4]{faw2} we deduce that ${\rm fpr}(x) \leqs |T|^{\ell-k}$, where $\ell$ is the number of $r$-cycles in the cycle-shape of $x$ (with respect to the action on $\{1, \ldots, k\}$). The result now follows since $\ell \leqs k/r$. 
\end{proof}

In order to complete the proof of Theorem \ref{t:main}, we may assume $G$ is either an almost simple group or a product type group. The latter groups will be handled in Section \ref{s:product} and we will see that the desired result is easily obtained by combining Proposition \ref{p:diagonal} with our main result for almost simple groups. So the almost simple groups will be our main focus for the remainder of the paper and we divide the analysis into three cases:
\begin{itemize}\addtolength{\itemsep}{0.2\baselineskip}
\item[{\rm (a)}] Non-classical groups (Section \ref{s:as});
\item[{\rm (b)}] Classical groups in non-subspace actions (Section \ref{s:class1}); and
\item[{\rm (c)}] Classical groups in subspace actions (Section \ref{s:class2}).
\end{itemize}

\section{Almost simple groups with non-classical socle}\label{s:as}

In this section we assume $G \leqs {\rm Sym}(\O)$ is a finite primitive almost simple group with socle $G_0$. We postpone the analysis of classical groups to Sections \ref{s:class1} (non-subspace actions) and \ref{s:class2} (subspace actions), so here we take $G_0$ to be a sporadic, alternating or exceptional group of Lie type. In part (ii) of the following result, we adopt the standard labeling of conjugacy classes from the Atlas \cite{ATLAS}.

\begin{prop}\label{p:spor}
Let $G \leqs {\rm Sym}(\O)$ be a finite almost simple primitive permutation group with point stabilizer $H$ and socle $G_0$, a sporadic simple group. Let $x \in G$ be an element of prime order $r$. Then either
\begin{itemize}\addtolength{\itemsep}{0.2\baselineskip}
\item[{\rm (i)}] ${\rm fpr}(x) \leqs (r+1)^{-1}$; or
\item[{\rm (ii)}] $G = {\rm M}_{22}{:}2$, $H = {\rm L}_{3}(4){:}2_2$, $x$ is an involution in the class ${\rm 2B}$ and ${\rm fpr}(x) = 4/11$.
\end{itemize}
\end{prop}

\begin{proof}
For $G \ne \mathbb{B}, \mathbb{M}$ we can verify the bound using \textsf{GAP} \cite{GAP}. Indeed, in each case the character tables of $G$ and $H$ are available in the \textsf{GAP} Character Table Library \cite{GAPCTL} (we use the \texttt{Maxes} function to access the character table of $H$), together with the  fusion map from $H$-classes to $G$-classes. This allows us to compute ${\rm fpr}(x)$ precisely for all $x \in G$ and it is routine to check the desired result.  

A very similar approach also applies when $G = \mathbb{B}$ is the Baby Monster. As before, the character tables of $G$ and $H$ are available in \cite{GAPCTL} and we can also access the stored fusion map if $H \ne (2^2 \times F_4(2)).2$. This quickly reduces the problem to the case $H = (2^2 \times F_4(2)).2$. Here we use the function \texttt{PossibleClassFusions} to determine a set of candidate fusion maps (there are $64$ such maps in total) and for each possibility one checks that ${\rm fpr}(x) \leqs (r+1)^{-1}$ for all $x \in G$ of prime order $r$.

To complete the proof, we may assume $G = \mathbb{M}$ is the Monster group. As discussed in \cite{Wil}, $G$ has $44$ known conjugacy classes of maximal subgroups and any additional maximal subgroup is almost simple with socle one of ${\rm L}_{2}(8)$, ${\rm L}_{2}(13)$, ${\rm L}_{2}(16)$ or ${\rm U}_{3}(4)$. In addition, we note that $r \leqs 71$ and by inspecting the character table of $G$ we can compute 
\[
a_r = \min\{ |x^G| \,:\, x \in G,\, |x|=r\},
\]
which yields the trivial bound ${\rm fpr}(x) \leqs |H|/a_r$. In this way, we immediately deduce that ${\rm fpr}(x) \leqs (r+1)^{-1}$ if $|H|<10^{20}$ and so by inspecting the list of known maximal subgroups of $G$, we have reduced the problem to the cases where $H$ is one of the following:
\[
2.\mathbb{B}, \, 2^{1+24}.{\rm Co}_1, \, 3.{\rm Fi}_{24}, \, 2^2.{}^2E_6(2).S_3, \, 2^{10+16}.\O_{10}^{+}(2), \,  2^{2+11+22}.({\rm M}_{24} \times S_3).
\]
In the first four cases, we can use the function \texttt{NamesOfFusionSources} to access the character table of $H$ in \textsf{GAP} and as above we can check the bound ${\rm fpr}(x) \leqs (r+1)^{-1}$ by working with the stored fusion map from $H$-classes to $G$-classes. 

Finally, suppose $H$ is one of the $2$-local subgroups $2^{10+16}.\O_{10}^{+}(2)$ or 
$2^{2+11+22}.({\rm M}_{24} \times S_3)$. If $r \geqs 3$ then 
\[
\frac{a_r}{r+1} \geqs 53644422509007885434880000000 > |H|
\]
and the result follows. Now assume $r=2$ and note that $G$ has two conjugacy classes of involutions, labeled \texttt{2A} and \texttt{2B}, where
\[
|\texttt{2A}| = 97239461142009186000,\;\; |\texttt{2B}| = 5791748068511982636944259375.
\]
If $x \in \texttt{2B}$ then $|x^G|>3|H|$ and the result follows. On the other hand, if $x \in \texttt{2A}$ then $|x^G \cap H|$ is given in \cite[Proposition 3.9]{BOW}; this allows us to compute ${\rm fpr}(x)$ precisely and it is easy to check that ${\rm fpr}(x) \leqs 1/3$ as required.
\end{proof}

Next we consider the groups with socle an alternating group. We refer the reader to Proposition \ref{p:alt2} for further information on the groups arising in part (i) of the following result. In Table \ref{tab:a6}, we write $(r^h,1^{6-rh})$ to denote any element in $S_6$ that is a product of $h$ disjoint $r$-cycles. We also write ``prim" if the given subgroup $H$ acts primitively on $\{1, \ldots, 6\}$.

\begin{prop}\label{p:alt}
Let $G \leqs {\rm Sym}(\O)$ be a finite almost simple primitive permutation group with point stabilizer $H$ and socle $G_0 = A_n$. Let $x \in G$ be an element of prime order $r$. Then either ${\rm fpr}(x) \leqs (r+1)^{-1}$, or one of the following holds (up to permutation isomorphism):
\begin{itemize}\addtolength{\itemsep}{0.2\baselineskip}
\item[{\rm (i)}] $G = S_n$ or $A_n$ acting on $\ell$-element subsets of $\{1, \ldots, n\}$ with $1 \leqs \ell < n/2$;
\item[{\rm (ii)}] $G = S_n$, $H = S_{n/2} \wr S_2$, $x$ is a transposition and 
\[
{\rm fpr}(x) = \frac{1}{3} + \frac{n-4}{6(n-1)};
\]
\item[{\rm (iii)}] $G = A_8$, $H = {\rm AGL}_{3}(2)$, $x$ is an involution with cycle shape $(2^4)$ and ${\rm fpr}(x) = 7/15$;
\item[{\rm (iv)}] $n=6$ and $(G,H,x,r,{\rm fpr}(x))$ is one of the cases in Table \ref{tab:a6}.
\end{itemize}
\end{prop}

\begin{proof}
Let $x \in G$ be an element of prime order $r$. Note that $r \leqs n$ and recall that we may as well assume $x \in H$. For $n=6$ we can use {\sc Magma} \cite{magma} to check that ${\rm fpr}(x) \leqs (r+1)^{-1}$ unless (i) or (ii) holds, or $(G,H,x)$ is one of the cases recorded in Table \ref{tab:a6}. For the remainder, we may assume that $G = S_n$ or $A_n$, with $n \ne 6$. We partition the analysis into two cases according to the action of $H$ on $[n] = \{1, \ldots, n\}$, noting that (i) holds if $H$ is intransitive.

\begin{table}
\[
\begin{array}{lllll} \hline
G & H & x & r & {\rm fpr}(x) \\ \hline
A_6 & A_5 \mbox{ (prim)} & (3^2) & 3 & 1/2 \\
S_6 & S_5 \mbox{ (prim)} & (2^3) & 2 & 2/3 \\
& & (3^2) & 3 & 1/2 \\
& S_2 \wr S_3 & (2,1^4) & 2 & 7/15 \\
A_6.2^2 & (S_3 \wr S_2).2 & (2,1^4) & 2 & 2/5 \\ \hline
\end{array}
\]
\caption{Some special cases with $G_0 = A_6$}
\label{tab:a6}
\end{table}

\vs

\noindent \emph{Case 1. $H$ is primitive.}

\vs

Suppose $H$ acts primitively on $[n]$. For $n \leqs 12$, a straightforward {\sc Magma} computation shows that ${\rm fpr}(x) \leqs (r+1)^{-1}$ unless $G = A_8$, $H = {\rm AGL}_{3}(2)$ and $x$ has cycle-shape $(2^4)$, in which case $r=2$ and ${\rm fpr}(x) = 7/15$. For the remainder, we will assume $n \geqs 13$. Our aim is to establish the bound ${\rm fpr}(x) \leqs (r+1)^{-1}$.

Let $\mu(H) = \min\{ |{\rm supp}(x)| \,:\, 1 \ne x \in H\}$ be the \emph{minimal degree} of $H$, which is defined to be the minimal number of points moved by a nontrivial element of $H$. Suppose $\mu(H) \geqs n/2$ and observe that this forces
\begin{equation}\label{e:xG}
|x^G| \geqs \frac{n!}{2^{n/4}\lceil n/4 \rceil!\lceil n/2 \rceil!}
\end{equation}
for all $x \in H$ of prime order (minimal if $n \equiv 0 \imod{4}$ and $x$ is an involution with cycle-shape $(2^{n/4},1^{n/2})$). If $n \geqs 25$ then \cite[Corollary 1.2]{Maroti} gives $|H|<2^n$ and thus
\[
{\rm fpr}(x) < \frac{2^{5n/4}\lceil n/4 \rceil!\lceil n/2 \rceil!}{n!}.
\]
One can check that this upper bound is at most $1/(n+1)$ for all $n \geqs 25$, which establishes the desired bound (recall that $r \leqs n$). For $13 \leqs n \leqs 24$ we work with the bound ${\rm fpr}(x) \leqs a/b_r$, where $a$ is the maximal order of a core-free primitive subgroup of $G$ (set $a=0$ if no such subgroup exists) and 
\[
b_r = \min\left\{ \frac{n!}{r^hh!(n-hr)!s} \,:\, \lceil n/2r \rceil \leqs h \leqs \lfloor n/r \rfloor \right\}
\]
is the minimal size of a conjugacy class in $G$ containing elements of order $r$ with at most $n/2$ fixed points on $[n]$ (here $s=2$ if $G = A_n$ and $r=n$, otherwise $s=1$). Using {\sc Magma}, it is easy to calculate $a$ and $b_r$ for every prime $r \leqs n$ and one checks that $a/b_r \leqs (r+1)^{-1}$ in all cases.

To complete the analysis of Case 1, we may assume $n \geqs 13$ and $\mu(H)<n/2$. Here  \cite[Theorem 1]{GM} describes the possibilities for $H$ (note that in our setting, $H$ is a maximal subgroup of $G$, which simplifies the list of cases we need to consider):
\begin{itemize}\addtolength{\itemsep}{0.2\baselineskip}
\item[{\rm (a)}] $H$ is an almost simple classical group with socle $H_0 = \O_{m}^{\e}(2)$, where $m \geqs 6$ and $H$ is acting on a set of hyperplanes of the natural module for $H_0$;
\item[{\rm (b)}] $H = S_{m}$ or $A_{m}$ acting on the set of $\ell$-element subsets of $[m]$, where $m \geqs 5$, $2 \leqs \ell < m/2$ and $n = \binom{m}{\ell}$;
\item[{\rm (c)}] $H = (S_{\ell} \wr S_k) \cap G$ where $\ell \geqs 5$, $k \geqs 2$ and $H$ is acting with its product action on $n = \ell^k$ points.
\end{itemize}

First consider case (a). As explained in \cite{GM}, either 
\begin{itemize}\addtolength{\itemsep}{0.2\baselineskip}
\item[{\rm (a$'$)}] $H_0 = \O_{2\ell+1}(2)$ or $\O_{2\ell}^{+}(2)$, $n = 2^{\ell-1}(2^{\ell}-1)$ and $\mu(H) = n/2 - 2^{\ell-2}$; or 
\item[{\rm (a$''$)}] $H_0 = \O_{2\ell}^{-}(2)$, $n = (2^{\ell}+1)(2^{\ell-1}-1)$ and $\mu(H) = n/2 - (2^{\ell-1}-1)/2$. 
\end{itemize}
In both cases we have $n \geqs 25$ (so $|H|<2^n$ as before) and we can proceed as above, working with a slightly modified version of the lower bound on $|x^G|$ in \eqref{e:xG} to account for the fact that $\mu(H)<n/2$. We omit the details. 

Next let us assume we are in case (b) above. Here $n = \binom{m}{\ell}$ and $\mu(H) \geqs 2\binom{m-2}{\ell-1}$ (see \cite[p.130]{GM}, where it is noted that a transposition has the largest number of fixed points). Suppose $\ell=2$. Here $n = m(m-1)/2$ and $\mu(H) \geqs  2m-4$, which implies that 
\[
|x^G| \geqs \frac{n!}{2^{m-2}(m-2)!(n-2m+4)!}
\]
for all $x \in H$ of prime order (minimal if $x$ is an involution with $n-2m+4$ fixed points). Since $|H| \leqs m!$, it follows that
\[
{\rm fpr}(x) \leqs \frac{2^{m-2}(m-2)!(n-2m+4)!m!}{n!}
\]
and it is straightforward to check that this upper bound is less than $(m+1)^{-1}$ for all $m \geqs 5$. The result now follows since $r \leqs m$. The reader can check that a very similar argument applies if $\ell>2$.

Finally, let us turn to case (c). Here $n = \ell^k$ and $\mu(H) = 2\ell^{k-1}$ (see \cite[p.130]{GM}). Suppose $k=2$, so $n = \ell^2$, $|H| \leqs 2(\ell!)^2$ and $r \leqs \ell$. Given the lower bound on $\mu(H)$, it follows that 
\[
|x^G| \geqs \frac{n!}{2^\ell\ell!(n-2\ell)!}
\]
and thus
\[
{\rm fpr}(x) \leqs \frac{2^{\ell+1}(\ell!)^3(n-2\ell)!}{n!},
\]
which is easy to check is at most $(\ell+1)^{-1}$ for all $\ell \geqs 5$. This gives the desired result for $k=2$ and the cases with $k>2$ can be handled in the same way. 

\vs

\noindent \emph{Case 2. $H$ is transitive and imprimitive.}

\vs

To complete the proof of the proposition, we may assume $H$ acts transitively and imprimitively on $[n]$. The groups with $n \leqs 13$ can be checked using {\sc Magma}, noting that the only exceptions (excluding $n=6$) are the cases where $G = S_n$, $H = S_{n/2} \wr S_2$ and $x$ is a transposition. Here $|x^G \cap H| = 2\binom{n/2}{2}$, $|x^G| = \binom{n}{2}$ and thus
\[
{\rm fpr}(x) = \frac{2\binom{n/2}{2}}{\binom{n}{2}} = \frac{1}{3} + \frac{n-4}{6(n-1)}.
\]
In particular, this gives an infinite family of exceptions to the bound ${\rm fpr}(x) \leqs (r+1)^{-1}$ and this special case is recorded in part (i)(b) of Theorem \ref{t:main}.

Now assume $n \geqs 14$. Fix a divisor $\ell$ of $n$ with $1 < \ell < n$ and identify $\O$ with $\Pi_{\ell}$, the set of partitions of $[n]$ into $\ell$ parts of equal size. Note that $H = (S_{n/\ell} \wr S_{\ell}) \cap G$. If $3 \leqs r \leqs \ell$ then \cite[Lemma 4.5]{BH} implies that ${\rm fpr}(x) < \ell^{-2}$ and the result follows. Similarly, if $r=2$ then the bound in \cite[Lemma 4.6]{BH} is sufficient unless $x$ is a transposition and $\ell=2$, which is a genuine exception, as noted above.

So to complete the analysis of this case, we may assume $r > \ell$. Note that $3 \leqs r \leqs n/\ell$. Let us also note that $x$ fixes a partition in $\Pi_{\ell}$ if and only if it fixes each part of the partition setwise. It follows that ${\rm fpr}(x) \leqs {\rm fpr}(y)$, where $y \in G$ is an $r$-cycle, and so it suffices to show that ${\rm fpr}(y) \leqs (r+1)^{-1}$. Clearly, we have $|y^G \cap H| = (r-1)!\binom{n/\ell}{r}\ell$ and $|y^G| = (r-1)!\binom{n}{r}$, whence
\[
{\rm fpr}(y) = \frac{(n/\ell)!(n-r)!\ell}{n!(n/\ell-r)!} \leqs (r+1)^{-1}
\]
if and only if $f(r) \geqs 1$, where
\[
f(r):= \frac{n!(n/\ell-r)!}{(n/\ell)!(n-r)!\ell(r+1)}.
\]
It is routine to check that $f$ is increasing as a function of $r$ and the result follows since  
\[
f(3) = \frac{(n-1)(n-2)}{4(n/\ell-1)(n/\ell-2)} \geqs \frac{(n-1)(n-2)}{4(n/2-1)(n/2-2)} > 1.
\]

In conclusion, if $n \ne 6$ and $H$ is transitive and imprimitive, then ${\rm fpr}(x) \leqs (r+1)^{-1}$ unless $G = S_n$, $H = S_{n/2} \wr S_2$ and $x$ is a transposition. This completes the proof of the proposition.
\end{proof}

\begin{rem}\label{r:alt}
Let us consider the special cases arising in parts (iii) and (iv) of Proposition \ref{p:alt}. First observe that the example in (iii) does not appear in the statement of Theorem \ref{t:main} since $G$ is permutation isomorphic to the classical group ${\rm L}_{4}(2)$ acting on the set of $1$-dimensional subspaces of its natural module. Similarly, if $G = S_6$ or $A_6$ and $H = S_5$ or $A_5$ is primitive, then there is a permutation isomorphism to the natural action of $G$ on $\{1, \ldots, 6\}$, which is included in part (i). If $G = S_6$ and $H = S_2 \wr S_3$ then $G$ is permutation isomorphic to ${\rm Sp}_{4}(2)$ acting on the set of $1$-dimensional subspaces of the natural module. And if $G = A_6.2^2$ and $H = (S_3 \wr S_2).2$, then $G$ is permutation isomorphic to the natural action of ${\rm P\Gamma L}_{2}(9)$ on $1$-spaces.
\end{rem}

The following result provides more information on the groups appearing in part (i) of Proposition \ref{p:alt}.

\begin{prop}\label{p:alt2}
Let $G \leqs {\rm Sym}(\O)$ be a finite almost simple primitive permutation group with socle $G_0 = A_n$, where $\O$ is the set of $\ell$-element subsets of $\{1,\ldots, n\}$ with $1 \leqs \ell < n/2$. Assume $G = S_n$ or $A_n$ and let $x \in G$ be an element of prime order $r$. 
\begin{itemize}\addtolength{\itemsep}{0.2\baselineskip}
\item[{\rm (i)}] We have ${\rm fpr}(x) \leqs {\rm fpr}(y)$, where $y$ is an $r$-cycle.
\item[{\rm (ii)}] If $r>n-\ell$ then ${\rm fpr}(y) = 0$, otherwise
\[
{\rm fpr}(y) = \prod_{i=0}^{r-1} \left(1 - \frac{\ell}{n-i}\right) + \a \prod_{i=0}^{r-1} \left(1 - \frac{n-\ell}{n-i}\right),
\]
where $\a=1$ if $r \leqs \ell$, otherwise $\a=0$.
\end{itemize}
\end{prop}

\begin{proof}
Let $\Lambda \in \O$ be an $\ell$-set and observe that $x$ fixes $\Lambda$ if and only if the support of each $r$-cycle comprising $x$ is contained in $\Lambda$ or $\{1, \ldots, n\} \setminus \Lambda$. In particular, the number of $\ell$-sets fixed by $x$ is at most the number fixed by a single $r$-cycle, which gives the bound in (i). 

For the expression in (ii), we may as well assume $G = S_n$, in which case $H = S_{\ell} \times S_{n-\ell}$. Clearly, if $r>n-\ell$ then $y^G \cap H$ is empty and thus ${\rm fpr}(y) = 0$. Now assume $r \leqs n-\ell$. Then
$|y^G \cap H| = a_{\ell}+a_{n-\ell}$ and $|y^G| = a_n$, where $a_{m}$ is the number of $r$-cycles in $S_{m}$. The result now follows since 
$a_{m} = m!/(m-r)!r$ if $r \leqs m$, otherwise $a_{m} = 0$.
\end{proof}

\begin{rem}\label{r:alt2}
Consider the expression for ${\rm fpr}(y)$ in part (ii) of Proposition \ref{p:alt2} and assume $r \leqs n-\ell$. If we fix $n$ and $r$, then it is straightforward to check that ${\rm fpr}(y)$ is decreasing as a function of $\ell$, which implies that ${\rm fpr}(y) \leqs 1- r/n$, with equality if and only if $\ell=1$. Similarly, for $r=2$ we deduce that ${\rm fpr}(y) \geqs 1/2-1/2n$, with equality if and only if $n$ is odd and $\ell=(n-1)/2$. In particular, if $r=2$ then ${\rm fpr}(y) > 1/3$ for all $n$ and $\ell$. Notice that if $r$ and $\ell$ are fixed, then ${\rm fpr}(y)$ tends to $1$ as $n$ tends to infinity.
\end{rem}

For the remainder of this section, we will assume $G_0$ is a simple exceptional group of Lie type (recall that the classical groups will be handled in the next two sections). In Proposition \ref{p:ex} below we assume $G_0 \ne G_2(2)', {}^2G_2(3)'$ since these groups are isomorphic to ${\rm U}_{3}(3)$ and ${\rm L}_{2}(8)$, respectively. 

\begin{rem}
For the record, in the two excluded cases we get ${\rm fpr}(x) \leqs (r+1)^{-1}$ unless $G_0 = {}^2G_2(3)'$, $H \cap G_0 = 2^3{:}7$ and either $r=7$ (${\rm fpr}(x) = 2/9$) or $x \in G \setminus G_0$ has order $3$ (${\rm fpr}(x) = 1/3$). Note that the special cases arising here correspond (up to permutation isomorphism) to the natural actions of ${\rm L}_{2}(8)$ or ${\rm P\Gamma L}_{2}(8)$ on the set of $1$-dimensional subspaces of the natural module for ${\rm L}_{2}(8)$. In particular, they are included in part (i)(d) of Theorem \ref{t:main}.
\end{rem}

\begin{prop}\label{p:ex}
Let $G \leqs {\rm Sym}(\O)$ be a finite almost simple primitive permutation group with point stabilizer $H$ and socle $G_0$, an exceptional group of Lie type over $\mathbb{F}_q$. Then ${\rm fpr}(x) \leqs (r+1)^{-1}$ for every element $x \in G$ of prime order $r$.
\end{prop}

\begin{proof}
Write $q=p^f$ with $p$ a prime and let $x \in G$ be an element of prime order $r$. Recall that the possibilities for $x$ are as follows, where ${\rm Inndiag}(G_0)$ denotes the subgroup of ${\rm Aut}(G_0)$ generated by the inner and diagonal automorphisms of $G_0$:
\begin{itemize}\addtolength{\itemsep}{0.2\baselineskip}
\item[{\rm (a)}] $x \in {\rm Inndiag}(G_0)$ is either semisimple ($r \ne p$) or unipotent ($r=p$);
\item[{\rm (b)}] $x$ is a graph automorphism ($r \leqs 3$ only);
\item[{\rm (c)}] $x$ is a field automorphism ($q=q_0^r$ only);
\item[{\rm (d)}] $x$ is a graph-field automorphism ($r \leqs 3$ and $q=q_0^r$ only).
\end{itemize}

In \cite{LLS}, Lawther, Liebeck and Seitz conduct an extensive study of  fixed point ratios for exceptional groups of Lie type. In particular, \cite[Theorem 1]{LLS} gives an explicit upper bound on ${\rm fpr}(x)$, which is presented as a function of $q$ and is valid for all nontrivial $x \in G$. It is easy to check that this bound  immediately reduces the problem to case (a) above with $r \ne p$ and $r \geqs 5$. For example, if $G_0 = G_2(q)$ and $q \ne 2,4$ then  \cite[Theorem 1]{LLS} gives ${\rm fpr}(x) \leqs (q^2-q+1)^{-1}$. This is less than $(q+1)^{-1}$, which in turn is at most $(r+1)^{-1}$ in cases (b), (c) and (d), as well as case (a) with $r=p$. 

For the remainder, let $x \in G_0$ be a semisimple element of prime order $r \geqs 5$ and note that $x$ is contained in a maximal torus of $G_0$. Here we appeal to the more refined bounds presented in \cite[Theorem 2]{LLS}, which make a distinction between three different possibilities for the maximal subgroup $H$:
\begin{itemize}\addtolength{\itemsep}{0.2\baselineskip}
\item[{\rm (I)}] $H$ is a maximal parabolic subgroup; 
\item[{\rm (II)}] $H$ is a non-parabolic maximal rank subgroup in the sense of \cite{LSS};
\item[{\rm (III)}] The remaining subgroups.
\end{itemize}

We proceed by considering the various possibilities for $G_0$ in turn. We first determine an upper bound on $r$ and we then apply the bound on ${\rm fpr}(x)$ in \cite[Theorem 2]{LLS}, considering cases (I), (II) and (III) separately if needed. This approach is effective, with the exception of a handful of cases where we need to produce a slightly sharper fixed point ratio estimate. For some low rank groups defined over small fields, we can also use {\sc Magma} \cite{magma} to verify the result. Set $H_0 = H \cap G_0$.

\vs

\noindent \emph{Case 1. $G_0 = E_8(q)$ or $E_7(q)$.}

\vs

First assume $G_0 = E_8(q)$. By expressing $|G_0|_{p'}$ as a product of cyclotomic polynomials, we note that $r \leqs \Phi_{30}(q) = q^8+q^7-q^5-q^4-q^3+q+1$ and the result follows since \cite[Theorem 1]{LLS} gives ${\rm fpr}(x) \leqs q^{-8}(q^4-1)^{-1}$. Similarly, if $G_0 = E_7(q)$ then $r \leqs \Phi_{7}(q) = (q^7-1)/(q-1)$ and \cite[Theorem 2]{LLS} implies that ${\rm fpr}(x) \leqs 2q^{-7}(q^4-1)^{-1}$, which is sufficient.

\vs

\noindent \emph{Case 2. $G_0 = E_6^{\e}(q)$.}

\vs

Next assume $G_0 = E_6^{\e}(q)$. Once again, by considering the order of $|G_0|$ we deduce that $r \leqs q^6 +\e q^3 +1$. We proceed by considering cases (I)-(III) in turn, working with the appropriate upper bound on ${\rm fpr}(x)$ in \cite[Theorem 2]{LLS}. 

First suppose $H$ is a parabolic subgroup, which is labeled in the usual way. If $\e=-$, or if $\e=+$ and $H$ is not of type $P_1$, $P_6$ or $P_{1,6}$, then 
\[
{\rm fpr}(x) \leqs \frac{1}{q^3(q^3-2)(q^2-1)}
\]
and the result follows. Now assume $\e=+$ and $H$ is of type $P_1$, $P_6$ or $P_{1,6}$. Here 
\[
{\rm fpr}(x) \leqs \frac{1}{q(q^3-1)(q^2-1)},
\]
which is sufficient unless $r+1 > q(q^3-1)(q^2-1)$. Since $r$ divides $|G_0|$, it is easy to check that the latter inequality holds if and only if $r = q^6+q^3+1$, in which case $x$ is regular and we have 
\[
|x^G \cap H| \leqs |H_0| < q^{62},\;\; |x^G| \geqs \frac{|{\rm Inndiag}(G_0)|}{q^6+q^3+1} >\frac{1}{2}q^{72}.
\]
This gives ${\rm fpr}(x) < 2q^{-10} < (r+1)^{-1}$ as required. 

Next assume $H$ is a maximal rank subgroup as in (II). If $q \geqs 3$ then \cite[Theorem 2]{LLS} gives ${\rm fpr}(x) \leqs 2q^{-12}$ and the result follows. On the other hand, if $q=2$ then ${\rm fpr}(x) \leqs 2^{-5}$ and the problem is reduced to the case where $\e=+$ and $r=73$ (note that the character table of $G_0$ is available in \cite{GAPCTL}). Here the maximal subgroups of $G$ are determined (up to conjugacy) in \cite{KW} and by inspection we deduce that $H_0 = {\rm L}_{3}(8){:}3$ is the only possibility with ${\rm fpr}(x) > 0$. In this case, $|x^G| = |G_0|/73$ and the trivial bound
$|x^G \cap H| \leqs |H_0|$ is sufficient. Finally, if $H$ is of type (III) then \cite[Theorem 2]{LLS} gives ${\rm fpr}(x) \leqs q^{-6}(q^6-q^3+1)^{-1}$ and the result follows.

\vs

\noindent \emph{Case 3. $G_0 = F_4(q)$ or $G_2(q)$.}

\vs

First assume $G_0 = F_4(q)$. Here $r \leqs q^4+1$ and we proceed as above, working with the upper bounds on ${\rm fpr}(x)$ in \cite[Theorem 2]{LLS}. If $H$ is a parabolic subgroup, then ${\rm fpr}(x) \leqs q^{-2}(q^3-2)^{-1}$ for $H \ne P_1$, which is sufficient. On the other hand, for $H = P_1$ we have ${\rm fpr}(x) \leqs (q^4-q^2+1)^{-1}$ and so we may assume $r \geqs q^4-q^2+1$. By considering the prime divisors of $|G_0|$, we deduce that $r = q^4-q^2+1$ or $q^4+1$ are the only options. Then
\[
|x^G \cap H| \leqs |H_0| < q^{37},\;\; |x^G| \geqs \frac{|G_0|}{q^4+1}> \frac{1}{2}q^{48}
\]
and thus ${\rm fpr}(x) < 2q^{-11}  < (r+1)^{-1}$. For cases (II) and (III), it is easy to check that the bounds in \cite[Theorem 2]{LLS} are sufficient.

Next suppose $G_0 = G_2(q)'$ with $q \geqs 2$. The groups with $q \leqs 5$ can be handled using {\sc Magma}, working with the functions \texttt{AutomorphismGroupSimpleGroup} and \texttt{MaximalSubgroups} to construct $G$ and $H$, and the \texttt{ConjugacyClasses} function to compute $|x^G \cap H|$ and $|x^G|$, which yields ${\rm fpr}(x)$. Now assume $q \geqs 7$ and note that $r \leqs q^2+q+1$. 

To begin with, let us assume $H$ is a maximal parabolic subgroup of $G$. If $H = P_{1}$ (or $P_{1,2}$) then \cite[Theorem 2]{LLS} gives ${\rm fpr}(x) \leqs 2(q^3+1)^{-1}$ and the result follows. Otherwise, if $H = P_2$ then ${\rm fpr}(x) \leqs (q^2-q+1)^{-1}$ and so we may assume $r>q^2-q$, which forces $r = q^2 \pm q +1$. Then
\[
|x^G \cap H| \leqs |H_0| = q^5(q-1)|{\rm SL}_{2}(q)|,\;\;  |x^G| \geqs \frac{|G_2(q)|}{q^2+q+1},
\]
which gives 
\[
{\rm fpr}(x) \leqs \frac{q-1}{q(q+1)(q^2-q+1)} \leqs \frac{1}{r+1}
\]
as required.

Next suppose $H$ is a maximal rank subgroup of type (II). Here \cite[Theorem 2]{LLS} gives ${\rm fpr}(x) \leqs (q^2-q+1)^{-1}$ and so as above we reduce to the case where $r = q^2 \pm q +1$. By inspecting the maximal subgroups of $G$ (for example, see \cite[Tables 8.30, 8.41, 8.42]{BHR}), we deduce that $|H_0| \leqs 2|{\rm SU}_{3}(q)|$ and by arguing as above we obtain ${\rm fpr}(x) \leqs 2q^{-3}(q-1)^{-1}$, which is sufficient. Similarly, if $H$ is of type (III) then by \cite[Theorem 2]{LLS} we have ${\rm fpr}(x) \leqs q^{-2}$ and so we may assume $r = q^2+q+1$. Since $x$ generates a maximal torus of $G_0$, it follows that every maximal subgroup of $G$ containing $x$ is of type (I) or (II), so this situation does not arise.

\vs

\noindent \emph{Case 4. $G_0 = {}^3D_4(q)$.}

\vs

Now assume $G_0 = {}^3D_4(q)$. The case $q=2$ can be checked using {\sc Magma}, so we will assume $q \geqs 3$. By expressing $|G_0|_{p'}$ as a product of cyclotomic polynomials, we deduce that $r \leqs q^4-q^2+1$. If $H$ is a maximal parabolic subgroup, then \cite[Theorem 2]{LLS} gives ${\rm fpr}(x) \leqs q^{-2}(q^3-2)^{-1}$ and this bound is sufficient. Now suppose $H$ is of type (II). If $q=3$ then ${\rm fpr}(x) \leqs 3^{-4}$ and the result follows. Now assume $q \geqs 4$. Here ${\rm fpr}(x) \leqs (q^4-q^2+1)^{-1}$ and so we may assume $r = q^4-q^2+1$. By inspecting \cite[Table 8.51]{BHR} we deduce that $H = N_G(\la x \ra)$ is the only option, whence
\[
|x^G \cap H| \leqs |H_0| = 4(q^4-q^2+1),\;\; |x^G| \geqs \frac{|{}^3D_4(q)|}{q^4-q^2+1}
\]
and the desired result follows. Finally, suppose $H$ is of type (III). Here ${\rm fpr}(x) \leqs (q^4-q^2+1)^{-1}$ and so as above we may assume $r=q^4-q^2+1$. But then $x$ generates a maximal torus, so it is not contained in a subgroup of type (III). This completes the argument for the groups with socle $G_0 = {}^3D_4(q)$.

\vs

\noindent \emph{Case 5. $G_0 = {}^2F_4(q)'$, ${}^2G_2(q)$ or ${}^2B_2(q)$.}

\vs

First assume $G_0 = {}^2F_4(q)'$, so $q = 2^{2m+1}$ with $m \geqs 0$. The case $q=2$ can be handled using {\sc Magma}, so we may assume $q \geqs 8$. By considering $|G_0|$, we observe that
\[
r \leqs q^2+\sqrt{2q^3}+q+\sqrt{2q}+1.
\]
It is now straightforward to check that the bounds on ${\rm fpr}(x)$ presented in \cite[Theorem 2]{LLS} are sufficient in every case.

Next suppose $G_0 = {}^2G_2(q)$, where $q = 3^{2m+1}$ and $m \geqs 1$ (recall that we exclude the case $G_0 = {}^2G_2(3)'$ since this group is isomorphic to ${\rm L}_{2}(8)$). Here $r \leqs q+\sqrt{3q}+1$ and once again the bounds in \cite[Theorem 2]{LLS} are good enough. 

Finally, suppose $G_0 = {}^2B_2(q)$ with $q=2^{2m+1}$ and $m \geqs 1$. The cases $q \in \{8,32\}$ can be verified using {\sc Magma}, so let us assume $q \geqs 2^7$. If $H$ is a Borel subgroup (that is, a subgroup of type (I)) then \cite[Theorem 2]{LLS} gives ${\rm fpr}(x) \leqs (q^2+1)^{-1}$ and the result follows since $r \leqs q+\sqrt{2q}+1$. Similarly, if $H$ is of type (II) or (III) then ${\rm fpr}(x) \leqs (q^{2/3}+1)/(q^2+1)$ and once again the result follows.
\end{proof}

\section{Almost simple classical groups: Non-subspace actions}\label{s:class1}

In this section we establish Theorem \ref{t:main} in the case where $G$ is an almost simple classical group in a so-called \emph{non-subspace} action (see Definition \ref{d:sub} below). So let $G \leqs {\rm Sym}(\O)$ be a finite primitive almost simple group with socle $G_0$, which is a classical group over $\mathbb{F}_q$ with natural module $V$. Write $q = p^f$ with $p$ a prime and set $n = \dim V$. 

Let $H$ be a point stabilizer in $G$ and note that $H$ is a maximal subgroup and $G = HG_0$ since $H$ is core-free. The main theorem on the subgroup structure of finite almost simple classical groups is due to Aschbacher \cite{asch}. Nine collections of subgroups of $G$ are defined, typically labelled $\C_1, \ldots, \C_8$ and $\mathcal{S}$, and Aschbacher proves that every maximal subgroup of $G$ is contained in one of these collections (some adjustments are needed when $G_0 = {\rm PSp}_{4}(q)$ (with $q$ even) or ${\rm P\O}_{8}^{+}(q)$, noting that a complete result in the latter case was established in later work by Kleidman \cite{K}). Here the so-called \emph{geometric} subgroups that comprise the $\C_i$ collections are defined in terms of the geometry of the underlying vector space $V$. For example, they include the stabilizers of subspaces and appropriate direct sum and tensor product decompositions of $V$ (see \cite[Table 1.2.A]{KL} for a brief description of each geometric collection). The \emph{non-geometric} subgroups in $\mathcal{S}$ are almost simple and irreducibly embedded in $G$. 

In \cite{KL}, Kleidman and Liebeck present a great deal of information on the maximal subgroups of classical groups, including a complete description of the structure and conjugacy of the maximal geometric subgroups when $n \geqs 13$. This is complemented by work of Bray, Holt and Roney-Dougal \cite{BHR}, which gives complete information on the low-dimensional groups with $n \leqs 12$. It is important to note that we adopt the precise definition of the $\C_i$ collection given in \cite{KL}, which differs slightly from Aschbacher's original description in \cite{asch}.

There is also an extensive literature on conjugacy classes in classical groups; in this regard, \cite[Chapter 3]{BG} will be a convenient reference for our purposes.

\begin{defn}\label{d:sub}
Let $G \leqs {\rm Sym}(\O)$ be a finite primitive almost simple classical group over $\mathbb{F}_q$ with socle $G_0$ and point stabilizer $H$. We say that the action of $G$ on $\O$ is a \emph{subspace action} if one of the following holds:
\begin{itemize}\addtolength{\itemsep}{0.2\baselineskip}
\item[{\rm (i)}] $H$ is in the collection $\C_1$; or
\item[{\rm (ii)}] $G_0 = {\rm Sp}_{n}(q)$, $q$ is even and $H \cap G_0 = {\rm O}_{n}^{\pm}(q)$.
\end{itemize}
In this situation, we will sometimes refer to the subgroup $H$ as a \emph{subspace subgroup} of $G$. Non-subspace actions and subgroups are defined accordingly.
\end{defn}

Note that the subspace actions appearing in part (ii) correspond to certain subgroups in the $\C_8$ collection. If we view $G_0$ as an orthogonal group ${\rm O}_{n+1}(q)$ with natural module $W$, then we may identify $\O$ with a set of nondegenerate hyperplanes in $W$, which explains why it is reasonable to view the action of $G$ on $\O$ in this situation as a subspace action. 

In this section, our aim is to prove Theorem \ref{t:main} when $G$ is a classical group in a non-subspace action, postponing the analysis of subspace actions to Section \ref{s:class2}. Our main result is the following.

\begin{thm}\label{t:nonsub}
Let $G \leqs {\rm Sym}(\O)$ be an almost simple finite primitive permutation group with socle $G_0$ and point stabilizer $H$. Let $x \in G$ be an element of prime order $r$ and assume $G_0$ is a classical group and $H$ is a non-subspace subgroup. Then either
\[
{\rm fpr}(x) \leqs \frac{1}{r+1}
\]
or $G$ is permutation isomorphic to one of the groups recorded in part (i) of Theorem \ref{t:main}.
\end{thm}

\begin{rem}
There are only a handful of exceptions to the main bound in Theorem \ref{t:nonsub}. For example, if $G = {\rm L}_2(7)$ and $H = S_4$ is a $\C_6$-subgroup of type $2^{1+2}_{-}.{\rm O}_{2}^{-}(2)$, then ${\rm fpr}(x) = 3/7$ if $x$ is an involution. But here $G$ is permutation isomorphic to ${\rm L}_3(2)$ acting on $1$-dimensional subspaces of the natural module, so this case is included in part (i)(d) of Theorem \ref{t:main}. Similarly, if $G = \O_8^{+}(2)$ and $H = {\rm Sp}_6(2)$ is an irreducibly embedded subgroup in $\mathcal{S}$, then ${\rm fpr}(x) = 3/10$ when $x = (\L^4)$ is an element of order $3$ with a trivial $1$-eigenspace. Here the action is permutation isomorphic to the action on nonsingular $1$-spaces, so once again this corresponds to a case in (i)(d) of Theorem \ref{t:main}. We also refer the reader to the statement of Lemma \ref{l:psl4} for further examples with $G_0 = {\rm L}_4^{\e}(q)$.
\end{rem}

Our main reason for making the distinction between subspace and non-subspace actions is encapsulated in the following result, which only applies in the non-subspace setting. This is \cite[Theorem 1]{Bur1}, which is proved in the sequence of papers \cite{Bur2, Bur3, Bur4}. It will be our main tool in this section.

\begin{thm}\label{t:fpr}
Let $G \leqs {\rm Sym}(\O)$ be a finite almost simple classical primitive permutation group with point stabilizer $H$ and socle $G_0$. Assume $H$ is a non-subspace subgroup. Then
\[
{\rm fpr}(x) < |x^G|^{-\frac{1}{2}+\frac{1}{n}+\iota}
\]
for all $x \in G$ of prime order, where $n$ is the dimension of the natural module for $G_0$ and either $\iota=0$ or $\iota$ is listed in \cite[Table 1]{Bur1}.
\end{thm}

For non-subspace actions of classical groups, the bound in Theorem \ref{t:fpr} can be viewed as a significant strengthening of the following more general result, which is due to Liebeck and Saxl (see \cite[Theorem 1]{LS91}). It is worth noting that almost all of the special cases appearing in \cite[Table 1]{LS91} involve groups with socle ${\rm L}_{2}(q)$.

\begin{thm}\label{t:ls91}
Let $G \leqs {\rm Sym}(\O)$ be a finite almost simple primitive permutation group with point stabilizer $H$ and socle $G_0$, which is a simple group of Lie type over $\mathbb{F}_q$. Then either 
\[
{\rm fpr}(x) \leqs \frac{4}{3q}
\]
for all nontrivial $x \in G$, or $(G,H,x)$ is one of the cases appearing in \cite[Table 1]{LS91}. 
\end{thm}

Finally, we need one more definition before we are ready to begin the proof of Theorem \ref{t:nonsub}. Recall that $V$ is the natural module for $G_0$, which is an $n$-dimensional vector space over $\mathbb{F}_{q^u}$, where $u=2$ if $G_0 = {\rm U}_{n}(q)$ and $u=1$ in all other cases. 

\begin{defn}\label{d:nu}
Let $K$ be the algebraic closure of $\mathbb{F}_{q^u}$. For any element $x \in G \cap {\rm PGL}(V)$, write $x = \hat{x}Z$ with $\hat{x} \in {\rm GL}_{n}(q^u)$ and $Z = Z({\rm GL}_{n}(q^u))$, and let $\nu(x)$ be the codimension of the largest eigenspace of $\hat{x}$ as an element of ${\rm GL}_{n}(K)$. Note that this is independent of the choice of $\hat{x}$. 
\end{defn}

Let $x \in G$ be an element of prime order $r$. If $x \in G \cap {\rm PGL}(V)$, then $x$ is either unipotent (if $r=p$) or semisimple (if $r \ne p$), and we refer the reader to \cite[Section 3]{Bur2} for bounds on $|x^G|$ in terms of $n$, $q$ and $\nu(x)$. In addition, we will use the notation from \cite[Chapter 3]{BG} to describe representatives of the conjugacy classes of unipotent and semisimple elements of prime order. In particular, if $G_0$ is a symplectic or orthogonal group in even characteristic then we adopt the notation from \cite{AS} for the conjugacy classes of unipotent involutions. 

Now suppose $x \not\in {\rm PGL}(V)$. Here $x$ is either a field, graph or graph-field automorphism of $G_0$ and once again we refer the reader to \cite[Section 3]{Bur2} for bounds on $|x^G|$. Let us also observe that if $x$ is a field or graph-field automorphism, then $q=q_0^r$ and unless we are in one of the handful of special cases recorded in \cite[Table 1]{LS91}, we deduce that
\begin{equation}\label{e:ls91}
{\rm fpr}(x) \leqs \frac{4}{3q} \leqs \frac{4}{3\cdot 2^r} \leqs \frac{1}{r+1}
\end{equation}
via Theorem \ref{t:ls91}.

We are now ready to begin the proof of Theorem \ref{t:nonsub}. First we handle the linear groups with socle $G_0 = {\rm L}_{n}^{\e}(q)$ and $n \geqs 5$. For small values of $n$, the bound in Theorem \ref{t:fpr} is less effective and we will consider the groups with $n \in \{2,3,4\}$ in Lemmas \ref{l:psl2}, \ref{l:psl3} and \ref{l:psl4} below.

\begin{lem}\label{l:linuni}
Suppose $G_0 = {\rm L}_{n}^{\e}(q)$ with $n \geqs 5$. If $H$ is a non-subspace subgroup, then ${\rm fpr}(x) \leqs (r+1)^{-1}$ for all $x \in G$ of prime order $r$.
\end{lem}

\begin{proof}
Let $x \in G$ be an element of prime order $r$ and note that $\iota \leqs 1/n$ in Theorem \ref{t:fpr} (see \cite[Table 1]{Bur1}), whence
\begin{equation}\label{e:psl}
{\rm fpr}(x) < |x^G|^{-\frac{1}{2}+\frac{2}{n}}.
\end{equation}
First assume $x \not\in {\rm PGL}_{n}^{\e}(q)$, so  either $r=2$ or $q = q_0^r$. By \cite[Corollary 3.49]{Bur2} we have
\[
|x^G|>\frac{1}{2}\left(\frac{q}{q+1}\right)q^{\frac{1}{2}(n^2-n-4)}
\]
and by combining this bound with \eqref{e:psl} we obtain ${\rm fpr}(x) \leqs (q+1)^{-1}$ and the result follows.

For the remainder, we may assume $x \in {\rm PGL}_{n}^{\e}(q)$. First observe that 
\[
|x^G|>\frac{1}{2}\left(\frac{q}{q+1}\right)q^{2n-2}
\]
by \cite[Corollary 3.38]{Bur2} and by combining this bound with \eqref{e:psl} we deduce that ${\rm fpr}(x) \leqs (q+2)^{-1}$ unless $(n,q) = (6,2)$. The  groups with $G_0 = {\rm L}_{6}^{\e}(2)$ can be handled using {\sc Magma}, so it remains to consider semisimple elements of odd prime order. Let $i \geqs 1$ be minimal such that $r$ divides $q^i-1$. Note that if $i \leqs 2$ then $r+1 \leqs q+2$ and the result follows as above, so we may assume $i \geqs 3$. In particular, this forces $\nu(x) \geqs 3$ and thus \cite[Corollary 3.38]{Bur2} implies that
\[
|x^G|>\frac{1}{2}\left(\frac{q}{q+1}\right)q^{6n-18}.
\]
By considering $|G_0|$ we see that $r \leqs (q^n-1)/(q-1)$ and one can check that this lower bound on $|x^G|$ (combined with \eqref{e:psl}) is sufficient if $n \geqs 9$. In fact, the same bound is also effective if $n=8$ and $q \geqs 4$. It is easy to check that if $n=8$ and $q=3$ then $r \leqs 1093$, whereas $r \leqs 127$ if $q=2$; the previous argument now goes through. Similarly, if $n=7$ then $\iota = 0$ in Theorem \ref{t:fpr} and once again the result follows as above.

To complete the proof, we may assume $n \in \{5,6\}$ (with $i \geqs 3$ as above). If $n=5$ then $\iota = 0$, $|x^G|>\frac{1}{2}q^{18}$ (minimal if $\e=+$ and $i=3$) and we conclude by applying the bound in Theorem \ref{t:fpr}. Now assume $n=6$ and recall that we may assume $q \geqs 3$ since we have already handled the case $q=2$ using {\sc Magma}. If $\iota = 0$ then the previous argument applies, so let us assume $\iota >0$, in which case $\iota = 1/6$ and $H$ is of type ${\rm Sp}_{6}(q)$ (see \cite[Table 1]{Bur1}). Here we have $i \in \{3,4,6\}$ and it is easy to check that $|x^G|>\frac{1}{2}q^{24}$ and $r \leqs q^2+q+1$. By applying the bound in \eqref{e:psl} we deduce that ${\rm fpr}(x) \leqs (r+1)^{-1}$ and the argument is complete.
\end{proof}

In the next lemma we assume $G_0 = {\rm PSp}_{n}(q)$ with $n \geqs 6$, noting that the case $n=4$ is handled separately in Lemma \ref{l:psp4} below.

\begin{lem}\label{l:symp}
Suppose $G_0 = {\rm PSp}_{n}(q)$ with $n \geqs 6$. If $H$ is a non-subspace subgroup, then ${\rm fpr}(x) \leqs (r+1)^{-1}$ for all $x \in G$ of prime order $r$.
\end{lem}

\begin{proof}
Suppose $x \in G$ has prime order $r$ and note that $\iota \leqs 1/n$, so \eqref{e:psl} holds once again. Now $|x^G| \geqs (q^n-1)/2$ (minimal if $G = G_0$, $q$ is odd and $x$ is a transvection) and one can check that \eqref{e:psl} yields ${\rm fpr}(x) \leqs (q+2)^{-1}$ unless $n=6$ or $(n,q) = (8,2)$. The latter case, together with $(n,q) = (6,2)$, can be handled using {\sc Magma}. For $n=6$ with $q \geqs 3$ we have $\iota = 0$ and the upper bound in Theorem \ref{t:fpr} is sufficient. For the remainder, we may assume $x$ is semisimple and $r \geqs 5$. Let $i \geqs 1$ be minimal such that $r$ divides $q^i-1$. By arguing as above, we may assume that $r>q+1$, so $i \geqs 3$ and thus $\nu(x) \geqs 4$. 

Notice that $r \leqs q^{n/2}+1$ and \cite[Proposition 3.36]{Bur2} gives $|x^G|>\frac{1}{2}q^{4n-16}$. One can now check that the bound in \eqref{e:psl} is sufficient for $n \geqs 10$, so the problem is reduced to the groups where $n \in \{6,8\}$ and $q \geqs 3$. Suppose $n=6$. Here $\iota = 0$, $i \in \{3,4,6\}$ and $|x^G| > \frac{1}{2}q^{16}$ (minimal if $i=4$). In addition, $r \leqs q^2+q+1$ and the result follows via the bound in Theorem \ref{t:fpr}. Similarly, if $n=8$ then $r \leqs q^4+1$, $|x^G|>\frac{1}{2}q^{24}$ (once again, minimal if $i=4$) and the bound in \eqref{e:psl} is sufficient.
\end{proof}

\begin{lem}\label{l:orth}
Suppose $G_0 = {\rm P\O}_{n}^{\e}(q)$ with $n \geqs 7$. If $H$ is a non-subspace subgroup, then ${\rm fpr}(x) \leqs (r+1)^{-1}$ for all $x \in G$ of prime order $r$.
\end{lem}

\begin{proof}
This is very similar to the proof of the previous lemma. First assume $n$ is odd (so $q$ is also odd) and write $n=2m+1$. The case $(n,q) = (7,3)$ can be handled using {\sc Magma}, so we may assume $(n,q) \ne (7,3)$. By inspecting \cite[Table 1]{Bur1} we see that $\iota = 0$ if $n \geqs 9$, otherwise $\iota \leqs 0.108$. Let us also note that 
\[
|x^G| \geqs \frac{|{\rm SO}_{n}(q)|}{2|{\rm SO}_{n-1}^{-}(q)|} = \frac{1}{2}q^m(q^m-1)
\]
with equality if $x \in {\rm SO}_{n}(q)$ is an involution with a minus-type eigenspace on $V$ of dimension $n-1$. By applying Theorem \ref{t:fpr} we deduce that ${\rm fpr}(x) \leqs (q+2)^{-1}$. 

To complete the argument for $n$ odd, we may assume $x$ is semisimple, $r \geqs 5$ and $i \geqs 3$, where $i$ is the smallest positive integer such that $r$ divides $q^i-1$. In particular, we have $\nu(x) \geqs 4$ and we quickly deduce that 
\[
|x^G| \geqs \frac{|{\rm SO}_{n}(q)|}{|{\rm SO}_{n-4}(q)||{\rm GU}_{1}(q^2)|}>\frac{1}{2}q^{4n-12}.
\]
In addition, we note that $r \leqs \frac{1}{2}(q^{m}+1)$ and it is routine to check that the desired bound now follows via Theorem \ref{t:fpr}.

For the remainder, we may assume $n = 2m \geqs 8$ is even. The groups with $(n,q) = (8,2)$ or $(8,3)$ can be handled using {\sc Magma}. Let us highlight the special case $G =  \O_{8}^{+}(2)$ with $H = {\rm Sp}_6(2)$ acting irreducibly on $V$: if $x \in G_0$ has order $3$ and $C_V(x) = 0$, then ${\rm fpr}(x) = 3/10>1/4$. Here the action of $G$ on $\O$ is permutation isomorphic to the action of $G$ on the set of $1$-dimensional nonsingular subspaces of the natural module, so this special case is included in part (i)(d) of Theorem \ref{t:main}. For the remainder we may assume $(n,q) \ne (8,2), (8,3)$. 

We will postpone the analysis of the special case where $(n,\e) = (8,+)$ and $H$ is an irreducible subgroup with socle  ${\rm Sp}_6(q)$ (if $p=2$) or $\O_7(q)$ (if $p$ is odd) to the end of the proof. By excluding this special case, we observe that $\iota \leqs 1/(n-2)$ in Theorem \ref{t:fpr} and it is easy to check that 
\[
|x^G| \geqs \frac{|{\rm O}_{n}^{\e}(q)|}{2d|{\rm Sp}_{n-2}(q)|} = \frac{1}{d}q^{m-1}(q^m-\e)
\]
with $d = (2,q-1)$. By combining this lower bound with Theorem \ref{t:fpr}, setting $\iota = 1/(n-2)$, we deduce that ${\rm fpr}(x) \leqs (q+2)^{-1}$. Therefore, to complete the analysis we may assume $x$ is semisimple, $r \geqs 5$ and $i \geqs 3$, so $\nu(x) \geqs 4$ and we have
\begin{equation}\label{e:bd2}
|x^G| \geqs \frac{|{\rm SO}_{n}^{+}(q)|}{|{\rm SO}_{n-4}^{-}(q)||{\rm GU}_{1}(q^2)|}>\frac{1}{2}q^{4n-12}.
\end{equation}
Now $r \leqs q^m+1$ and by applying Theorem \ref{t:fpr} we deduce that ${\rm fpr}(x) \leqs (r+1)^{-1}$ as required.

Finally, to complete the proof we may assume $(n,\e) = (8,+)$, $q \geqs 4$  and $H$ is irreducible with socle ${\rm Sp}_6(q)$ (if $p=2$) or $\O_7(q)$ (if $p$ is odd). Here \cite[Table 1]{Bur1} gives $\iota = 0.219$. If $r=2$ then the bound in Theorem \ref{t:ls91} is sufficient, so we may assume $r$ is odd and thus
\[
|x^G| \geqs \frac{|{\rm SO}_{8}^{+}(q)|}{q^9|{\rm SO}_{4}^{+}(q)||{\rm Sp}_{2}(q)|} = (q^2+1)^2(q^6-1)
\]
(minimal if $x$ is unipotent with Jordan form $(J_2^2,J_1^4)$). Then as above, by applying the lower bound in Theorem \ref{t:fpr} with $\iota = 0.219$, we deduce that ${\rm fpr}(x) \leqs (q+2)^{-1}$. Finally, we may assume $x$ is semisimple, $r \geqs 5$ and $i \geqs 3$, so $i \in \{3,4,6\}$ and $r \leqs q^2+q+1$. We can now proceed as before, using the lower bound on $|x^G|$ in \eqref{e:bd2}.  
\end{proof}

In order to complete the proof of Theorem \ref{t:main} for classical groups in non-subspace actions, we may assume $G_0$ is one of the following groups:
\[
{\rm L}_{2}(q), \; {\rm L}_{3}^{\e}(q), \; {\rm L}_{4}^{\e}(q),\; {\rm PSp}_{4}(q).
\]

First assume $G_0 = {\rm L}_{2}(q)$. We refer the reader to \cite[Tables 8.1, 8.2]{BHR} for a convenient list of the maximal subgroups of $G$ up to conjugacy. We will assume $q \geqs 7$ and $q \ne 9$, noting that $G_0$ is isomorphic to an alternating group when $q = 4,5$ or $9$. Note that in the special case arising in the following lemma, $G$ is permutation isomorphic to ${\rm L}_{3}(2)$ acting on the set of $1$-dimensional subspaces of its natural module. In particular, this case is included in part (i)(d) of Theorem \ref{t:main}.

\begin{lem}\label{l:psl2}
Suppose $G_0 = {\rm L}_{2}(q)$ and $H$ is a non-subspace subgroup, where $q \geqs 7$ and $q \ne 9$. If $x \in G$ has prime order $r$, then either ${\rm fpr}(x) \leqs (r+1)^{-1}$, or $G = {\rm L}_{2}(7)$, $H = S_4$, $x$ is an involution and ${\rm fpr}(x) = 3/7$.
\end{lem}

\begin{proof}
We need to consider the various possibilities for $H$ arising in \cite[Tables 8.1, 8.2]{BHR}. We begin by handling the subfield subgroups.

\vs

\noindent \emph{Case 1. $H$ is a subfield subgroup.}

\vs

Suppose $H$ is a subfield subgroup of type ${\rm GL}_{2}(q_0)$, where $q=q_0^k$ with $k$ a prime and $q_0 \geqs 3$ (see \cite[Table 8.1]{BHR}). Let $x \in G$ be an element of prime order $r$ and note that $H \cap {\rm PGL}(V) \leqs {\rm PGL}_{2}(q_0)$. Let $i_m(X)$ be the number of elements of order $m$ in the finite group $X$.

If $r=p$ and $x$ is unipotent, then $|x^G \cap H| \leqs q_0^2-1$ and $|x^G| \geqs (q^2-1)/2$, whence
\[
{\rm fpr}(x) \leqs \frac{2(q_0^2-1)}{q^{2}-1} \leqs \frac{2}{q_0^2+1} \leqs \frac{1}{q_0+1}
\]
and the result follows. Similarly, if $x$ is a semisimple involution, then $|x^G \cap H| \leqs i_2({\rm PGL}_{2}(q_0)) = q_0^2$, $|x^G| \geqs q(q-1)/2$ and we deduce that ${\rm fpr}(x) \leqs 1/3$ as required.

Next assume $x$ is semisimple and $r$ is odd. As usual, we may assume $r$ divides $|H_0|$, which implies that $r$ divides $q_0^2-1$. If $r$ divides $q_0-1$ then 
\[
{\rm fpr}(x) = \frac{|x^{G_0}\cap H_0|}{|x^{G_0}|} = \frac{q_0(q_0+1)}{q(q+1)} \leqs \frac{q_0+1}{q_0(q_0^2+1)} \leqs \frac{1}{q_0}
\]
and the result follows since $r \leqs q_0-1$. Similarly, if $r$ divides $q_0+1$ then 
\[
{\rm fpr}(x) \leqs \frac{q_0-1}{q_0(q_0^2+1)} \leqs \frac{1}{q_0+2}
\]
and once again this bound is sufficient.

Finally, suppose $q=q_1^r$ and $x$ is a field automorphism. First assume $k$ is odd. If $r=2$ then 
\[
|x^G \cap H| \leqs \frac{|{\rm PGL}_{2}(q_0)|}{|{\rm PGL}_{2}(q_0^{1/2})|} = q_0^{1/2}(q_0+1),\;\; |x^G| \geqs \frac{1}{2}q^{1/2}(q+1)
\]
and we quickly deduce that ${\rm fpr}(x) \leqs 1/3$. Similarly, if $r \geqs 3$ then 
\[
|x^G \cap H| \leqs |{\rm PGL}_{2}(q_0)|<q,\;\; |x^G|>\frac{1}{2}q^{3\left(1-\frac{1}{r}\right)},
\]
which implies that 
\[
{\rm fpr}(x) \leqs 2q_1^{3-2r} \leqs 2^{4-2r} \leqs (r+1)^{-1}.
\]
Now assume $k=2$. If $r=2$ then 
\[
|x^G \cap H| \leqs i_2({\rm PGL}_{2}(q_0))+1 \leqs q+1,\;\; |x^G| \geqs \frac{1}{2}q^{1/2}(q+1)
\]
and we obtain ${\rm fpr}(x) \leqs 2q_0^{-1}$. This is at most $1/3$ if $q_0 \geqs 7$ and the remaining cases $q_0 \in \{4,5\}$ can be checked directly (recall that $q \ne 4,9$). Finally, if $r \geqs 3$ then $q_1 = q_2^2$ and we have
\[
|x^G \cap H| \leqs |{\rm PGL}_{2}(q_0)|  < q^{3/2},\;\; |x^G| = \frac{|{\rm PGL}_{2}(q)|}{|{\rm PGL}_{2}(q^{1/r})|} > (q+1)^{2}.
\]
Therefore, ${\rm fpr}(x) <q_0^{-1} \leqs 2^{-r}$ and the result follows.

\vs

\noindent \emph{Case 2. $H$ is of type ${\rm GL}_{1}(q) \wr S_2$ or ${\rm GL}_{1}(q^2)$.}

\vs

Here $H$ is the normalizer of a maximal torus of $G_0$ and we have $H \cap {\rm PGL}(V) \leqs D_{2(q-\e)}$, where $\e=1$ if $H$ is of type ${\rm GL}_{1}(q) \wr S_2$, otherwise $\e=-1$. Let $x \in H$ be an element of prime order $r$.

First assume $x$ is semisimple or unipotent. If $r=2$ then $|x^G \cap H| \leqs i_2(D_{2(q-\e)}) \leqs q+2$ and $|x^G| \geqs q(q-1)/2$. These bounds yield ${\rm fpr}(x) \leqs 1/3$ for $q \geqs 11$ and the cases $q \in \{7,8\}$ can be checked directly. On the other hand, if $r$ is odd then $r$ divides $q-\e$ and the result follows since $|x^{G_0} \cap H| = 2$ and $|x^{G}| \geqs q(q-1)$.

Finally, suppose $q = q_0^r$ and $x$ is a field automorphism. If $r \geqs 3$ then $|x^G \cap H| \leqs 2(q+1)$ and $|x^G| > (q+1)^2$, whence ${\rm fpr}(x) \leqs 2(q+1)^{-1} < 2^{1-r}$ and the result follows. For $r=2$ we have ${\rm fpr}(x) = 0$ if $\e=-1$, whereas $|x^G \cap H| \leqs 2q^{1/2}$ and $|x^G| \geqs q^{1/2}(q+1)/2$ if $\e=1$. From the latter bounds we obtain ${\rm fpr}(x) \leqs 1/3$ since $q \geqs 16$.

\vs

\noindent \emph{Case 3. The remaining possibilities for $H$.} 

\vs

First assume $H$ is of type $2^{1+2}_{-}.{\rm O}_{2}^{-}(2)$, so $H = A_4$ or $S_4$, $q=p \geqs 7$ and we may assume $r \in \{2,3\}$. Here $|x^G \cap H| \leqs i_r(H) \leqs 9$ and $|x^G| \geqs q(q-1)/2$; these bounds are sufficient unless $r=2$ and $q=7$. In the latter case,  $G = {\rm L}_{2}(7)$, $H = S_4$, $|x^G \cap H| = 9$ and $|x^G| = 21$, which gives ${\rm fpr}(x) = 3/7 > 1/3$. This is the special case recorded in the statement of the lemma.

Finally, suppose $H = S_5$ or $A_5$, $q \in \{p,p^2\}$ and $p \equiv \pm 1, \pm 3 \imod{10}$, so $r \in \{2,3,5\}$. First assume $r=2$, so $|x^G| \geqs q^{1/2}(q+1)/2$ (minimal if $x$ is an involutory field automorphism) and we note that $i_2(H) \leqs 25$. The subsequent bound on ${\rm fpr}(x)$ is less than $1/3$ if $q \geqs 29$ and the cases with $q<29$ can be checked very easily with the aid of {\sc Magma}. Similarly, if $r \in \{3,5\}$ then the bounds $|x^G \cap H| \leqs 24$ and $|x^G| \geqs q(q-1)$ are sufficient for $q \geqs 13$ and once again we can use {\sc Magma} when $q<13$.
\end{proof}

\begin{lem}\label{l:psl3}
Suppose $G_0 = {\rm L}_{3}^{\e}(q)$ and $H$ is a non-subspace subgroup. Then ${\rm fpr}(x) \leqs (r+1)^{-1}$ for all $x \in G$ of prime order $r$.
\end{lem}

\begin{proof}
Since ${\rm U}_{3}(2)$ is solvable and ${\rm L}_{3}(2) \cong {\rm L}_{2}(7)$, we may assume $q \geqs 3$. The groups with $3 \leqs q \leqs 13$ can be checked using {\sc Magma}, so for the remainder, we may assume $q \geqs 16$. Let $x \in G$ be an element of prime order $r$ and set $H_0 = H \cap G_0$.

By Theorem \ref{t:ls91}, we have ${\rm fpr}(x) \leqs 4/3q$ (there are no exceptions in \cite[Table 1]{LS91} with $G_0 = {\rm L}_{3}^{\e}(q)$ and $q \geqs 16$). As a consequence, we may assume $r \geqs 13$ divides $|H_0|$ and $x \in G_0$ is either semisimple or unipotent (see \eqref{e:ls91}). Therefore, 
\begin{equation}\label{e:3bd}
|x^G| \geqs \frac{|{\rm GU}_{3}(q)|}{q^3|{\rm GU}_{1}(q)|^2} = (q^3+1)(q-1),
\end{equation}
with equality if $\e=-$, $r=p$ and $x$ has Jordan form $(J_2,J_1)$. Let us also note that $r \leqs q^2+q+1$. If $r \ne p$ then let $i \geqs 1$ be minimal such that $r$ divides $q^i-1$. By inspecting 
\cite[Tables 8.3-8.6]{BHR}, recalling that we may assume $r$ divides $|H_0|$, we can reduce to the cases where $H$ is a geometric subgroup in one of the collections $\C_2$, $\C_3$, $\C_5$ or $\C_8$ (recall that we follow \cite{KL} in defining the various subgroup collections arising in Aschbacher's theorem \cite{asch}, which is consistent with \cite{BHR}).

Suppose $H$ is a $\C_2$-subgroup of type ${\rm GL}_{1}^{\e}(q) \wr S_3$. Since we may assume $r$ divides $|H_0|$, it follows that $r \leqs q+1$ and the trivial bound $|x^G \cap H| \leqs (q+1)^2$ with \eqref{e:3bd} is sufficient. Next assume $H$ is a $\C_3$-subgroup of type ${\rm GL}_{1}^{\e}(q^3)$. Here $r \ne p$ and $(\e,i) = (+,3)$ or $(-,6)$ since $r>3$. Moreover, $|x^{G_0} \cap H| = 3$ and the result follows via \eqref{e:3bd}.

Now assume $H$ is a subfield subgroup of type ${\rm GL}_{3}^{\e}(q_0)$, where $q=q_0^k$ and $k$ is an odd prime. Here $|H_0|<q^{8/3}$ and $r$ is at most $q_0^2+q_0+1$, which implies that $r \leqs q^{2/3}+q^{1/3}+1$. It is easy to check that the trivial bound $|x^G \cap H|<q^{8/3}$ combined with the lower bound on $|x^G|$ in \eqref{e:3bd} is sufficient. Similarly, if $H$ is of type ${\rm O}_{3}(q)$, then $q$ is odd, $r \leqs (q+1)/2$ and $|H_0| = q(q^2-1)$. Once again, the bound in \eqref{e:3bd} is effective.

To complete the proof of the lemma, we may assume $\e=+$, $q=q_0^2$ and $H$ is a subgroup of type ${\rm GL}_{3}^{\e'}(q_0)$.
Note that $r \leqs q+q^{1/2}+1$. If $\nu(x)=2$ then $|x^G| \geqs q(q^2-1)(q^3-1)/3$ (minimal if $r=p$ and $x$ has Jordan form $(J_3)$) and the result follows since $|H_0|<q^4$. On the other hand, if $\nu(x)=1$ then either $r=p$ and $x$ has Jordan form $(J_2,J_1)$, or $r \leqs q_0+1$ and $x = (I_2, \omega)$ (up to conjugacy), where $\omega$ is a primitive $r$-th root of unity. In the former case, $|x^G \cap H| \leqs (q_0-1)(q_0^3+1)$ and $|x^G| = (q+1)(q^3-1)$. And in the latter, we have $|x^{G_0} \cap H| \leqs q_0^2(q_0^2+q_0+1)$ and $|x^G| \geqs q^2(q^2+q+1)$. In both cases, it is routine to check that the given bounds are sufficient.
\end{proof}

Note that in the next lemma we can assume $G_0  \ne {\rm L}_{4}(2)$ since ${\rm L}_{4}(2) \cong A_8$. Let us also observe that each special case appearing in the statement is permutation isomorphic to a subspace action of an isomorphic classical group. For example, in (i), the action is permutation isomorphic to the action of an almost simple group with socle ${\rm PSp}_{4}(3)$ on the set of $2$-dimensional totally singular subspaces of the natural module. Similarly, in (ii) with the action of ${\rm O}_{6}^{-}(2)$ on the set of nonsingular $1$-spaces, and ${\rm P\O}_{6}^{+}(3).2$ (extended by an involutory graph automorphism) on a set of nondegenerate $1$-spaces in (iii). In addition, let us say that an involutory graph automorphism $x$ of $G_0 = {\rm L}_4^{\e}(q)$ is of \emph{symplectic-type} if $C_{G_0}(x)$ has socle ${\rm PSp}_4(q)$.

\begin{lem}\label{l:psl4}
Suppose $G_0 = {\rm L}_{4}^{\e}(q)$ and $H$ is a non-subspace subgroup. Assume $G_0 \ne {\rm L}_4(2)$ and let $x \in G$ be an element of prime order $r$. Then either ${\rm fpr}(x) \leqs (r+1)^{-1}$, or $r=2$ and one of the following holds:
\begin{itemize}\addtolength{\itemsep}{0.2\baselineskip}
\item[{\rm (i)}] $G_0 = {\rm U}_4(2)$, $H$ is of type ${\rm GU}_{1}(2) \wr S_4$, $x = (J_2,J_1^2)$ and ${\rm fpr}(x) = 2/5$;
\item[{\rm (ii)}] $G = {\rm U}_{4}(2).2$, $H$ is of type ${\rm Sp}_{4}(2)$, $x$ is a symplectic-type graph automorphism and ${\rm fpr}(x) = 4/9$; or
\item[{\rm (iii)}] $G = {\rm L}_{4}(3).2_2$, $H$ is of type ${\rm Sp}_{4}(3)$, $x$ is a symplectic-type graph automorphism and ${\rm fpr}(x) = 5/13$.
\end{itemize}
\end{lem}

\begin{proof}
The result for $q \leqs 7$ can be checked using {\sc Magma}, so we will assume $q \geqs 8$. Let $x \in G$ be an element of prime order $r$. In view of Theorem \ref{t:ls91}, we may assume $r \geqs 7$ and $x \in G_0$ is semisimple or unipotent. Note that $r \leqs q^2+q+1$. Let us also observe that 
\begin{equation}\label{e:psl4}
|x^G| \geqs \frac{|{\rm GL}_{4}(q)|}{|{\rm GL}_{3}(q)||{\rm GL}_{1}(q)|}  = q^3(q^2+1)(q-1),
\end{equation}
minimal if $\e = +$ and $x$ is semisimple with $\nu(x) = 1$. 

By inspecting \cite[Tables 8.8-8.11]{BHR}, noting that we may assume $r$ divides $|H_0|$, we deduce that $H$ is either contained in one of the geometric subgroup collections labeled $\C_{\ell}$ with $\ell \in \{2,3,5,8\}$, or $H \in \mathcal{S}$ is a non-geometric subgroup with socle ${\rm L}_{2}(7)$ or $A_7$. For the non-geometric subgroups we have $r=7$, $q = p \geqs 11$ and the bound in Theorem \ref{t:ls91} is sufficient. We now consider the remaining possibilities for $H$ in turn. As usual, if $x$ is semisimple then we define $i$ to be the smallest positive integer such that $r$ divides $q^i-1$. It will be useful to note that the constant $\iota$ in Theorem \ref{t:fpr} is zero, unless $H$ is a $\C_8$-subgroup of type ${\rm Sp}_{4}(q)$, in which case $\iota = 1/4$.

First assume $H$ is a $\C_2$-subgroup of type ${\rm GL}_{1}^{\e}(q) \wr S_4$ or ${\rm GL}_{2}^{\e}(q) \wr S_2$. In the former case, we have $r \leqs q+1$ and the trivial bound $|x^G \cap H| \leqs (q+1)^3$ combined with \eqref{e:psl4} is sufficient. Now assume $H$ is of type ${\rm GL}_{2}^{\e}(q) \wr S_2$ and note that $r \leqs q+1$ once again. If $\nu(x) = 1$ then 
\[
|x^{G_0} \cap H| \leqs 2\left(\frac{|{\rm GL}_{2}(q)|}{|{\rm GL}_{1}(q)|^2}\right) = 2q(q+1)
\]
(maximal if $\e=+$ and $x$ is semisimple) and the bound in \eqref{e:psl4} is sufficient. Now assume $\nu(x) \geqs 2$. Here
\begin{equation}\label{e:psl411}
|x^G| \geqs \frac{|{\rm GL}_{4}(q)|}{2q^4|{\rm GL}_{2}(q)|} = \frac{1}{2}q(q^3-1)(q^4-1)
\end{equation}
(minimal if $\e=+$, $G = G_0$ and $x$ is unipotent with Jordan form $(J_2^2)$). By applying the bound in Theorem \ref{t:fpr}, noting that $\iota = 0$, we deduce that ${\rm fpr}(x) \leqs (q+2)^{-1}$ and the result follows.

Next let us assume $H$ is of type ${\rm GL}_{2}(q^2)$. Here $r \leqs q^2+1$ and $\nu(x) \geqs 2$ for all $x \in H_0$. In particular, if $r \leqs q+1$ then the previous argument applies (using Theorem \ref{t:fpr} with $\iota = 0$). Now assume $r>q+1$, so $i=4$ and $x$ is a regular semisimple element. Here 
\begin{equation}\label{e:psl42}
|x^G| \geqs \frac{|{\rm GL}_{4}(q)|}{|{\rm GL}_{1}(q^4)|} = q^6(q-1)(q^2-1)(q^3-1)
\end{equation}
and once again the desired result follows by applying Theorem \ref{t:fpr}. 

Now suppose $H$ is of type ${\rm GL}_{4}^{\e'}(q_0)$, where $q=q_0^k$ and $k$ is a prime. Here $r \leqs q_0^2+q_0+1 \leqs q+q^{1/2}+1$ and the result follows via Theorem \ref{t:fpr}, using the lower bound on $|x^G|$ in \eqref{e:psl4}. In order to complete the proof of the lemma, we may assume $H$ is a $\C_8$-subgroup of type ${\rm Sp}_{4}(q)$ or ${\rm O}_{4}^{\e'}(q)$. 

First assume $H$ is of type ${\rm O}_{4}^{\e'}(q)$, in which case $q$ is odd. If $r \leqs q$ then the usual argument (using Theorem \ref{t:fpr} and \eqref{e:psl4}) is sufficient. On the other hand, if $r>q$ then $\e'=-$, $x$ is semisimple, $i=4$, $r \leqs (q^2+1)/2$ and the bound in \eqref{e:psl42} is satisfied. We now conclude by applying Theorem \ref{t:fpr}.

Finally, suppose $H$ is of type ${\rm Sp}_{4}(q)$. Here $\iota = 1/4$, so the bound in Theorem \ref{t:fpr} is not useful and we need to consider the various possibilities for $x$ in turn. Fortunately, the embedding of $H$ in $G$ is transparent and it is easy to determine good bounds on ${\rm fpr}(x)$. First assume $r=p$. If $x$ has Jordan form $(J_2,J_1^2)$, then
\[
|x^G \cap H| \leqs \frac{|{\rm Sp}_{4}(q)|}{q^3|{\rm Sp}_{2}(q)|} = q^4-1
\]
and the bound $|x^G| \geqs (q^2-q+1)(q^4-1)$ is sufficient. Similarly, if $x = (J_2^2)$ then 
\[
|x^G \cap H| \leqs \frac{|{\rm Sp}_{4}(q)|}{q^3|{\rm O}_{2}^{+}(q)|} + \frac{|{\rm Sp}_{4}(q)|}{q^3|{\rm O}_{2}^{-}(q)|} = q^2(q^4-1)
\]
and the result follows since \eqref{e:psl411} holds. And for $x = (J_4)$, the bounds $|x^G \cap H|<q^8$ and $|x^G|>\frac{1}{2}q^{12}$ are sufficient (here we are using the fact that ${\rm PGSp}_{4}(q)$ contains precisely $q^8$ unipotent elements).

Now suppose $r \ne p$, so $i \in \{1,2,4\}$ and $\nu(x) \in \{2,3\}$. If $\nu(x) = 3$ then $x$ is regular, $r \leqs q^2+1$, $|x^{G_0} \cap H| < 2q^8$ and the bound $|x^G|>\frac{1}{2}q^{12}$ is sufficient. Now assume $\nu(x) = 2$, so $i \in \{1,2\}$ and $r \leqs q+1$. Here it is straightforward to verify the bounds $|x^{G_0} \cap H|<2q^6$ and $|x^G|>\frac{1}{2}q^8$, whence ${\rm fpr}(x) < 4q^{-2} \leqs (q+2)^{-1}$ and the result follows.
\end{proof}

Finally, we handle the almost simple symplectic groups with socle ${\rm PSp}_{4}(q)$. Note that we may assume $q \geqs 4$ since ${\rm PSp}_{4}(2)' \cong A_6$ and ${\rm PSp}_{4}(3) \cong {\rm U}_4(2)$.
 
\begin{lem}\label{l:psp4}
Suppose $G_0 = {\rm PSp}_{4}(q)$ and $H$ is a non-subspace subgroup, where $q \geqs 4$. Then ${\rm fpr}(x) \leqs (r+1)^{-1}$ for all elements $x \in G$ of prime order $r$. 
\end{lem}

\begin{proof}
Let $x \in G$ be an element of prime order $r$ and recall that we may assume $r$ divides $|H_0|$, where $H_0 = H \cap G_0$. The groups with $q \leqs 16$ can be checked using {\sc Magma}, so we may assume $q \geqs 17$. In addition, by applying Theorem \ref{t:ls91}, we may assume that $r \geqs 13$ and $x \in G_0$ is either semisimple or unipotent. Note that either $r=p$ and
\begin{equation}\label{e:sp41}
|x^G| \geqs \frac{|{\rm Sp}_{4}(q)|}{dq^3|{\rm Sp}_{2}(q)|} = \frac{1}{d}(q^4-1),
\end{equation}
where $d = (2,q-1)$, or $r \ne p$ and
\begin{equation}\label{e:sp42}
|x^G| \geqs \frac{|{\rm Sp}_{4}(q)|}{|{\rm GU}_{2}(q)|} = q^3(q-1)(q^2+1).
\end{equation}
We now partition the argument into two cases, according to the parity of $q$.

\vs

\noindent\emph{Case 1. $q$ odd.}

\vs

To begin with, we will assume $q$ is odd. By inspecting \cite[Tables 8.12, 8.13]{BHR}, we may assume $H$ is either a geometric subgroup in one of the collections $\C_2$, $\C_3$ or $\C_5$, or $H$ is a non-geometric subgroup with socle ${\rm L}_{2}(q)$. We consider each possibility in turn.

First assume $H$ is a $\C_2$-subgroup of type ${\rm Sp}_{2}(q) \wr S_2$. Suppose $x$ is unipotent, so $r=p$ and $x$ has Jordan form $(J_2,J_1^2)$ or $(J_2^2)$ since we may assume $x \in H$. In the first case,  
\[
|x^G \cap H| \leqs 2\left(\frac{|{\rm Sp}_{2}(q)|}{q}\right) = 2(q^2-1)
\]
and the bound in \eqref{e:sp41} is sufficient. Similarly, if $x = (J_2^2)$ then $|x^G \cap H| \leqs (q^2-1)^2$ and the result follows since
\begin{equation}\label{e:sp43}
|x^G| \geqs \frac{|{\rm Sp}_{4}(q)|}{q^3|{\rm O}_{2}^{-}(q)|} = \frac{1}{2}q(q-1)(q^4-1).
\end{equation}
Now assume $r \ne p$ and note that $r \leqs q+1$. Here  
\[
|x^{G_0} \cap H| \leqs \left(\frac{|{\rm Sp}_{2}(q)|}{|{\rm GL}_{1}(q)|}\right)^2 = q^2(q+1)^2
\]
and the bound in \eqref{e:sp42} is sufficient.

Next suppose $H$ is of type ${\rm GL}_{2}^{\e}(q)$, so $r \leqs q+1$. If $r=p$ then $x = (J_2^2)$ is the only option and it is easy to check that the bounds $|x^G\cap H| \leqs q^2-1$ and \eqref{e:sp43} are effective. Similarly, if $x$ is semisimple then 
\[
|x^{G_0} \cap H| \leqs 2\left(\frac{|{\rm GL}_{2}(q)|}{|{\rm GL}_{1}(q)|^2}\right) = 2q(q+1)
\]
and the result follows via \eqref{e:sp42}. 

A similar argument applies when $H$ is of type ${\rm Sp}_{2}(q^2)$. Indeed, if $r=p$ then $x = (J_2^2)$, $|x^G \cap H| \leqs q^4-1$ and we conclude by applying the bound in \eqref{e:sp43}. For $r \ne p$ we have $r \leqs q^2+1$ and the bounds $|x^{G_0} \cap H| \leqs q^2(q^2+1)$ and \eqref{e:sp42} are sufficient.

To complete the argument for $q$ odd, we may assume that $H$ is either a subfield subgroup of type ${\rm Sp}_{4}(q_0)$, where $q=q_0^k$ with $k$ a prime, or $H$ is a non-geometric subgroup with socle ${\rm L}_{2}(q)$ and $p \geqs 5$. The latter case is easy to handle. Indeed, we have $|H_0|<q^3$, $r \leqs q+1$ and it is easy to check that the nontrivial unipotent elements in $H_0$ have Jordan form $(J_4)$ on the natural module $V$ for $G_0$ (this follows from the fact that $V = S^3(W)$, where $W$ is the natural $2$-dimensional module for $H_0$). Therefore, the bound in \eqref{e:sp42} holds for all $x \in H$ of prime order and the result follows.

Now assume $H$ is a subfield subgroup of type ${\rm Sp}_{4}(q_0)$, where $q=q_0^k$. First assume $k \geqs 3$, in which case $r \leqs q_0^2+1 \leqs q^{2/3}+1$. If $x= (J_2,J_1^2)$ then $|x^G \cap H| \leqs q_0^4-1 \leqs q^{4/3}-1$ and the bound in \eqref{e:sp41} is sufficient. For the remaining elements, the bound in \eqref{e:sp43} is satisfied and the trivial bound $|x^G \cap H| \leqs |H_0|<q_0^{10} \leqs q^{10/3}$ is good enough. 

Finally, suppose $k=2$. If $r=p$, then either $x = (J_2,J_1^2)$, $|x^G \cap H| \leqs q^2-1$ and \eqref{e:sp41} holds, or $|x^G \cap H|<q_0^8 = q^4$ (this upper bound is the total number of unipotent elements in $H_0$) and we have the bound on $|x^G|$ in \eqref{e:sp43}. In both cases, the given  bounds are sufficient. Now assume $r \ne p$ and note that $r \leqs q+1$. If $\nu(x) = 3$ then $|x^G|>\frac{1}{2}q^8$ and the trivial bound $|x^G \cap H| < |H_0|<q^5$ is good enough. Similarly, if $\nu(x)=2$ then \eqref{e:sp42} holds and the result follows since 
\[
|x^{G_0} \cap H| \leqs \frac{|{\rm Sp}_{4}(q_0)|}{|{\rm GL}_{2}(q_0)|} =q^{3/2}(q^{1/2}+1)(q+1).
\]

\vs

\noindent\emph{Case 2. $q$ even.}

\vs

To complete the proof of the lemma, we may assume $q \geqs 32$ is even. In view of Theorem \ref{t:ls91}, we may also assume that $r \geqs 23$. In particular, $x$ is semisimple and \eqref{e:sp42} holds. We now work through the various possibilities for $H$ arising in \cite[Table 8.14]{BHR}. 

First assume $H$ is a Borel subgroup, so $H_0 = [q^4]{:}C_{q-1}^2$ and thus $r \leqs q-1$ (note that $H$ is maximal when $G \not\leqs {\rm \Gamma Sp}_{4}(q)$). If $\nu(x)=3$, then $|x^G|>\frac{1}{2}q^{8}$ and $|x^{G_0} \cap H| = 8q^4$ as explained in the proof of \cite[Lemma 5.8]{Bur20}. Similarly, if $\nu(x)=2$ then $|x^{G_0} \cap H| = 4q^3$ and we have the bound on $|x^G|$ in \eqref{e:sp42}. In both cases, the given bounds are sufficient.

The argument when $H$ is of type ${\rm Sp}_{2}(q) \wr S_2$, ${\rm Sp}_{2}(q^2)$ or ${\rm Sp}_{4}(q_0)$ is entirely similar to the one given above in the case where $q$ is odd. For this reason, we omit the details. Next suppose $H$ is a non-geometric subgroup with socle ${}^2B_2(q)$, in which case $\log_2q$ is odd and we note that $|H_0|<q^5$ and $r \leqs q+\sqrt{2q}+1$. If $\nu(x) = 3$ then $|x^G|>\frac{1}{2}q^8$ and the trivial bound $|x^{G} \cap H|<q^5$ is sufficient. Similarly, if $\nu(x)=2$ then $r$ divides $q-1$, 
\[
|x^{G}| \geqs \frac{|{\rm Sp}_{4}(q)|}{|{\rm GL}_{2}(q)|} = q^3(q+1)(q^2+1)
\]
and once again the trivial bound $|x^{G} \cap H|<q^5$ is good enough.

Finally, let us assume $H = N_G(T)$ is the normalizer of a maximal torus (recall that the $\C_8$-subgroups of type ${\rm O}_{4}^{\e}(q)$ are subspace subgroups, so they are excluded here; see Remark \ref{r:sp4}). Here $H_0 < M < G_0$ for some maximal non-subspace subgroup $M$ of $G_0$ (indeed, $H$ is maximal only if $G \not\leqs {\rm \Gamma Sp}_{4}(q)$) and so the desired bound on ${\rm fpr}(x)$ automatically holds by our earlier work in this proof.
\end{proof}

\begin{rem}\label{r:sp4}
Using the same approach as in the proof of Lemma \ref{l:psp4}, it is straightforward to show that if $G_0 = {\rm PSp}_{4}(q)$, $q \geqs 4$ is even and $H$ is a subspace subgroup of type ${\rm O}_{4}^{\e}(q)$, then ${\rm fpr}(x) \leqs (r+1)^{-1}$ for all $x \in G$ of prime order $r$. See Lemma \ref{l:symp_sub4} for the details. 
\end{rem}

\section{Almost simple classical groups: Subspace actions}\label{s:class2}

In this section we handle the subspace actions of classical groups, which will complete the proof of Theorem \ref{t:main} for almost simple groups. Recall from Definition \ref{d:sub} that the subspace actions correspond to the groups where a point stabilizer $H$ is either contained in Aschbacher's $\C_1$ collection of maximal subgroups, or $G_0 = {\rm Sp}_n(q)$ is a symplectic group with $q$ even and $H \cap G_0 = {\rm O}_n^{\e}(q)$ is a naturally embedded orthogonal group (the stabilizer of a suitable nondegenerate quadratic form). So in almost all cases we may identify $\O$ with a set of subspaces (or pairs of subspaces) of the natural module, which makes subspace actions more amenable to direct computation since we have a concrete description of the action. 

\subsection{Main result and notation}\label{ss:res}

Our main theorem is the following. 

\begin{thm}\label{t:subspace}
Let $G \leqs {\rm Sym}(\O)$ be an almost simple finite primitive permutation group with socle $G_0$ and point stabilizer $H$. Let $x \in G$ be an element of prime order $r$ and assume $G_0$ is a classical group over $\mathbb{F}_q$ and $H$ is a subspace subgroup. Then either 
${\rm fpr}(x) \leqs (r+1)^{-1}$, or one of the following holds:
\begin{itemize}\addtolength{\itemsep}{0.2\baselineskip}
\item[{\rm (i)}] $G$ is permutation isomorphic to a group recorded in part (a) or (b) in Theorem \ref{t:main}(i); 
\item[{\rm (ii)}] $(G,H,x,{\rm fpr}(x))$ is one of the cases listed in Table \ref{tab:class}.
\end{itemize}
\end{thm}

\renewcommand{\arraystretch}{1.2}
{\small \begin{table}
\[
\begin{array}{lcccll} \hline
G_0 & \mbox{Type of $H$} & x & r & {\rm fpr}(x) & \mbox{Conditions} \\ \hline
{\rm L}_{n}(q) & P_1 & (J_2,J_{1}^{n-2}) & q & \frac{1}{q+1} + \frac{q(q^{n-2}-1)}{(q+1)(q^n-1)} & n \geqs 3 \\
\mbox{{\tiny $n \geqs 2$}} & & (\omega, I_{n-1}) & q-1 & \frac{1}{q} +\frac{(q-1)^2}{q(q^n-1)} & \\ 
& & \varphi & 3 & \frac{1}{3} & (n,q) = (2,8) \\
& & & & & \\
{\rm U}_{n}(q) & P_1 & (\omega, I_{n-1}) & 3 & \frac{1}{4}+\frac{3}{4(2^n+1)} & \mbox{$n \geqs 5$ odd, $q=2$} \\
\mbox{{\tiny $n \geqs 3$}} & P_2 & \tau & 2 & \frac{5}{9} \, (q=2); \; \frac{5}{14} \, (q=3) & \mbox{$n=4$, $q \in \{2,3\}$} \\
& &  (\omega I_2, I_2) & 3 & \frac{1}{3} & (n,q) = (4,2) \\
& N_1 & (\omega, I_{n-1}) & 3 & \frac{1}{4} + \frac{3(2^{n-3}+1)}{2^{n-1}(2^n-1)} & \mbox{$n$ even, $q=2$} \\
& & & & & \\
{\rm PSp}_{n}(q) & P_1 & (J_2,J_1^{n-2}) & q & \frac{1}{q+1} + \frac{q(q^{n-2}-1)}{(q+1)(q^n-1)} & \\
\mbox{{\tiny $n \geqs 4$}} & {\rm O}_{n}^{\e}(q) & (J_2,J_1^{n-2}) & 2 & \frac{1}{3}+\frac{2^{n/2-1}-\e}{3(2^{n/2}+\e)} & \mbox{$n \geqs 6$, $q=2$}  \\
& {\rm O}_{n}^{-}(q) & (\Lambda, I_{n-2}) & 3 & \frac{1}{4}+\frac{3}{4(2^{n/2}-1)} & \mbox{$n \geqs 6$, $q= 2$} \\
& & & & & \\
\O_n(q) & P_1 & (-I_{n-1}, I_1)^{+} & 2 & \frac{1}{3} + \frac{2}{3(3^{(n-1)/2}+1)} & q = 3 \\ 
\mbox{{\tiny $n \geqs 7$}} & N_1^{-} & (-I_{n-1}, I_1)^{-} & 2 & \frac{1}{3}+\frac{2(3^{(n-3)/2}+1)}{3^{(n-1)/2}(3^{(n-1)/2}-1)} & q = 3 \\
& & & & & \\ 
{\rm P\O}_{n}^{\e}(q) & P_1 & (J_2,J_1^{n-2}) & 2 & \frac{1}{3}+\frac{2^{n-2}-\e 2^{n/2-1}-2}{3(2^{n/2-1}+\e)(2^{n/2}-\e)} & q = 2 \\ 
\mbox{{\tiny $n \geqs 8$}} & & (-I_{n-1},I_{1}) & 2 & \frac{1}{3}+\frac{2}{3(3^{n/2}+1)} & (q,\e) = (3,-) \\
\mbox{{\tiny $\e=\pm$}} & & (\Lambda,I_{n-2}) & 3 & \frac{1}{4}+\frac{3}{4(2^{n/2}+1)} & (q,\e) = (2,-) \\ 
& P_2 & (\Lambda,I_6) & 5 & \frac{1}{5} & (n,q,\e) = (8,4,+) \\
& N_1 & (-I_{n-1},I_{1})^{\boxtimes} & 2 & \frac{1}{3}+\frac{4}{3(3^{n/2}-1)} & (q,\e) = (3,+) \\
& &  (-I_{n-1},I_{1})^{\square} & 2 & \frac{1}{3}+\frac{2(3^{n/2-2
}+1)}{3^{n/2-1}(3^{n/2}+1)} & (q,\e) = (3,-) \\ 
& & (J_2,J_1^{n-2}) & 2 & \frac{1}{3}+\frac{2^{n/2-1}+\e}{3(2^{n/2}-\e)} & q = 2 \\
& & (\Lambda,I_{n-2}) & 3 & \frac{1}{4}+\frac{3}{4(2^{n/2}-1)} & (q,\e) = (2,+) \\ \hline
\end{array}
\]
\caption{The exceptional subspace actions of classical groups}
\label{tab:class}
\end{table}}
\renewcommand{\arraystretch}{1}

\begin{rem}
Write $q=p^f$ where $p$ is a prime. In the proof of Theorem \ref{t:subspace}, we will often establish stronger upper bounds. For example, if $x$ is unipotent, then $r=p$ and we usually aim to show that ${\rm fpr}(x) \leqs (q+1)^{-1}$. Similarly, if $x$ is semisimple and $r$ is odd then we typically establish a bound on ${\rm fpr}(x)$ in terms of $q$ and the order $i$ of $q$ mod $r$ (that is, in terms of the smallest positive integer $i$ such that $r$ divides $q^i-1$). More precisely, if $i$ is even then we will often show that ${\rm fpr}(x) \leqs (q^{i/2}+2)^{-1}$, which establishes the desired bound since $r$ divides $q^{i/2}+1$.  Similarly, if $i$ is odd then we typically aim to show that ${\rm fpr}(x) \leqs q^{-i}$.
\end{rem}

\begin{rem}
Explicit bounds on fixed point ratios for subspace actions of classical groups are presented in \cite{FM,GK}. In particular, we will repeatedly apply the results of Guralnick and Kantor given in \cite[Section 3]{GK}.
\end{rem}

Before we embark on the proof of Theorem \ref{t:subspace}, we first need to define our notation for subspace actions, which we will use throughout Section \ref{s:class2}. The additional notation appearing in Table \ref{tab:class} is discussed in Remark \ref{r:tab} below.

The various possibilities for $G$ and $H$ that we need to consider are recorded in Table \ref{tab:subspace}, which provides a framework for the proof of Theorem \ref{t:subspace}. In the second column, we present the \emph{type of $H$}, which is designed to describe the type of subspace (or pair of subspaces) stabilized by $H$. For parabolic subgroups, our notation is consistent with \cite{KL}. In particular, we write $P_m$ to denote the stabilizer of a totally singular $m$-space (for $G_0 = {\rm L}_{n}(q)$, we adopt the standard convention that every subspace of $V$ is totally singular). Similarly, if $G_0 = {\rm L}_{n}(q)$ then $P_{m,n-m}$ and ${\rm GL}_{m}(q) \oplus {\rm GL}_{n-m}(q)$ denote the stabilizers of a pair of subspaces $(U,W)$ with $\dim U = m$ and $\dim W = n-m$, where $U \subset W$ in the first case and $V = U \oplus W$ in the second. As indicated in the final column, these two subgroups only arise when $G$ contains graph or graph-field automorphisms (indeed, if $G \leqs {\rm P\Gamma L}_{n}(q)$, then $H$ is non-maximal). Similarly, if $G_0 = {\rm PSp}_{4}(q)$ and $q$ is even then we write $P_{1,2}$ to denote a Borel subgroup of $G$. There is also a parabolic subgroup labeled $P_{1,3,4}$, which arises when $G_0 = {\rm P\O}_{8}^{+}(q)$ and $G$ contains triality automorphisms. 

For $G_0 = {\rm U}_{n}(q)$ or ${\rm PSp}_{n}(q)$, we write $N_m$ for the stabilizer of a nondegenerate $m$-space and this notation extends (with suitable modifications) to orthogonal groups. Suppose $G_0$ is an orthogonal group and let $Q$ be the defining quadratic form on $V$. First recall that if $U$ is a nondegenerate $m$-dimensional subspace of $V$ with $m$ even, then $U$ is a \emph{plus-type} space if it contains a totally singular subspace of dimension $m/2$, otherwise it is a \emph{minus-type} space (and every maximal totally singular subspace has dimension $m/2-1$). Now, if $G_0 = \O_n(q)$ with $nq$ odd then we write $N_{m}^{\eta}$ with $\eta=\pm$ to denote the stabilizer of a nondegenerate $m$-space $U$ with the property that either $m$ is even and $U$ has type $\eta$, or $m$ is odd and the orthogonal complement $U^{\perp}$ is a nondegenerate $(n-m)$-space of type $\eta$. Similarly, if $G_0 = {\rm P\O}_{n}^{\e}(q)$ with $n$ even, then $N_m^{\eta}$ is the stabilizer of a nondegenerate $m$-space of type $\eta$ when $m$ is even. On the other hand, if $m$ is odd then $q$ is odd and we adopt the convention that $N_m$ is the stabilizer of a nondegenerate $m$-space $U$ with square discriminant (that is, the discriminant of the restriction of $Q$ to $U$ is a square in $\mathbb{F}_{q}^{\times}$). We will use the term 
\emph{square $m$-space} to describe such a subspace, noting that the actions on the set of square and nonsquare $m$-spaces are permutation isomorphic (so there is no need to consider nonsquare spaces). Finally, if $n$ and $q$ are even, then we also write $N_1$ to denote the stabilizer of a nonsingular $1$-space.

Various conditions are recorded in the final column of Table \ref{tab:subspace}, which are designed to eliminate an unnecessary repetition of cases. For example, in the first row we may assume $m \leqs n/2$ because the action of $G$ on the set of $m$-spaces is permutation isomorphic to the action on $(n-m)$-spaces. We refer the reader to \cite{BHR,KL} for the precise conditions needed to ensure that the given subgroup $H$ is maximal in $G$.

\begin{table}
\[
\begin{array}{llllll} \hline
G_0 & \mbox{Type of $H$} & \mbox{Conditions} \\ \hline
{\rm L}_{n}(q) & P_m & 1 \leqs m \leqs n/2 \\
& P_{m,n-m} & \mbox{$1 \leqs m < n/2$, $G \not\leqs {\rm P\Gamma L}_{n}(q)$} \\
& {\rm GL}_{m}(q) \oplus {\rm GL}_{n-m}(q) & \mbox{$1 \leqs m < n/2$, $G \not\leqs {\rm P\Gamma L}_{n}(q)$} \\
{\rm U}_{n}(q) & P_m & 1 \leqs m \leqs n/2 \\
& N_m & 1 \leqs m < n/2 \\
{\rm PSp}_{n}(q) & P_m & 1 \leqs m \leqs n/2 \\
& P_{1,2} & \mbox{$n=4$, $p = 2$, $G \not\leqs {\rm P\Gamma Sp}_{4}(q)$} \\
& N_m & \mbox{$2 \leqs m < n/2$, $m$ even} \\
& {\rm O}_{n}^{\e}(q) & p = 2 \\
{\rm P\O}_{n}^{\e}(q) & P_m & 1 \leqs m \leqs n/2 \\
& N_m^{\eta} & \mbox{$1 \leqs m \leqs n/2$, $(\e,\eta) = (-,+)$ if $m=n/2$} \\
& P_{1,3,4} & \mbox{$(n,\e)=(8,+)$, $G \not\leqs {\rm P\Gamma O}_{8}^{+}(q)$} \\ \hline   
\end{array}
\]
\caption{The subspace actions}
\label{tab:subspace}
\end{table}

\begin{rem}\label{r:tab}
Let us comment on the conditions and notation in Table \ref{tab:class}. 
\begin{itemize}\addtolength{\itemsep}{0.2\baselineskip}
\item[{\rm (a)}] In the first column, we record the socle $G_0$ of $G$, noting that the given conditions on $n$ are justified in view of the well known isomorphisms among the low-dimensional classical groups (see \cite[Proposition 2.9.1]{KL}). 
\item[{\rm (b)}] In the second column, we describe the type of $H$ using the same notation as in Table \ref{tab:subspace}. The relevant elements $x$ are described in the third column. If $x$ is unipotent, then the Jordan form of $x$ on the natural module $V$ is presented, where $J_i$ denotes a standard unipotent Jordan block of size $i$. Similarly, if $x$ is semisimple, then we describe the eigenvalues of $x$ on $V$, up to scalars. Here $\omega \in \mathbb{F}_{q^u}^{\times}$ is a primitive $r$-th root of unity and we write $\L$ to denote an arbitrary irreducible element in ${\rm GL}_2(q)$ of order $r$ (in this case, $r$ divides $q+1$). 
\item[{\rm (c)}] In the two rows with $G_0 = \O_n(q)$, we write $(-I_{n-1},I_1)^{\e}$ to denote a semisimple involution whose $(-1)$-eigenspace on $V$ is nondegenerate of type $\e \in \{+,-\}$. Similarly, if $G_0 = {\rm P\O}_{n}^{\e}(q)$ with $n$ even, then $(-I_{n-1},I_1)^{\delta}$ is an involution whose $1$-eigenspace $\la v \ra$ has discriminant $\delta \in \{\square,\boxtimes\}$. That is, if $Q$ is the defining quadratic form, then $Q(v) \in \mathbb{F}_q^{\times}$ is a square or nonsquare according to $\delta$.
\item[{\rm (d)}] Finally, the element $\varphi$ in the third row is a field automorphism of order $3$, while $\tau$ in row $5$ is an involutory graph automorphism with $C_{G_0}(\tau) = {\rm PSp}_{4}(q)$.
\end{itemize}
\end{rem}

We continue to adopt the notation in \cite[Chapter 3]{BG} for unipotent and semisimple elements. In particular, we use the notation of Aschbacher and Seitz \cite{AS} for unipotent involutions in symplectic and orthogonal groups when $p=2$. 

\subsection{Linear groups}\label{ss:lin}

In this section, we begin the proof of Theorem \ref{t:subspace} by handling the linear groups with socle $G_0 = {\rm L}_{n}(q)$. First we consider the groups with $n=2$, in which case $H = P_1$ is a Borel subgroup. As before, we may assume $q \geqs 7$ and $q \ne 9$. Note that in part (i) of the following result, $q$ is even and $r$ is a Mersenne prime. 

\begin{prop}\label{p:psl2}
Let $G \leqs {\rm Sym}(\O)$ be a finite almost simple primitive permutation group with point stabilizer $H = P_1$ and socle $G_0 = {\rm L}_{2}(q)$ with $q \geqs 7$ and $q \ne 9$. Let $x \in G$ be an element of prime order $r$. Then either ${\rm fpr}(x) \leqs (r+1)^{-1}$, or one of the following holds:
\begin{itemize}\addtolength{\itemsep}{0.2\baselineskip}
\item[{\rm (i)}] $x \in G_0$ has order $r=q-1$ and 
\[
{\rm fpr}(x) = \frac{1}{q}+\frac{q-1}{q(q+1)}.
\] 
\item[{\rm (ii)}] $G = {\rm L}_{2}(8){:}3$, $x$ is a field automorphism of order $3$ and ${\rm fpr}(x) = 1/3$.
\end{itemize}
\end{prop}

\begin{proof}
Here $|\O| = q+1$ and we may identify $\O$ with the set of $1$-dimensional subspaces of $V$. First assume $x \in H \cap {\rm PGL}_{2}(q)$. Then either $r=p$ and $|C_{\O}(x)| = 1$, in which case the bound ${\rm fpr}(x) \leqs (r+1)^{-1}$ clearly holds, or $r$ divides $q-1$ and $|C_{\O}(x)| = 2$. In the latter case, we have ${\rm fpr}(x) = 2(q+1)^{-1}$, which is at most $(r+1)^{-1}$ if and only if $r \leqs (q-1)/2$. Notice that if $r>(q-1)/2$ then $r=q-1$, so $q$ is even, $r$ is a Mersenne prime and ${\rm fpr}(x) = 2(r+2)^{-1}$, which is the special case appearing in part (i). Finally, if $q=q_0^r$ and $x$ is a field automorphism, then $|C_{\O}(x)| = q_0+1$ and we deduce that either ${\rm fpr}(x) \leqs (r+1)^{-1}$, or $q_0 = 2$, $r=3$ and ${\rm fpr}(x) = 1/3$ as in part (ii).
\end{proof}

For the remainder of Section \ref{ss:lin}, we will assume $G_0 = {\rm L}_{n}(q)$ with $n \geqs 3$. Our main result is the following (in the statement, we assume $G_0 \ne {\rm L}_{3}(2), {\rm L}_{4}(2)$ since ${\rm L}_{3}(2) \cong {\rm L}_{2}(7)$ and ${\rm L}_{4}(2) \cong A_8$).

\begin{prop}\label{p:psl_sub}
Let $G \leqs {\rm Sym}(\O)$ be a finite almost simple primitive permutation group with point stabilizer $H$ and socle $G_0 = {\rm L}_{n}(q)$ with $n \geqs 3$ and $(n,q) \ne (3,2)$, $(4,2)$. Assume $H$ is a subspace subgroup and $x \in G$ has prime order $r$. Then either ${\rm fpr}(x) \leqs (r+1)^{-1}$, or $H = P_1$ and one of the following holds:
\begin{itemize}\addtolength{\itemsep}{0.2\baselineskip}
\item[{\rm (i)}] $r=p=q$, $x = (J_2,J_{1}^{n-2})$ and 
\[
{\rm fpr}(x) = \frac{1}{q+1}+\frac{q(q^{n-2}-1)}{(q+1)(q^n-1)}.
\]
\item[{\rm (ii)}] $r=q-1$, $x = (\omega, I_{n-1})$ and 
\[
{\rm fpr}(x) = \frac{1}{q}+\frac{(q-1)^2}{q(q^n-1)}.
\]
\end{itemize}
\end{prop}

We will prove Proposition \ref{p:psl_sub} in a sequence of lemmas, where we consider each possibility for $H$ arising in Table \ref{tab:subspace}. Before launching into the details, we present the following elementary lemma on the fixed points of semisimple elements in a subspace action. 

\begin{lem}\label{l:sub_useful}
Consider the natural action of $G = {\rm GL}_n(q)$ on the set $\O$ of $m$-dimensional subspaces of the natural module $V$. Let $x \in G$ be a semisimple element of odd prime order $r$ and let $i \geqs 1$ be minimal such that $r$ divides $q^i-1$, so $x$ is conjugate to  
\[
(\L_1^{a_1}, \ldots, \L_t^{a_t},I_e),
\]
where $\L_1, \ldots, \L_t$ represent the distinct conjugacy classes in ${\rm GL}_i(q)$ of elements of order $r$ and the $a_j$ are non-negative integers such that $i\sum_ja_j \leqs n$.
\begin{itemize}\addtolength{\itemsep}{0.2\baselineskip}
\item[{\rm (i)}] We have $|C_{\O}(x)| \leqs |C_{\O}(y)|$, where $y = (\L_1^a,I_e)$ and $a = \sum_j a_j$.
\item[{\rm (ii)}] If $m = i+k$ with $0 \leqs k < i$, then 
\[
|C_{\O}(y)| = \br{e}{m}_q + \br{e}{k}_q\left(\frac{q^{ia}-1}{q^i-1}\right).
\]
\end{itemize}
\end{lem}

\begin{proof}
Part (i) is clear because we can choose $y$ so that there is a natural containment $C_{\O}(x) \subseteq C_{\O}(y)$ of the fixed point sets. Now consider (ii) and observe that $y$ preserves a decomposition
\[
V = W_1 \oplus \cdots \oplus W_a \oplus C_V(y)
\]
where the $W_j$ are isomorphic $i$-dimensional irreducible $\mathbb{F}_q\la y\ra$-modules. Write 
\[
|C_{\O}(y)| = \a+\b,
\]
where $\a$ is the number of $m$-spaces on which $y$ acts trivially (this is simply the number of $m$-spaces in $C_V(y)$). 

Next let $U$ be an $m$-space in $C_{\O}(y)$ on which $y$ acts nontrivially. Since $m=i+k$ with $0 \leqs k <i$, it follows that $y$ preserves a decomposition $U = U_1 \oplus U_2$, acting irreducibly on $U_1$ and trivially on the $k$-space $U_2$. Therefore $\b=\b_1\b_2$,  where $\b_1$ is the number of $m$-spaces in $W = W_1 \oplus \cdots \oplus W_a$ fixed by $z = (\L^a) \in {\rm GL}_{ai}(q) = {\rm GL}(W)$ and $\b_2$ is the number of $k$-spaces in $C_V(y)$. To compute $\b_2$ we may view $z$ as a scalar matrix in an extension field subgroup ${\rm GL}_{a}(q^i) < {\rm GL}_{ai}(q)$. Then each $m$-space in $W$ fixed by $z$ corresponds to a $1$-dimensional subspace of the natural module for ${\rm GL}_{a}(q^i)$. Since $z$ is a scalar in ${\rm GL}_a(q^i)$, it fixes every $1$-space and thus $\b_2$ is the total number of $1$-spaces in $(\mathbb{F}_{q^i})^a$. Therefore
\[
\a = \br{e}{m}_q,\;\; \b_1 = \br{e}{k}_q, \;\; \b_2 = \frac{(q^i)^a-1}{q^i-1}
\]
and this completes the proof of part (ii).
\end{proof}

\begin{lem}\label{l:psl_sub1}
The conclusion to Proposition \ref{p:psl_sub} holds if $H = P_1$.
\end{lem}

\begin{proof}
First observe that $|\O| = (q^n-1)/(q-1)$ and the maximality of $H$ implies that $G \leqs {\rm P\Gamma L}_{n}(q)$. Let $x \in G$ be an element of prime order $r$. In view of \eqref{e:ls91}, we may assume $x \in {\rm PGL}_{n}(q)$ is semisimple or unipotent with $\nu(x) = s$ (see Definition \ref{d:nu}). 

If $r=p$ then $C_{\O}(x)$ is the set of $1$-dimensional subspaces in the $1$-eigenspace $C_V(x)$, so   
\[
|C_{\O}(x)| = \frac{q^{n-s}-1}{q-1},\;\; {\rm fpr}(x) = \frac{q^{n-s}-1}{q^n-1} = \frac{1}{q^s} - \frac{q^s-1}{q^s(q^n-1)}.
\]
From here, it is straightforward to check that ${\rm fpr}(x) > (r+1)^{-1}$ if and only if $s=1$ and $q=p$, which is the case recorded in part (i) of Proposition \ref{p:psl_sub}.

Now suppose $r \ne p$. Let $i \geqs 1$ be minimal such that $r$ divides $q^i-1$ and note that $r+1 \leqs q^i$. First assume $i=1$. If $s=1$ then $x$ is of the form $(\omega, I_{n-1})$ (up to scalars) and 
\[
|C_{\O}(x)| = 1+\frac{q^{n-1}-1}{q-1}, \;\; {\rm fpr}(x) = \frac{1}{q}+\frac{(q-1)^2}{q(q^n-1)}.
\]
This is greater than $(r+1)^{-1}$ if and only if $r = q-1$ (so either $r=2$, or $r$ is a Mersenne prime) and this special case appears in part (ii) of Proposition \ref{p:psl_sub}. Now assume $i=1$ and $s \geqs 2$. If $n=3$ then $x$ is regular, $|C_{\O}(x)| = 3$ and the result follows. For $n \geqs 4$ we observe that $|C_{\O}(x)|$ is maximal when $x$ is of the form $(\omega I_2, I_{n-2})$, whence
\[
|C_{\O}(x)| \leqs \frac{q^2-1}{q-1}+\frac{q^{n-2}-1}{q-1}
\] 
and we deduce that ${\rm fpr}(x) \leqs (r+1)^{-1}$.

Finally, suppose $r \ne p$ and $i \geqs 2$. Here $|C_{\O}(x)|$ is equal to the number of $1$-dimensional subspaces in $C_V(x)$, whence $|C_{\O}(x)|$ is maximal when $x = (\L, I_{n-i})$ in the notation of \cite[Proposition 3.2.1]{BG}. Here the notation indicates that $x$ preserves a decomposition $V = U \oplus W$, acting irreducibly on the $i$-space $U$ and trivially on $W$. This implies that $|C_{\O}(x)|$ is at most $(q^{n-i}-1)/(q-1)$ and the result follows since ${\rm fpr}(x) \leqs q^{-i}$. 
\end{proof}

\begin{lem}\label{l:psl_sub2}
The conclusion to Proposition \ref{p:psl_sub} holds if $H = P_m$ with $2 \leqs m \leqs n/2$. 
\end{lem}

\begin{proof}
Here we identify $\O$ with the set of $m$-dimensional subspaces of $V$ and we note that
\[
|\O| = \frac{|{\rm GL}_{n}(q)|}{q^{m(n-m)}|{\rm GL}_{m}(q)||{\rm GL}_{n-m}(q)|} = \br{n}{m}_q = \prod_{\ell=0}^{m-1}\frac{q^{n-\ell}-1}{q^{m-\ell}-1}.
\]
By applying \cite[Lemma 2.1]{FM} we obtain
\begin{equation}\label{e:om}
q^{m(n-m)} <  |\O| < 2\left(\frac{q}{q-1}\right)q^{m(n-m)}.
\end{equation}
For $m \ne n/2$, notice that the maximality of $H$ implies that $G \leqs {\rm P\Gamma L}_{n}(q)$. 

Let $x \in G$ be an element of prime order $r$. By arguing as in the proof of the previous lemma, we may assume that either
\begin{itemize}\addtolength{\itemsep}{0.2\baselineskip}
\item[{\rm (a)}] $x \in {\rm PGL}_{n}(q)$ is semisimple or unipotent; or
\item[{\rm (b)}] $n$ is even, $m = n/2$ and $x$ is an involutory graph automorphism.
\end{itemize}
By combining \cite[Proposition 3.1]{GK} with \cite[Lemma 3.11(d)]{GK} we observe that
\begin{equation}\label{e:gk0}
{\rm fpr}(x) < 2q^{-m}.
\end{equation}
In particular, if (b) holds then ${\rm fpr}(x) \leqs 1/3$ unless $m=q=2$. But in the latter case we have $G_0 = {\rm L}_{4}(2)$, which is excluded in Proposition \ref{p:psl_sub} (for the record, ${\rm fpr}(x) = 3/7>1/3$ if $x$ is a symplectic-type graph automorphism, otherwise ${\rm fpr}(x) = 3/35$). For the remainder, we may assume $x \in {\rm PGL}_{n}(q)$ is semisimple or unipotent.

If $r=p$ then \eqref{e:gk0} is sufficient unless $m=q=2$. Here $n \geqs 5$ and we claim that $|C_{\O}(x)|$ is maximal when $x = (J_2, J_1^{n-2})$. To see this, suppose $x = (J_2^{\ell},J_1^{n-2\ell})$ with $1 \leqs \ell \leqs n/2$. If $U$ is a $2$-space fixed by $x$, then either $U \subseteq C_V(x)$ or $U = \la u, xu \ra$ with $u \in V \setminus C_V(x)$.
Therefore,
\[
|C_{\O}(x)| = \frac{1}{3}(2^{n-\ell}-1)(2^{n-\ell-1}-1) + \frac{1}{2}\left((2^n-1)-(2^{n-\ell}-1)\right) 
\]
and it is easy to check that this is maximal when $\ell=1$. This justifies the claim and we quickly deduce that ${\rm fpr}(x) \leqs 1/3$ for all $n \geqs 5$. 

So to complete the proof, we may assume $r \ne p$. As usual, let $i \geqs 1$ be minimal such that $r$ divides $q^i-1$. The bound in \eqref{e:gk0} implies that ${\rm fpr}(x) < q^{-1}$, so we may assume $i \geqs 2$. Note that $r$ divides $t = (q^i-1)/(q-1)$. There are two cases to consider.

First assume $i>m$. Here $|C_{\O}(x)|$ is the number of $m$-spaces in $C_V(x)$, whence $x = (\L,I_{n-i})$ has the most fixed points and thus 
\[
|C_{\O}(x)| \leqs \br{n-i}{m}_q < 2\left(\frac{q}{q-1}\right)q^{m(n-i-m)}.
\]
By applying the lower bound on $|\O|$ in \eqref{e:om} we get  
\[
{\rm fpr}(x) < 2\left(\frac{q}{q-1}\right)q^{-mi} \leqs 2\left(\frac{q}{q-1}\right)q^{-2i} \leqs (t+1)^{-1}.
\]

Now suppose $2 \leqs i \leqs m$. Here the bound in \eqref{e:gk0} is sufficient unless $m = i$ and $q \in \{2,3\}$, or $m = i+1$ and $q = 2$. The groups with $n \leqs 6$ can be checked directly with the aid of {\sc Magma}, so we may assume $n \geqs 7$.

First assume $m=i$ and $q \in \{2,3\}$. Now $\dim C_V(x) = n - \ell i$ for some $1 \leqs \ell \leqs \lfloor n/i \rfloor$ and by applying Lemma \ref{l:sub_useful}(i) we see that $|C_{\O}(x)|$ is maximal when $x = (\L^{\ell},I_{n-\ell i})$. Then by part (ii) of Lemma \ref{l:sub_useful} we get
\[
|C_{\O}(x)| = \br{n-\ell i}{i}_q + \frac{q^{i\ell}-1}{q^i-1}.
\]
For $\ell=1$, this implies that
\[
{\rm fpr}(x) < 2\left(\frac{q}{q-1}\right)q^{-i^2} \leqs (t+1)^{-1}
\]
and the result follows. Now assume $\ell \geqs 2$, in which case we compute
\begin{equation}\label{e:rol1}
{\rm fpr}(x) < 4q^{-i^2\ell}+2q^{-i(n+1-\ell-i)}.
\end{equation}
Since $n \geqs 7$, $i \geqs 2$ and $2 \leqs \ell \leqs \lfloor n/i \rfloor$ we deduce that 
\[
4q^{-i(i\ell-1)}+2q^{-i(n-\ell-i)} \leqs 4q^{-6} + 2q^{4-n} \leqs 1
\]
and thus \eqref{e:rol1} implies that ${\rm fpr}(x) < q^{-i}$.

A very similar argument applies when $m = i+1$ and $q = 2$. Once again we can reduce to the case where $x = (\L^{\ell},I_{n-\ell i})$ and Lemma \ref{l:sub_useful} gives 
\[
|C_{\O}(x)| = \br{n-\ell i}{i+1}_2 + \frac{(2^{n-\ell i}-1)(2^{i\ell}-1)}{2^i-1}.
\]
Therefore   
\[
{\rm fpr}(x) < 2^{-i(i+1)\ell+2}+2^{-i(n-i-1)+2} 
\]
and by setting $\ell=1$ we deduce that ${\rm fpr}(x) < 2^{-i} \leqs (r+1)^{-1}$ as required.
\end{proof}

In order to complete the proof of Proposition \ref{p:psl_sub}, we may assume $G \not\leqs {\rm P\Gamma L}_{n}(q)$ and $H$ is of type $P_{m,n-m}$ or ${\rm GL}_{m}(q) \oplus {\rm GL}_{n-m}(q)$ (in both cases we have $m < n/2$). 

\begin{lem}\label{l:psl_sub3}
The conclusion to Proposition \ref{p:psl_sub} holds if $H$ is of type $P_{m,n-m}$ or ${\rm GL}_{m}(q) \oplus {\rm GL}_{n-m}(q)$.
\end{lem}

\begin{proof}
In both cases, we may identify $\O$ with a set of pairs $(U,W)$ of subspaces of $V$, where $\dim U=m$ and $\dim W = n-m$. For $H = P_{m,n-m}$, each pair $(U,W)$ in $\O$ satisfies the condition $U \subset W$, whereas $V = U \oplus W$ when $H$ is of type ${\rm GL}_{m}(q) \oplus {\rm GL}_{n-m}(q)$. Let $x \in G$ be an element of prime order $r$ and note that \eqref{e:gk0} holds (see \cite[Lemma 3.12(a)]{GK}).

By the usual argument, we may assume $x$ is not a field or graph-field automorphism. Next assume $x$ is an involutory graph automorphism. Here \eqref{e:gk0} is sufficient unless $m = 1$, or $m=q=2$. In both cases, it is clear that $|C_{\O}(x)|$ is at most the total number of $m$-spaces in $V$. Therefore, if $m=q=2$ then 
\[
|C_{\O}(x)| \leqs \br{n}{2}_2 = \frac{1}{3}(2^n-1)(2^{n-1}-1)
\]
and we quickly deduce that ${\rm fpr}(x) \leqs 1/3$. For $m=1$ we have $|C_{\O}(x)| \leqs (q^n-1)/(q-1)$ and the same conclusion holds.

To complete the proof, we may assume $x \in {\rm PGL}_{n}(q)$ is semisimple or unipotent. Since $x$ fixes a pair $(U,W) \in \O$ only if it fixes the $m$-space $U$, by applying Lemmas \ref{l:psl_sub1} and \ref{l:psl_sub2} we can immediately reduce to the case where $m=1$ and $x$ is one of the elements arising in parts (i) and (ii) in the statement of Proposition \ref{p:psl_sub}. Since $|C_{\O}(x)|$ is at most the number of $1$-spaces fixed by $x$, we deduce that $|C_{\O}(x)| \leqs 1+(q^{n-1}-1)/(q-1)$ and 
one can check that this bound is sufficient since $r \in \{q-1,q\}$.
\end{proof}

This completes the proof of Proposition \ref{p:psl_sub}.

\subsection{Unitary groups}\label{ss:unit}

Here is our main result for subspace actions of unitary groups. Note that in part (ii), an involutory graph automorphism $x$ of $G_0 = {\rm U}_{4}(q)$ is said to be of \emph{symplectic-type} if $C_{G_0}(x)$ has socle ${\rm PSp}_{4}(q)$.

\begin{prop}\label{p:psu_sub}
Let $G \leqs {\rm Sym}(\O)$ be a finite almost simple primitive permutation group with point stabilizer $H$ and socle $G_0 = {\rm U}_{n}(q)$ with $n \geqs 3$. Assume $H$ is a subspace subgroup of $G$. If $x \in G$ has prime order $r$, then either ${\rm fpr}(x) \leqs (r+1)^{-1}$, or one of the following holds:
\begin{itemize}\addtolength{\itemsep}{0.2\baselineskip}
\item[{\rm (i)}] $H = P_1$, $r=3$, $n$ is odd, $q=2$,  $x = (\omega, I_{n-1})$ and 
\[
{\rm fpr}(x) = \frac{1}{4}+\frac{3}{4(2^n+1)}.
\]
\item[{\rm (ii)}] $H=P_2$, $r=2$, $n=4$, $q \in \{2,3\}$, $x$ is a symplectic-type graph automorphism and ${\rm fpr}(x) = 5/9$ or $5/14$ for $q=2$ or $3$, respectively.
\item[{\rm (iii)}] $H=P_2$, $r=3$, $n=4$, $q = 2$, $x = (\omega I_2, I_2)$ and ${\rm fpr}(x) = 1/3$.
\item[{\rm (iv)}] $H = N_1$, $r=3$, $n$ is even, $q=2$, $x = (\omega, I_{n-1})$ and 
\[
{\rm fpr}(x) = \frac{1}{4} + \frac{3(2^{n-3}+1)}{2^{n-1}(2^n-1)}. 
\]
\end{itemize}
\end{prop}

\begin{lem}\label{l:psu_sub1}
The conclusion to Proposition \ref{p:psu_sub} holds if $H = P_1$.
\end{lem}

\begin{proof}
First observe that we may identify $\O$ with the set of $1$-dimensional totally singular subspaces of $V$ and note that 
\[
|\O| = \frac{|{\rm GU}_{n}(q)|}{q^{2n-3}|{\rm GU}_{n-2}(q)||{\rm GL}_{1}(q^2)|} = \frac{(q^n-(-1)^{n})(q^{n-1} - (-1)^{n-1})}{q^2-1}.
\]
Let $x \in G$ be an element of prime order $r$. 

If $x$ is a field automorphism, then $r$ is odd, $q=q_0^r$ and the usual argument via \eqref{e:ls91} applies. Next suppose $x$ is an involutory graph automorphism. As explained in the proof of \cite[Lemma 3.14]{GK} (see Case D), $|C_{\O}(x)|$ is at most the number of $1$-dimensional subspaces of $(\mathbb{F}_q)^n$, so $|C_{\O}(x)| \leqs (q^n-1)/(q-1)$ and we deduce that 
\[
{\rm fpr}(x) \leqs \frac{(q^n-1)(q+1)}{(q^n-(-1)^{n})(q^{n-1} - (-1)^{n-1})}.
\]
One can check that this expression is at most $1/3$ unless $(n,q) = (3,3)$, in which case a straightforward {\sc Magma} calculation gives ${\rm fpr}(x) = 1/7$. For the remainder of the proof, we may assume $x \in {\rm PGU}_{n}(q)$ is semisimple or unipotent. 

First assume $r=p$ and note that $|C_{\O}(x)|$ coincides with the number of totally singular $1$-spaces in $C_V(x)$. We claim that ${\rm fpr}(x) \leqs (r+1)^{-1}$. If $x \ne (J_2,J_1^{n-2})$ then $\dim C_V(x) \leqs n-2$ and one checks that the bound
\[
|C_{\O}(x)| \leqs \br{n-2}{1}_{q^2} = \frac{q^{2n-4}-1}{q^2-1}
\]
yields ${\rm fpr}(x) \leqs (q+1)^{-1}$. Now assume $x = (J_2, J_1^{n-2})$. Here $x$ preserves an orthogonal decomposition $V = U \perp W$, where $x$ has Jordan form $(J_2)$ on the nondegenerate $2$-space $U$. Setting $C_U(x) = \la u \ra$, which is totally singular, we have  
\[
C_{\O}(x) = \{ \la u \ra, \la \l u + w \ra \,:\, \mbox{$\l \in \mathbb{F}_{q^2}$ and $\la w \ra \subseteq W$ is totally singular} \}.
\]
Therefore, $|C_{\O}(x)| = \a q^2 + 1$, where 
\[
\a = \frac{(q^{n-2} - (-1)^{n-2})(q^{n-3} - (-1)^{n-3})}{q^2-1}
\]
is the number of totally singular $1$-spaces in $W$ (this can also be computed via \cite[Lemma 2.13(2)]{FM}, noting that $\la u \ra$ is the radical of $C_V(x)$). It is now a routine exercise to check that ${\rm fpr}(x) \leqs (q+1)^{-1}$ for all $n$ and $q$.

Finally, suppose $r \ne p$ and let $i \geqs 1$ be minimal such that $r$ divides $q^i-1$. We will adopt the notation for semisimple elements given in \cite[Proposition 3.3.2]{BG}.

First assume $i=2$, so $r$ is odd, $r$ divides $q+1$ and $|C_{\O}(x)|$ is equal to the total number of totally singular $1$-spaces in each eigenspace of $\hat{x}$ on $V$ (where $x$ is the image of $\hat{x} \in {\rm GU}_{n}(q)$ modulo scalars). Since every eigenspace of $\hat{x}$ is nondegenerate, it follows that $|C_{\O}(x)|$ is maximal when $x$ is of the form $(\omega,I_{n-1})$, in which case
\begin{equation}\label{e:unit}
{\rm fpr}(x) = \frac{q^{n-2} - (-1)^{n}}{q^n-(-1)^n}.
\end{equation}
If $n$ is even, then it is straightforward to check that ${\rm fpr}(x) \leqs (q+2)^{-1}$ and the result follows. On the other hand, if $n$ is odd then the same conclusion holds if and only if $q \geqs 3$. Indeed, if $q=2$ then $r=3$ and 
\[
{\rm fpr}(x) = \frac{1}{4} + \frac{3}{4(2^n+1)},
\]
which corresponds to the special case recorded in part (i) of Proposition \ref{p:psu_sub}. If $n$ is odd, $q=2$ and $x$ is any other element of order $3$, then $|C_{\O}(x)|$ is maximal when $x = (\omega I_2, I_{n-2})$ and thus    
\[
|C_{\O}(x)| \leqs \frac{(2^2-1)(2+1)}{2^2-1} + \frac{(2^{n-2}+1)(2^{n-3}-1)}{2^2-1}.
\]
It is easy to check that ${\rm fpr}(x) \leqs 1/4$.

Next assume $i \equiv 2 \imod{4}$ and $i \geqs 6$. Note that $r$ divides $t = (q^{i/2}+1)/(q+1)$ and $|C_{\O}(x)|$ coincides with the number of totally singular $1$-spaces in $C_V(x)$. It follows that $|C_{\O}(x)|$ is maximal when $x = (\L,I_{n-i/2})$ and it is easy to check that ${\rm fpr}(x) \leqs (t+1)^{-1}$ as required.

Now suppose $i \equiv 0 \imod{4}$, so $r$ divides $q^{i/2}+1$. Once again, $|C_{\O}(x)|$ is the number of totally singular $1$-spaces in $C_V(x)$ and thus elements of the form $(\L,\L^{-q},I_{n-i})$ fix the most points. Given this observation, it is straightforward  to verify the bound ${\rm fpr}(x) \leqs (q^{i/2}+2)^{-1}$.

Finally, suppose $i$ is odd. First assume $r=2$, so $q$ is odd and $i=1$. Here $|C_{\O}(x)|$ is maximal when $x$ is of the form $(-I_1, I_{n-1})$, in which case \eqref{e:unit} holds and it is easy to check that ${\rm fpr}(x) \leqs 1/3$. Next assume $r$ is odd and $i=1$, so $|C_{\O}(x)|$ is maximal when $x$ is of the form $((\L,\L^{-q})^{\ell},I_{n-2\ell})$ for some $\ell \geqs 1$. Here the $1$-eigenspace is nondegenerate, while the other two eigenspaces are totally singular, whence
\[
|C_{\O}(x)| = 2\left(\frac{q^{2\ell}-1}{q^2-1}\right) + \frac{(q^{n-2\ell}-(-1)^n)(q^{n-2\ell-1}-(-1)^{n-1})}{q^2-1}.
\]

If $n=4$ then $|C_{\O}(x)|$ is maximal when $\ell=2$, in which case $|C_{\O}(x)| = 2(q^2+1)$ and ${\rm fpr}(x) = 2(q^3+1)^{-1}$. For $n \geqs 5$ one can check that $|C_{\O}(x)|$ is maximal when $\ell=1$ and it is plain to see that the same conclusion holds when $i \geqs 3$ (since in this case, $|C_{\O}(x)|$ is just the number of totally singular $1$-spaces in $C_V(x)$). Here $x = (\L,\L^{-q},I_{n-2i})$ preserves a decomposition $V = (U_1 \oplus U_2) \perp W$, where $U_1$ and $U_2$ are totally singular $i$-spaces and $W=C_V(x)$ is nondegenerate (or trivial). Moreover, $|C_{\O}(x)| = \a+\b$, where $\a$ is the number of totally singular $1$-spaces in $W$ and we set $\b=2$ if $i=1$ (since $x$ also fixes the totally singular $1$-spaces $U_1$ and $U_2$), otherwise $\b=0$. It is now straightforward to verify the bound ${\rm fpr}(x) \leqs q^{-i}$ and the result follows.
\end{proof}

\begin{lem}\label{l:psu_sub2}
The conclusion to Proposition \ref{p:psu_sub} holds if $H = P_m$ with $2 \leqs m \leqs n/2$.
\end{lem}

\begin{proof}
Identify $\O$ with the set of totally singular $m$-dimensional subspaces of $V$ and note that
\[
|\O| = \frac{|{\rm GU}_{n}(q)|}{q^{m(2n-3m)}|{\rm GU}_{n-2m}(q)||{\rm GL}_{m}(q^2)|}.
\]
By applying \cite[Proposition 3.9]{Bur2}, we get
\[
\frac{1}{2}q^{m(2n-3m)} < |\O| < 2\left(\frac{q+1}{q}\right)q^{m(2n-3m)}.
\]
Suppose $x \in G$ has prime order $r$. If $n \geqs 6$ then \cite[Proposition 3.15]{GK}
gives
\begin{equation}\label{e:gk}
{\rm fpr}(x) < q^{-(n-1)/2} + 2q^{-(n-2)} + q^{-2m}.
\end{equation}

As usual, the desired bound holds if $x$ is a field automorphism. Next suppose $x$ is an involutory graph automorphism, so $r=2$ and we may assume $q \leqs 3$ in view of Theorem \ref{t:ls91}. The groups with $n \leqs 5$ can be checked using {\sc Magma}, noting the two special cases that arise when $n=4$, $H=P_2$ and $x$ is a symplectic-type graph automorphism (see part (ii) in the statement of Proposition \ref{p:psu_sub}). 
On the other hand, if $n \geqs 6$ then the bound in \eqref{e:gk} is sufficient unless $(n,q) = (6,2)$, which we can handle using {\sc Magma} (we get ${\rm fpr}(x) \leqs 5/33$, with equality if $m=3$ and $C_{G_0}(x) = {\rm Sp}_6(2)$). For the remainder, we may assume $x \in {\rm PGU}_{n}(q)$ is semisimple or unipotent.

First assume $r=p$. If $n \geqs 6$ then the bound in \eqref{e:gk} gives ${\rm fpr}(x) \leqs (q+1)^{-1}$ unless $(n,q) = (6,2)$, which we can check using {\sc Magma}. Now assume $n \in \{4,5\}$ and $m=2$. If $p=2$ and $q \geqs 4$ then Theorem \ref{t:ls91} yields ${\rm fpr}(x) \leqs 1/3$ as required, while the case $q=2$ can be handled using {\sc Magma}. Therefore, we may assume $q$ is odd. 

For $n=4$, we claim that $|C_{\O}(x)| \leqs q^2+q+1$, which immediately implies that 
${\rm fpr}(x) \leqs (q+1)^{-1}$. To see this, first observe that $G_0 \cong {\rm P\O}_6^{-}(q)$ and the action of $G_0$ on $\O$ is permutation isomorphic to the action of ${\rm P\O}_6^{-}(q)$ on the set $\Gamma$ of totally singular $1$-dimensional subspaces of the $6$-dimensional orthogonal module $W$. The effect of this isomorphism on unipotent elements is as follows:
\[
(J_2,J_1^2) \mapsto (J_2^2,J_1^2),\;\; (J_2^2) \mapsto (J_3,J_1^3),\;\; (J_3,J_1) \mapsto (J_3^2),\;\; (J_4) \mapsto (J_5,J_1).
\]
In particular, if $x \in G_0$ is sent to $y \in {\rm P\O}_{6}^{-}(q)$, then $|C_{\O}(x)| = |C_{\Gamma}(y)|$ is equal to the number of totally singular $1$-spaces in $C_W(y)$. If $\dim C_W(y) \leqs 2$, then we deduce that $|C_{\O}(x)| \leqs q+1$. In the remaining two cases, we can appeal to the analysis of unipotent elements in the proof of Lemma \ref{l:orth2_sub1}, which allows us to conclude that $|C_{\O}(x)| = q+1$ if $x=(J_2,J_1^2)$ and $|C_{\O}(x)| = q^2+q+1$ when $x = (J_2^2)$. This justifies the claim.

Next assume $n=5$, so $|\O| = (q^5+1)(q^3+1)$. First observe that $|C_{\O}(x)|=1$ when $x = (J_3,J_2)$, $(J_4,J_1)$ or $(J_5)$. Next let $x = (J_2,J_1^3)$. Here $x$  preserves an orthogonal decomposition $V = V_1 \perp V_2$, where $\dim V_1 = 2$ and $C_{V_1}(x) = \la u \ra$ is totally singular. Then every space in $C_{\O}(x)$ is of the form $\la u, w\ra$, where $\la w \ra$ is a totally singular $1$-space in the nondegenerate $3$-space $V_2$, and thus $|C_{\O}(x)| = q^3+1$. 

Suppose $x = (J_3,J_1^2)$. Here we claim that 
\[
|C_{\O}(x)| \leqs (q^2-1)^2+q+1  
\]
which implies that ${\rm fpr}(x) \leqs 1/(q+1)$. To see this, write $V = V_1 \perp V_2$ as an orthogonal decomposition into nondegenerate spaces, where $x$ has Jordan form $(J_3,J_1)$ on $V_1$ and $V_2 = \la v \ra$ is centralized by $x$. If $(\, , \,)$ is the defining unitary form on $V$, then we may assume $(v,v)=1$. Let $U$ be a totally singular $2$-space fixed by $x$. By our above analysis of the case $n=4$, we see that there are at most $q+1$ such spaces contained in $V_1$. Now assume $U \cap V_1 = \la u \ra$ is $1$-dimensional, so $U = \la u, w+v\ra$ for some $w \in \la u \ra^{\perp} \cap V_1$ with $(w,w) = -1$. Here $\la u \ra$ has to be the radical of the $2$-space  $C_{V_1}(x)$ and we calculate that there are $(q^2-1)^2$ nondegenerate $1$-spaces in the $3$-space $\la u \ra^{\perp} \cap V_1$. This justifies the claim.

Finally, suppose $x = (J_2^2,J_1)$. Here the claim is
\[
|C_{\O}(x)| \leqs (q^4-1)(q^2-1)+q^2+q+1,
\]
which once again is sufficient. Write $V = V_1 \perp V_2$, where $x$ has Jordan form $(J_2^2)$ on $V_1$ and $x$ centralizes $V_2 = \la v \ra$. From our earlier work in the case $n=4$, we see that $x$ fixes $q^2+q+1$ totally singular $2$-spaces in $V_1$. Now suppose $U = \la u, w+v\ra$ is a totally singular $2$-space fixed by $x$. Here $\la u\ra$ is contained in the totally singular $2$-space $C_{V_1}(x)$, so there are $q^2+1$ choices for $\la u \ra$. As before, there are $(q^2-1)^2$ nondegenerate $1$-spaces in $\la u \ra^{\perp} \cap V_1$, whence there are at most $(q^2+1)(q^2-1)^2$ totally singular $2$-spaces fixed by $x$ that are not contained in $V_1$. The result follows.

Now assume $r \ne p$ and let $i \geqs 1$ be minimal such that $r$ divides $q^i-1$. First assume $r=2$. For $q \geqs 5$, the bound in Theorem \ref{t:ls91} is clearly sufficient and so we may assume $q=3$. If $n \geqs 6$ then the bound in \eqref{e:gk} is effective, while the remaining cases with $n \in \{4,5\}$ and $m=2$ can be checked using {\sc Magma}. Now assume $r$ is odd. There are several cases to consider.

\vs

\noindent \emph{Case 1. $i \equiv 2 \imod{4}$.}

\vs

First assume $i=2$, so $r$ divides $q+1$. If $n \geqs 6$ then the bound in \eqref{e:gk} is sufficient unless $(n,q) = (6,2)$, which we can handle directly using {\sc Magma}. Now assume $n \in \{4,5\}$ and $m=2$. For $n=4$, it is easy to check that  $|C_{\O}(x)|$ is maximal when $x = (\omega I_2, I_2)$, whereas $x = (\omega, I_4)$ has the most fixed points when $n=5$. Therefore, if $n=4$ we have $|C_{\O}(x)| \leqs (q+1)^2$ since every nondegenerate $2$-space contains $q+1$ totally singular $1$-spaces. Similarly, $|C_{\O}(x)|$ is at most $(q^3+1)(q+1)$ when $n=5$, which is the number of totally singular $2$-spaces in a nondegenerate $4$-space. One can check that these bounds yield ${\rm fpr}(x) \leqs (q+2)^{-1}$ unless $(n,q) = (4,2)$. Here $r=3$ and ${\rm fpr}(x) = 1/3$; this special case is recorded in part (iii) of Proposition \ref{p:psu_sub}.

Now assume $i \geqs 6$. Note that $r$ divides $t = (q^{i/2}+1)/(q+1)$, so it suffices to show that ${\rm fpr}(x) \leqs (t+1)^{-1}$. If $m \geqs i/2$ then $n \geqs i$ and the bound in \eqref{e:gk} implies that 
\[
{\rm fpr}(x) < q^{-(i-1)/2}+2q^{-(i-2)}+q^{-i}.
\]
One can check that this yields ${\rm fpr}(x) \leqs (t+1)^{-1}$ unless $(i,q) = (6,2)$. But $2^6-1$ does not have a primitive prime divisor, so the case $(i,q) = (6,2)$ does not arise. Now assume $m < i/2$. Here $|C_{\O}(x)|$ is the number of totally singular $m$-spaces in the nondegenerate $1$-eigenspace $C_V(x)$, which is maximal when $x = (\L, I_{n-i/2})$. Therefore,
\[
|C_{\O}(x)| < 2\left(\frac{q+1}{q}\right)q^{m(2(n-i/2)-3m)} = 2\left(\frac{q+1}{q}\right)q^{m(2n-3m)}\cdot q^{-mi}
\]
and the result follows since 
\[
{\rm fpr}(x) < 4\left(\frac{q+1}{q}\right)q^{-mi} \leqs (t+1)^{-1}.
\]

\vs

\noindent \emph{Case 2. $i \equiv 0 \imod{4}$.}

\vs

Now suppose $i \equiv 0 \imod{4}$, so $n \geqs i$ and $r$ divides $q^{i/2}+1$. First assume $n \in \{i,i+1\}$, in which case every element of order $r$ is of the form $(\L,\L^{-q},I_{n-i})$ and we see that $|C_{\O}(x)| = 2$ if $m = i/2$, otherwise $|C_{\O}(x)| = 0$. Therefore, ${\rm fpr}(x) < 4q^{-n^2/4}$, which in turn implies that ${\rm fpr}(x) < (q^{i/2}+2)^{-1}$ unless $(n,q) = (4,2)$. In this special case, we have $(m,r)=(2,5)$ and we compute ${\rm fpr}(x) = 2/27$. For the remainder, we may assume $n \geqs i+2$. 

If $m \geqs i/2$ then the bound in \eqref{e:gk} is sufficient unless $q=2$ and $(n,i) = (6,4)$, $(7,4)$, $(8,4)$ or $(10,8)$. Each of these cases can be handled directly. For example, suppose $(n,q,i) = (10,2,8)$, so $r=17$, $m \in\{4,5\}$ and $x = (\L,\L^{-2},I_2)$ preserves an orthogonal decomposition $V = (U_1 \oplus U_2) \perp W$, where $U_1$ and $U_2$ are totally singular $4$-spaces and $W = C_V(x)$. If $m=4$ then $C_{\O}(x) = \{U_1, U_2\}$ and the result follows. Similarly, if $m = 5$ then each space in $C_{\O}(x)$ is of the form $U_j \oplus \la w \ra$, where $\la w \ra \subseteq W$ is totally singular. Since $W$ contains $q+1 = 3$ totally singular $1$-spaces, it follows that $|C_{\O}(x)|=6$ and once again the desired bound holds. The other cases can be handled in a similar fashion (either by hand or via {\sc Magma}).   

Now assume $m< i/2$. Here $|C_{\O}(x)|$ is the number of totally singular $m$-spaces in $C_V(x)$, whence $x = (\L,\L^{-q},I_{n-i})$ has the most fixed points. Working with this element, we compute 
\[
{\rm fpr}(x) < 4\left(\frac{q+1}{q}\right)q^{-2mi} \leqs (q^{i/2}+2)^{-1}
\]
and the result follows.

\vs

\noindent \emph{Case 3. $i$ odd.}

\vs

Finally, let us assume $i$ is odd. Suppose $i=1$, so $q \geqs 4$ (since $r$ is odd) and $r+1 \leqs q$. If $n \geqs 6$ then the bound in \eqref{e:gk} is sufficient, so we may assume $n \in \{4,5\}$ and $m=2$. We claim that 
\[
|C_{\O}(x)| \leqs \left\{\begin{array}{ll}
2(q^3+1) & \mbox{if $n=5$} \\
q^2+3 & \mbox{if $n=4$}
\end{array}\right.
\]
To see this, first observe that $|C_{\O}(x)|=4$ if $x$ is regular, so we may assume $x = (\L,\L^{-q},I_{n-2})$ or $((\L,\L^{-q})^2,I_{n-4})$. Now $x = (\L,\L^{-q},I_{n-2})$ preserves a decomposition $V = (U_1 \oplus U_2) \perp W$ with $W = C_V(x)$ and  each totally singular $2$-space fixed by $x$ is of the form $U_j \oplus \la w \ra$, where $\la w \ra$ is a totally singular $1$-space in $W = C_V(x)$. Therefore, $|C_{\O}(x)|$ is $2(q^3+1)$ if $n=5$ and $2(q+1)$ if $n=4$. Similarly, $x = ((\L,\L^{-q})^2,I_{n-4})$ preserves a decomposition $V = (U_1 \oplus U_2) \perp W$, where $W=C_V(x)$ and $x$ acts as a scalar on the totally singular $2$-spaces $U_1$ and $U_2$. Therefore, the totally singular $2$-spaces fixed by $x$ are $U_1$, $U_2$ and $\la u \ra \oplus \la u' \ra$, where $\la u \ra \subseteq U_1$ is an arbitrary $1$-space and $\la u' \ra = U_2 \cap \la u \ra^{\perp}$. This implies that $|C_{\O}(x)| = q^2+3$, which is the total number of subspaces of $U_1$. This justifies the claim and it is easy to check that ${\rm fpr}(x) \leqs q^{-1}$.

Now assume $i \geqs 3$, so $n \geqs 2i \geqs 6$ and $r$ divides $t = (q^i-1)/(q-1)$. If $n \in \{2i,2i+1\}$ then $x = (\L,\L^{-q},I_{n-2i})$ is the only possibility and it is easy to check that ${\rm fpr}(x) \leqs (t+1)^{-1}$ since $|C_{\O}(x)| = 2$ if $m=i$, otherwise $|C_{\O}(x)| = 0$. Now assume $n \geqs 2i+2$. If $m \geqs i$ then \eqref{e:gk} is sufficient unless $(n,q,i) = (8,2,3)$. Here $m \in \{3,4\}$, $r=7$ and the result follows since $|C_{\O}(x)| \leqs 6$. Finally, suppose $m < i$. In this case, $|C_{\O}(x)|$ is the number of totally singular $m$-spaces in $C_V(x)$ and we deduce that  
\[
{\rm fpr}(x) < 2\left(\frac{q+1}{q}\right)q^{-4mi} \leqs (t+1)^{-1}
\]
as required.
\end{proof}

\begin{lem}\label{l:psu_sub3}
The conclusion to Proposition \ref{p:psu_sub} holds if $H = N_1$.
\end{lem}

\begin{proof}
We may identify $\O$ with the set of nondegenerate $1$-dimensional subspaces of $V$ and we note that
\[
|\O| = \frac{|{\rm GU}_{n}(q)|}{|{\rm GU}_{n-1}(q)||{\rm GU}_{1}(q)|} = \frac{q^{n-1}(q^n-(-1)^n)}{q+1}.
\]
Let $x \in G$ be an element of prime order $r$. If $x$ is a field automorphism, then the usual argument applies. Next suppose $x$ is an involutory graph automorphism. By embedding ${\rm GU}_n(q)$ in ${\rm GL}_{n}(q^2)$, we may view $x$ as an involutory field automorphism of ${\rm GL}_{n}(q^2)$ and thus $|C_{\O}(x)|$ is at most the number of $1$-dimensional subspaces of $V$ that are defined over $\mathbb{F}_q$. In other words, $|C_{\O}(x)| \leqs (q^n-1)/(q-1)$ and the result follows unless $(n,q) = (4,2)$. In the latter case, we have $|\O| = 40$ and using {\sc Magma} we calculate that ${\rm fpr}(x) \leqs 1/5$.

To complete the argument, we may assume $x \in {\rm PGU}_{n}(q)$ is semisimple or unipotent. If $r=p$ then $|C_{\O}(x)|$ is the number of nondegenerate $1$-spaces in $C_V(x)$. In particular, if $\dim C_V(x) \leqs n-2$ then $|C_{\O}(x)|$ is at most 
$(q^{2n-4}-1)/(q^2-1)$, which is the total number of $1$-spaces in an $(n-2)$-dimensional vector space over $\mathbb{F}_{q^2}$, and we immediately deduce that ${\rm fpr}(x) \leqs (q+1)^{-1}$. Now assume $\dim C_V(x) = n-1$, so $x = (J_2,J_1^{n-2})$ preserves a decomposition $V = U \perp W$ into nondegenerate spaces, where $\dim U = 2$ and $C_V(x) = \la u \ra \oplus W$ with $\la u \ra$ totally singular. Then every nondegenerate $1$-space in $C_V(x)$ is of the form $\la \l u + w\ra$, where $\l \in \mathbb{F}_{q^2}$ and $\la w \ra$ is a nondegenerate $1$-space in $W$. Therefore, 
\[
|C_{\O}(x)| = q^2\left(\frac{q^{n-3}(q^{n-2}-(-1)^{n-2})}{q+1}\right)
\]
and once again it is straightforward to check that ${\rm fpr}(x) \leqs (q+1)^{-1}$.

For the remainder, let us assume $r \ne p$ and $x$ is semisimple. Let $i \geqs 1$ be minimal such that $r$ divides $q^i-1$ and set 
\begin{equation}\label{e:j}
j = \left\{\begin{array}{ll}
i/2 & \mbox{if $i \equiv 2 \imod{4}$} \\
i & \mbox{if $i \equiv 0 \imod{4}$} \\
2i & \mbox{if $i$ is odd.} 
\end{array}\right.
\end{equation}

First assume $i=2$, so $r$ divides $q+1$. Here the eigenspaces of $\hat{x}$ on $V$ are nondegenerate (where $x$ is the image of $\hat{x} \in {\rm GU}_{n}(q)$ modulo scalars) and $|C_{\O}(x)|$ is the total number of nondegenerate $1$-spaces in each eigenspace. As a consequence, we quickly deduce that $|C_{\O}(x)|$ is maximal when $x = (\omega, I_{n-1})$, in which case 
\begin{equation}\label{e:u0}
{\rm fpr}(x) = \frac{q^{n-2}(q^{n-1}-(-1)^{n-1})+q+1}{q^{n-1}(q^n-(-1)^n)}
\end{equation}
and one can check that this is at most $(q+2)^{-1}$ unless $n$ is even and $q=2$. In this special case we have $r=3$ and 
\[
{\rm fpr}(x) = \frac{1}{4} + \frac{3(2^{n-3}+1)}{2^{n-1}(2^n-1)},
\]
so this is a genuine exception and it is recorded in part (iv) of Proposition \ref{p:psu_sub}. Let us also observe that if $n$ is even, $q=2$, $r=3$ and $\nu(x) \geqs 2$, then 
\[
|C_{\O}(x)| \leqs \frac{q(q^2-1)+ q^{n-3}(q^{n-2}-1)}{q+1}
\]
(maximal if $x = (\omega I_2, I_{n-2})$) and it is easy to verify the bound ${\rm fpr}(x) \leqs 1/4$.

Now assume $i \ne 2$. If $r=2$ then $|C_{\O}(x)|$ is maximal when $x = (-I_{1}, I_{n-1})$, in which case \eqref{e:u0} holds and we deduce that ${\rm fpr}(x) \leqs 1/3$. Now assume $r$ is odd. Here $|C_{\O}(x)|$ is the number of nondegenerate $1$-spaces in the $1$-eigenspace $C_V(x)$ and we deduce that 
\begin{equation}\label{e:j2}
|C_{\O}(x)| \leqs \frac{q^{n-j-1}(q^{n-j}-(-1)^{n-j})}{q+1}.
\end{equation}
For example, suppose $i \equiv 2 \imod{4}$. Here $j = i/2 \geqs 3$, $r$ divides $t=(q^{j}+1)/(q+1)$ and $|C_{\O}(x)|$ is maximal when $x = (\L,I_{n-j})$, which gives the upper bound in \eqref{e:j2}. Moreover, it is straightforward to show that ${\rm fpr}(x) \leqs (t+1)^{-1}$. A very similar argument applies when $i \not\equiv 2 \imod{4}$ and we omit the details.
\end{proof}

Finally, we complete the proof of Proposition \ref{p:psu_sub} by handling the case where $H$ is the stabilizer of a nondegenerate $m$-space with $m \geqs 2$.

\begin{lem}\label{l:psu_sub4}
The conclusion to Proposition \ref{p:psu_sub} holds if $H = N_m$ with $2 \leqs m < n/2$.
\end{lem}

\begin{proof}
Let us identify $\O$ with the set of nondegenerate $m$-dimensional subspaces of $V$ 
and observe that $n \geqs 5$ and 
\begin{equation}\label{e:unim2}
\left(\frac{q-1}{q}\right)q^{2m(n-m)} < |\O| = \frac{|{\rm GU}_{n}(q)|}{|{\rm GU}_{m}(q)||{\rm GU}_{n-m}(q)|} < q^{2m(n-m)}
\end{equation}
(see Section 2 in \cite{FM}, for example). Let $x \in G$ be an element of prime order $r$. If $n \geqs 6$ then \cite[Proposition 3.16]{GK} gives
\begin{equation}\label{e:gk00}
{\rm fpr}(x) < 2q^{-(n-4)}+q^{-(n-1)}+q^{-2\lfloor m/2 \rfloor} + q^{-2(n-m)}.
\end{equation}
For integers $a \geqs b$, it will be convenient to write $f(a,b)$ for the number of nondegenerate $b$-spaces in an $a$-dimensional unitary space over $\mathbb{F}_{q^2}$. In particular, $f(a,b) \leqs q^{2b(a-b)}$ and $|\O| = f(n,m)$.

For an involutory graph automorphism $x$ we have $|C_{\O}(x)| \leqs \br{n}{m}_q$ and by combining the relevant bounds in \eqref{e:om} and \eqref{e:unim2} we deduce that
\[
{\rm fpr}(x) < 2\left(\frac{q}{q-1}\right)^2q^{-m(n-m)} \leqs \frac{1}{3}.
\]
Field automorphisms can be handled in the usual way, so for the remainder we may assume $x \in {\rm PGU}_{n}(q)$ is semisimple or unipotent.

First assume $r=p$. If $n \geqs 6$ then the bound in \eqref{e:gk00} is sufficient unless $q=2$ and $n \in \{6,7,8\}$, or $(n,q) = (6,3)$. All of these cases can be checked using {\sc Magma}. Now assume $n=5$, so $m=2$. Here $|C_{\O}(x)|$ is maximal when $x = (J_2,J_1^3)$ and we observe that $|x^G \cap H| = \a(2,q)+\a(3,q)$ and $|x^G| = \a(5,q)$, where 
\[
\a(d,q) = \frac{|{\rm GU}_{d}(q)|}{q^{2d-3}|{\rm GU}_{d-2}(q)||{\rm GU}_{1}(q)|}
\]
is the number of transvections in ${\rm GU}_{d}(q)$. Since ${\rm fpr}(x) = |x^G \cap H|/|x^G|$, it is straightforward to check that ${\rm fpr}(x) \leqs (q+1)^{-1}$ as required.

For the remainder, let us assume $x$ is semisimple. As before, let $i \geqs 1$ be minimal such that $r$ divides $q^i-1$ and define $j$ as in \eqref{e:j}.

Suppose $i \equiv 2 \imod{4}$. We may write $x = (x_1,x_2)$ in terms of an orthogonal decomposition $V = V_1 \perp V_2$ into nondegenerate spaces, where $\la x \ra$ acts homogeneously (and nontrivially) on $V_1$ with the additional property that $V_1$ and $V_2$ have no common $\la x \ra$-irreducible constituent. If $U$ is a nondegenerate $m$-space fixed by $x$, then $U = U_1 \perp U_2$ and $U_i \subseteq V_i$ is fixed by $x_i$. In particular, $U$ is fixed by $y = (x_1,1)$ and thus $C_{\O}(x) \subseteq C_{\O}(y)$. Therefore, in the notation of \cite[Proposition 3.3.2]{BG}, we may assume that $x = (\L^{\ell},I_{n-j\ell})$ for some $\ell \geqs 1$. Similar reasoning shows that we may assume $x = ((\L,\L^{-q})^{\ell},I_{n-j\ell})$ when $i \not\equiv 2 \imod{4}$. 

With this observation in hand, let us begin the main analysis by considering the case  $i=2$. If $m \geqs 4$ then the upper bound in \eqref{e:gk00} is sufficient, so we may assume $m \in \{2,3\}$ and $x = (\omega I_{\ell}, I_{n-\ell})$. In terms of the $f(a,b)$ notation introduced above, we have
\[
|C_{\O}(x)| = \sum_{k=0}^{m} f(\ell,m-k)\cdot f(n-\ell,k)
\]
and it is easy to check that ${\rm fpr}(x) \leqs (q+2)^{-1}$. For example, if $m=2$ then we get
\[
|C_{\O}(x)| < q^{4\ell-8}+q^{2n-4}+q^{4n-4\ell-8},
\]
which is sufficient when combined with the lower bound on $|\O|$ in \eqref{e:unim2}.

Next assume $i \equiv 2 \imod{4}$ and $i \geqs 6$, in which case $r$ divides $t = (q^{i/2}+1)/(q+1)$. Suppose $m \geqs i/2$. Here \eqref{e:gk00} is sufficient unless $i=6$, $m=3$ and either $n=7$ or $(n,q) = (8,3)$ (note that $q \geqs 3$ if $i=6$). These cases can be handled directly. For example, if $(n,m,i) = (7,3,6)$ then we may assume $x = (\L,I_4)$ or $(\L^2,I_1)$. In the latter case, we compute
\[
|x^{G_0} \cap H| = \frac{|{\rm GU}_{3}(q)|}{|{\rm GU}_{1}(q^3)|} \cdot \frac{|{\rm GU}_{4}(q)|}{|{\rm GU}_{1}(q^3)||{\rm GU}_{1}(q)|} < 2q^{18}
\]
and thus ${\rm fpr}(x) < 4q^{-18}$ since $|x^G|>\frac{1}{2}q^{36}$. Finally, if $m < i/2$ then $|C_{\O}(x)|$ is equal to the number of nondegenerate $m$-spaces in $C_V(x)$, which implies that $x = (\L,I_{n-i/2})$ has the most fixed points. Therefore $|C_{\O}(x)|  \leqs f(n-i/2,m)$ and we quickly deduce that ${\rm fpr}(x) \leqs (t+1)^{-1}$ as required.

Now suppose $i \equiv 0 \imod{4}$, so $r$ divides $q^{i/2}+1$. If $m \geqs i$ then the bound in \eqref{e:gk00} is sufficient. On the other hand, if $m < i$ then we may assume $x = (\L,\L^{-q},I_{n-i})$, in which case $|C_{\O}(x)| = f(n-i,m)$ and it is easy to check that ${\rm fpr}(x) \leqs (q^{i/2}+2)^{-1}$. 

Finally, suppose $i$ is odd. If $i=1$ then $q \geqs 3$ and for $n \geqs 6$ one can check that the bound in \eqref{e:gk00} is sufficient unless $(n,m,q) = (6,2,3)$. In the latter case, $r=2$, $|C_{\O}(x)|$ is maximal when $x = (-I_1,I_5)$ and we compute ${\rm fpr}(x) \leqs 1/81$. Similarly, if $i=1$ and $n=5$ then $m=2$ and $|C_{\O}(x)|$ is maximal when $x = (-I_1,I_4)$, in which case 
$|C_{\O}(x)| = f(4,1)+f(4,2)$ and we obtain ${\rm fpr}(x) \leqs q^{-1}$. Finally, suppose $i \geqs 3$ and note that $r$ divides $t = (q^i-1)/(q-1)$. If $m \geqs 2i$ then it is easy to check that \eqref{e:gk00} is sufficient. For $m<2i$ we observe that $|C_{\O}(x)|$ is maximal when $x = (\L,\L^{-q},I_{n-2i})$. Here $|C_{\O}(x)| = f(n-2i,m)$ and it is straightforward to verify the bound ${\rm fpr}(x) \leqs (t+1)^{-1}$.
\end{proof}

This completes the proof of Proposition \ref{p:psu_sub}.

\subsection{Symplectic groups}\label{ss:symp}

Next we turn to the subspace actions of almost simple symplectic groups. Throughout this section, we assume $G_0 \ne {\rm PSp}_{4}(2)', {\rm PSp}_{4}(3)$ since ${\rm PSp}_{4}(2)' \cong A_6$ and ${\rm PSp}_{4}(3) \cong {\rm U}_{4}(2)$. 

\begin{prop}\label{p:symp_sub}
Let $G \leqs {\rm Sym}(\O)$ be a finite almost simple primitive permutation group with point stabilizer $H$ and socle $G_0 = {\rm PSp}_{n}(q)$ with $n \geqs 4$ and $(n,q) \ne (4,2)$, $(4,3)$. Assume $H$ is a subspace subgroup of $G$ and let $x \in G$ be an element of prime order $r$. Then either ${\rm fpr}(x) \leqs (r+1)^{-1}$, or one of the following holds:
\begin{itemize}\addtolength{\itemsep}{0.2\baselineskip}
\item[{\rm (i)}] $H = P_1$, $r=q=p$, $x = (J_2,J_1^{n-2})$ and 
\[
{\rm fpr}(x) = \frac{1}{q+1} + \frac{q(q^{n-2}-1)}{(q+1)(q^n-1)}.
\]
\item[{\rm (ii)}] $H$ is of type ${\rm O}_{n}^{\e}(q)$, $r=q=2$, $x = (J_2,J_1^{n-2})$ and 
\[
{\rm fpr}(x) = \frac{1}{3}+\frac{2^{n/2-1}-\e}{3(2^{n/2}+\e)}.
\] 
\item[{\rm (iii)}] $H$ is of type ${\rm O}_{n}^{-}(q)$, $r=3$, $q=2$, $x = (\omega, \omega^{-1}, I_{n-2})$ and 
\[
{\rm fpr}(x) = \frac{1}{4}+\frac{3}{4(2^{n/2}-1)}.
\]
\end{itemize}
\end{prop}

Note that in view of the proof of Lemma \ref{l:psp4}, we are free to assume that $G \leqs {\rm P\Gamma Sp}_{4}(q)$ when $n=4$.

\begin{lem}\label{l:symp_sub1}
The conclusion to Proposition \ref{p:symp_sub} holds if $H = P_1$. 
\end{lem}

\begin{proof}
First identify $\O$ with the set of $1$-dimensional subspaces of $V$ (note that every $1$-dimensional subspace is totally singular) and observe that $|\O| = (q^n-1)/(q-1)$.  Let $x \in G$ be an element of prime order $r$. The usual argument applies if $x$ is a field automorphism, so we may assume $x \in {\rm PGSp}_{n}(q)$ is semisimple or unipotent.

First assume $r=p$ and recall that each block $J_i$ in the Jordan form of $x$ on $V$ has even multiplicity if $i$ is odd. Here $|C_{\O}(x)|$ is equal to the number of $1$-dimensional subspaces of $C_V(x)$, so $|C_{\O}(x)|$ is maximal when $x = (J_2,J_1^{n-2})$. Working with this element, we get $|C_{\O}(x)| = (q^{n-1}-1)/(q-1)$ 
and thus ${\rm fpr}(x) = (q^{n-1}-1)/(q^n-1)$. If $q \ne p$, then it is easy to check that this gives ${\rm fpr}(x) \leqs (r+1)^{-1}$ as required. However, if $q=p$ then ${\rm fpr}(x) > (q+1)^{-1}$ and this case is recorded in part (i) of Proposition \ref{p:symp_sub}. Finally, if $q=p$ and $\nu(x) \geqs 2$, then $|C_{\O}(x)|$ is maximal when $x$ has Jordan form $(J_2^2,J_1^{n-4})$, in which case $|C_{\O}(x)| = (q^{n-2}-1)/(q-1)$ and we deduce that ${\rm fpr}(x) \leqs (q+1)^{-1}$.

For the remainder, let us assume $r \ne p$. If $r=2$ then $|C_{\O}(x)|$ is maximal when $x = (-I_2,I_{n-2})$ and we compute 
\[
|C_{\O}(x)| = \frac{q^2-1}{q-1}+\frac{q^{n-2}-1}{q-1}.
\]
This implies that ${\rm fpr}(x) \leqs 1/3$. Now assume $r$ is odd and let $i \geqs 1$ be minimal such that $r$ divides $q^i-1$. Set $j = 2i$ if $i$ is odd, otherwise $j = i$. Note that if $i \geqs 2$ then $|C_{\O}(x)|$ coincides with the number of $1$-dimensional subspaces in $C_V(x)$. In particular, if $i$ is even, then $|C_{\O}(x)|$ is maximal when $x = (\L, I_{n-i})$, so $|C_{\O}(x)| \leqs (q^{n-j}-1)/(q-1)$ and it is easy to check that ${\rm fpr}(x) \leqs (q^{i/2}+2)^{-1}$ as required. The same upper bound on $|C_{\O}(x)|$ holds if $i \geqs 3$ is odd (with equality if $x = (\L,\L^{-1},I_{n-j})$) and we deduce that ${\rm fpr}(x) \leqs q^{-i}$. Finally, if $i=1$ then $|C_{\O}(x)| \leqs 2+(q^{n-2}-1)/(q-1)$ and once again the result follows.
\end{proof}

\begin{lem}\label{l:symp_sub2}
The conclusion to Proposition \ref{p:symp_sub} holds if $H = P_m$ with $2 \leqs m \leqs n/2$. 
\end{lem}

\begin{proof}
Here we identify $\O$ with the set of totally singular $m$-spaces in $V$, so  
\[
\frac{1}{2}q^{m(2n-3m+1)/2} <  |\O| = \frac{|{\rm Sp}_{n}(q)|}{q^{m(2n-3m+1)/2}|{\rm Sp}_{n-2m}(q)||{\rm GL}_{m}(q)|} <  2\left(\frac{q}{q-1}\right)q^{m(2n-3m+1)/2}.
\] 
We may assume $x \in {\rm PGSp}_{n}(q)$ has prime order $r$. For $n \geqs 6$,  \cite[Proposition 3.15]{GK} gives
\begin{equation}\label{e:sp1}
{\rm fpr}(x) < 2q^{-(n/2-1)}+q^{-n/2}+q^{-m}.
\end{equation}

First assume $r=p$. If $n \geqs 6$ then one can check that the upper bound in \eqref{e:sp1} is sufficient unless $q=2$ and $n \in \{6,8,10\}$, or $(n,q) = (6,3)$. All of these special cases can be handled using {\sc Magma}. For example, suppose $(n,q) = (6,2)$. If $m=2$, then $|\O| = 315$ and $|C_{\O}(x)| \leqs 75$ (maximal if $x = (J_2,J_1^4)$), and for $m=3$ we have $|\O| = 135$ and $|C_{\O}(x)| \leqs 39$ (maximal if $x$ is an $a_2$-type involution in the notation of \cite{AS}). In both cases, ${\rm fpr}(x) \leqs 1/3$ as required. 

Now assume $r=p$ and $n=4$, so $m=2$. If $p=2$, then $H$ is ${\rm Aut}(G_0)$-conjugate to $P_1$ and so the result in this case follows from the proof of the previous lemma. For $q$ odd, we can use the fact $G_0 \cong \O_5(q)$ and the action of $G_0$ on $\O$ is permutation isomorphic to the action of $\O_5(q)$ on the set $\Gamma$ of $1$-dimensional totally singular subspaces of the $5$-dimensional orthogonal module. In terms of the respective Jordan forms, this isomorphism induces the following correspondence: 
\[
(J_2,J_1^2) \mapsto (J_2^2,J_1),\;\; (J_2^2) \mapsto (J_3,J_1^2),\;\; (J_4) \mapsto (J_5).
\]
If $x = (J_4)$ then it is easy to see that $|C_{\O}(x)|=1$. In the remaining cases, we can appeal to the proof of Lemma \ref{l:orth1_sub1}, which shows that $|C_{\Gamma}(y)| = q+1$ if $y \in \O_5(q)$ has Jordan form $(J_2^2,J_1)$, and $|C_{\Gamma}(y)| = 2q+1$ if $y = (J_3,J_1^2)$. In particular, $|C_{\O}(x)| \leqs 2q+1$ when $q$ is odd and once again we deduce that ${\rm fpr}(x) \leqs (q+1)^{-1}$. 

Next suppose $r=2$ and $p$ is odd, so $x$ is a semisimple involution. If $q \geqs 5$ then the bound in Theorem \ref{t:ls91} is sufficient, so we may assume $q=3$ (and thus $n \geqs 6$ since we are excluding the case $(n,q) = (4,3)$). Here the bound in \eqref{e:sp1} is sufficient unless $(n,m) = (6,2)$, in which case a {\sc Magma} computation yields ${\rm fpr}(x) \leqs 5/91$. 

For the remainder, let us assume $r \ne p$ and $r$ is odd. Let $i \geqs 1$ be minimal such that $r$ divides $q^i-1$. First assume $i=1$, so $q \geqs 4$. If $n \geqs 6$ then the bound in \eqref{e:sp1} is effective. Similarly, if $i=2$ and $n \geqs 6$, then the same bound is sufficient unless $m=q=2$ or $(n,m,q) = (8,4,2)$. In the latter case we have $r=3$ and a {\sc Magma} calculation shows that ${\rm fpr}(x) \leqs 1/51$. Now suppose $m=i=q=2$, so $n \geqs 6$, $r=3$ and $x = (\L^{\ell},I_{n-2\ell})$ for some $\ell \geqs 1$. Let $\a$ be the number of totally singular $2$-spaces in $C_V(x)$. Then by arguing as in the proof of Lemma \ref{l:sub_useful}(ii) we deduce that $|C_{\O}(x)| \leqs \a + (2^{2\ell}-1)/3$ and it is straightforward to check that ${\rm fpr}(x) \leqs 1/4$.

Now assume $i \in \{1,2\}$ and $n=4$, in which case $m=2$. If $i=1$ then $|C_{\O}(x)| \leqs 2(q+1)$ (maximal if $x = (\L,\L^{-1},I_2)$) and thus ${\rm fpr}(x) \leqs q^{-1}$ as required. Similarly, if $i=2$ then ${\rm fpr}(x)>0$ if and only if $x = (\L^2)$, in which case $|C_{\O}(x)| = q+1$ and we conclude that ${\rm fpr}(x) \leqs (q+2)^{-1}$. 

Next suppose $i \geqs 4$ is even. First assume $m \geqs i$, so $n \geqs 2i$. Here one can check that the bound in \eqref{e:sp1} yields ${\rm fpr}(x) \leqs (q^{i/2}+2)^{-1}$ unless $i=4$ and $(n,q) = (8,2)$, $(8,3)$ or $(10,2)$. Each of these cases can be handled using {\sc Magma}. Now assume $m<i$. In this case, $|C_{\O}(x)|$ coincides with the number of totally singular $m$-spaces in $C_V(x)$, so by working with the element $x = (\L, I_{n-i})$ we deduce that  
\[
{\rm fpr}(x) <4\left(\frac{q}{q-1}\right)q^{-mi} \leqs (q^{i/2}+2)^{-1}
\]  
and the result follows. 

Finally, let us assume $i \geqs 3$ is odd. Note that $r$ divides $t = (q^i-1)/(q-1)$. If $m<i$ then it is clear that $|C_{\O}(x)|$ is maximal when $x = (\L,\L^{-1},I_{n-2i})$ and thus
\[
{\rm fpr}(x) < 4\left(\frac{q}{q-1}\right)q^{-2mi} \leqs (t+1)^{-1}.
\]
Now suppose $m \geqs i$. If $n<4i$ then $x = (\L, \L^{-1},I_{n-2i})$ is the only possibility  and we have $|C_{\O}(x)| = \b+2\gamma$, where $\b$ (respectively, $\gamma$) is the number of totally singular $m$-spaces (respectively $(m-i)$-spaces) in $C_V(x)$. Therefore,
\[
|C_{\O}(x)| < 2\left(\frac{q}{q-1}\right)q^{m(2n-3m+1)/2}\left(q^{-2im}+2q^{-i(2n-2m+1-i)/2}\right)
\]
and thus
\[
{\rm fpr}(x) < 4\left(\frac{q}{q-1}\right)\left(q^{-2im}+2q^{-i(2n-2m+1-i)/2}\right).
\]
One can check that this gives ${\rm fpr}(x) \leqs (t+1)^{-1}$ unless $(n,m,q)= (6,3,2)$ and $i=3$. Here $r=7$, $|C_{\O}(x)| = 2$ and the result follows.

To complete the proof, we may assume $i \geqs 3$ is odd with $m \geqs i$ and $n \geqs 4i$. If $m \geqs i+1$ then one can check that the bound in \eqref{e:sp1} is sufficient unless $(n,m,q,i) = (12,4,2,3)$. Here $r=7$ and either $x  = ((\L,\L^{-1})^2)$ and $|C_{\O}(x)| = 0$, or $x = (\L,\L^{-1},I_6)$ and $|C_{\O}(x)| = \b+2\gamma$ as above. The reader can check that ${\rm fpr}(x) \leqs 1/8$. Finally, suppose $m=i$. In this case, we find that \eqref{e:sp1} is sufficient unless $q=2$ and $r = 2^i-1$ is a Mersenne prime. Here $\dim C_V(x) = n-2m\ell$ with $\ell \geqs 1$ and it is straightforward to see that $|C_{\O}(x)|$ is maximal when $x=((\L,\L^{-1})^{\ell},I_{n-2m\ell})$. Then by arguing as in the proof of Lemma \ref{l:sub_useful}(ii), viewing $z = ((\L,\L^{-1})^{\ell})$ as an element of the field extension subgroup ${\rm Sp}_{2\ell}(2^m)<{\rm Sp}_{2\ell m}(2)$, we deduce that 
\[
|C_{\O}(x)| = \delta + 2\left(\frac{2^{m\ell}-1}{2^m-1}\right),
\]
where $\delta$ is the number of totally singular $m$-spaces in $C_V(x)$. From here, the desired bound ${\rm fpr}(x) \leqs 2^{-m}$ quickly follows.
\end{proof}

\begin{lem}\label{l:symp_sub3}
The conclusion to Proposition \ref{p:symp_sub} holds if $H = N_m$ with $2 \leqs m < n/2$ even. 
\end{lem}

\begin{proof}
Here $n \geqs 6$ and we identify $\O$ with the set of nondegenerate $m$-dimensional subspaces of $V$. Note that 
\[
q^{m(n-m)} < |\O| = \frac{|{\rm Sp}_{n}(q)|}{|{\rm Sp}_{m}(q)||{\rm Sp}_{n-m}(q)|} < 2q^{m(n-m)}.
\]
Let $x \in G$ be an element of prime order $r$. By \cite[Proposition 3.16]{GK} we have
\begin{equation}\label{e:gk2}
{\rm fpr}(x) < 2q^{-(n/2-d)} +q^{-n/2} + q^{-m/2} + q^{-(n-m)},
\end{equation}
where $d = (2,q-1)$. Our aim is to establish the bound ${\rm fpr}(x) \leqs (r+1)^{-1}$ with no exceptions. By the usual argument, we may assume $x \in {\rm PGSp}_{n}(q)$.

Suppose $x$ is unipotent, so $r=p$. If $m \geqs 4$ then it is easy to check that the bound in \eqref{e:gk2} is sufficient unless $(n,m,q) = (10,4,2)$. In the latter case, an easy {\sc Magma} computation shows that ${\rm fpr}(x) \leqs 26/341$ (maximal if $x$ is a $b_1$ involution). Now assume $m=2$. Here $|C_{\O}(x)|$ is maximal when $x = (J_2,J_1^{n-2})$ and we compute the bounds 
\[
|x^G \cap H| \leqs (q^2-1) + (q^{n-2}-1),\;\; |x^G| \geqs \frac{1}{d}(q^n-1),
\]
whence 
\[
{\rm fpr}(x) \leqs \frac{d(q^{n-2}+q^2-2)}{q^n-1} \leqs (q+1)^{-1}
\]
and the result follows.

Now assume $r \ne p$ and $x$ is semisimple. Suppose $r=2$. By inspecting the bounds in Theorem \ref{t:ls91} and \eqref{e:gk2}, we may assume $m=2$ and $q=3$. 
The case $n=6$ can be handled using {\sc Magma}, so let us assume $n \geqs 8$. Here one can check that $|C_{\O}(x)|$ is maximal when $x = (-I_2,I_{n-2})$, which preserves an orthogonal decomposition $V = U \perp W$ into nondegenerate spaces with $\dim U = 2$. In particular, $C_{\O}(x)$ comprises $U$ and every nondegenerate $2$-space in $W$, whence 
\[
|C_{\O}(x)| = 1+\frac{|{\rm Sp}_{n-2}(3)|}{|{\rm Sp}_{2}(3)||{\rm Sp}_{n-4}(3)|} < 2\cdot 3^{2(n-4)}
\]
and thus ${\rm fpr}(x) < 2/81$. 

To complete the proof, we may assume $x$ is semisimple and $r$ is odd. Let $i \geqs 1$ be minimal such that $r$ divides $q^i-1$. First assume $i$ is even, so $r$ divides $q^{i/2}+1$. By arguing as in the proof of Lemma \ref{l:psu_sub4}, we observe that $|C_{\O}(x)|$ is maximal when $x$ is of the form $(\L^{\ell},I_{n-\ell i})$ for some $\ell \geqs 1$. If $m<i$ then $|C_{\O}(x)|$ coincides with the number of nondegenerate $m$-spaces in $C_V(x)$, so $|C_{\O}(x)|$ is maximal when $x = (\L,I_{n-i})$ and we deduce that ${\rm fpr}(x) < 2q^{-mi}$. This yields ${\rm fpr}(x) \leqs (q^{i/2}+2)^{-1}$ and the result follows. 

Next assume $m \geqs i+2$. Here one can check that the upper bound in \eqref{e:gk2} is sufficient unless $(n,m,i,q) = (14,6,4,2)$ or $(m,i,q) = (4,2,2)$. In the former case, $r = 5$ and $x = (\L^{\ell},I_{14-4\ell})$ with $1 \leqs \ell \leqs 3$, and it is straightforward to verify the bound ${\rm fpr}(x) \leqs 1/6$ by computing $|x^G \cap H|$ and $|x^G|$. For example, if $\ell=1$ then
\[
|x^G \cap H| = \frac{|{\rm Sp}_{6}(2)|}{|{\rm GU}_{1}(4)||{\rm Sp}_{2}(2)|} + \frac{|{\rm Sp}_{8}(2)|}{|{\rm GU}_{1}(4)||{\rm Sp}_{4}(2)|} < 2^{16}(2^8+1)
\]
and $|x^G|>2^{47}$. Similarly, if $(m,i,q) = (4,2,2)$ then $r=3$ and we have $x = (\L^{\ell},I_{n-2\ell})$ with $1 \leqs \ell \leqs n/2$. Here $|x^G|>\frac{1}{2}2^{2n\ell-3\ell^2+\ell}$ and for $\ell \geqs 2$ we compute
\begin{align*}
|x^G \cap H| < & \; \frac{|{\rm Sp}_{4}(2)|}{|{\rm GU}_{2}(2)|} \cdot \frac{|{\rm Sp}_{n-4}(2)|}{|{\rm GU}_{\ell-2}(2)||{\rm Sp}_{n-2\ell}(2)|} + \frac{|{\rm Sp}_{4}(2)|}{|{\rm GU}_{1}(2)||{\rm Sp}_{2}(2)|} \cdot \frac{|{\rm Sp}_{n-4}(2)|}{|{\rm GU}_{\ell-1}(2)||{\rm Sp}_{n-2\ell-2}(2)|} \\
& \; + \frac{|{\rm Sp}_{n-4}(2)|}{|{\rm GU}_{\ell}(2)||{\rm Sp}_{n-2\ell-4}(2)|} \\
< & \; 2^{2n\ell-3\ell^2+\ell}\left(2^{-4n+4\ell+8} + 2^{-2n-2\ell+10}+2^{-8\ell}\right).
\end{align*}
It is easy to check that this yields ${\rm fpr}(x) \leqs 1/4$ and a very similar argument shows that the same conclusion holds when $\ell=1$. 

To complete the argument when $i$ is even, it remains to handle the case $m=i$. As noted above, we may assume $x = (\L^{\ell},I_{n-m\ell})$ and we compute
\begin{align*}
|x^G| & > \frac{1}{2}q^{m\ell(2n - \ell-m\ell+1)/2} \\
|x^G \cap H| & < \frac{|{\rm Sp}_{m}(q)|}{|{\rm GU}_{1}(q^{m/2})|} \cdot \frac{|{\rm Sp}_{n-m}(q)|}{|{\rm GU}_{\ell-1}(q^{m/2})||{\rm Sp}_{n-m\ell}(q)|} + \frac{|{\rm Sp}_{n-m}(q)|}{|{\rm GU}_{\ell}(q^{m/2})||{\rm Sp}_{n-m-m\ell}(q)|} \\
& < q^{m\ell(2n - \ell-m\ell+1)/2}\left(q^{m\ell+m^2-m-mn}+q^{-m^2\ell}\right).
\end{align*}
It is straightforward to check that these bounds imply that ${\rm fpr}(x) \leqs (q^{m/2}+2)^{-1}$.

Finally, suppose $i$ is odd.  If $m<2i$ then $|C_{\O}(x)|$ is the number of nondegenerate $m$-spaces in $C_V(x)$, which is maximal when $x = (\L,\L^{-1},I_{n-2i})$. The reader can check that ${\rm fpr}(x)<2q^{-2mi} \leqs q^{-i}$ and the result follows. For $m \geqs 2i+2$ it is easy to show that the upper bound in \eqref{e:gk2} is sufficient, so we may assume $m=2i$. As noted above, we may also assume that $x = ((\L,\L^{-1})^{\ell},I_{n-m\ell})$ and it is straightforward to show that 
\[
{\rm fpr}(x) < 8\left(q^{m\ell+m^2-m-mn}+q^{-m^2\ell}\right),
\]
which yields ${\rm fpr}(x) \leqs q^{-i}$. 
\end{proof}

\begin{lem}\label{l:symp_sub4}
The conclusion to Proposition \ref{p:symp_sub} holds if $q$ is even and $H$ is of type ${\rm O}_{n}^{\e}(q)$. 
\end{lem}

\begin{proof}
Here $H \cap G_0 = {\rm O}_{n}^{\e}(q)$ and $|\O| = q^{n/2}(q^{n/2}+\e)/2$. Let $x \in G$ be an element of prime order $r$. As usual, we may assume $x \in {\rm PGSp}_{n}(q)$. 

First assume $r=2$. If $x = (J_2,J_1^{n-2})$ is a $b_1$-type involution then 
\[
|x^G \cap H| = \frac{|{\rm O}_{n}^{\e}(q)|}{2|{\rm Sp}_{n-2}(q)|} = q^{n/2-1}(q^{n/2}-\e),\;\; |x^G| = q^n-1
\]
and thus ${\rm fpr}(x) = q^{n/2-1}/(q^{n/2}+\e)$. This is at most $1/3$ if and only if $q \geqs 4$, so the case $q=2$ is an exception and it is recorded in part (ii) of Proposition \ref{p:symp_sub}. Now assume $x \in G$ is an involution with $\nu(x) = s \geqs 2$. By applying the bounds in the proof of \cite[Proposition 3.22]{Bur2}, we deduce that ${\rm fpr}(x) < 4q^{-s}$, which is sufficient unless $q=2$ and $s \in \{2,3\}$. If $x=a_2$ then
\[
|x^G \cap H| = \frac{|{\rm O}_{n}^{\e}(2)|}{2^{2n-7}|{\rm O}_{n-4}^{\e}(2)||{\rm Sp}_{2}(2)|} = \frac{1}{3}(2^{n-2}-1)(2^{n/2-2}+\e)(2^{n/2}-\e)
\]
and 
\[
|x^G| = \frac{|{\rm Sp}_{n}(2)|}{2^{2n-5}|{\rm Sp}_{n-4}(2)||{\rm Sp}_{2}(2)|} = \frac{1}{3}(2^{n-2}-1)(2^{n}-1)
\]
which yields ${\rm fpr}(x)  = (2^{n/2-2}+\e)/(2^{n/2}+\e) \leqs 1/3$ (with equality if $(n,\e) = (6,+)$). A similar calculation shows that the same conclusion holds when $x = c_2$. And for $x = b_3$ we get
\[
{\rm fpr}(x) = \frac{2^{n/2-3}}{2^{n/2}+\e} \leqs \frac{1}{7}
\]
and the result follows.

For the remainder, let us assume $r$ is odd. As usual, let $i \geqs 1$ be minimal such that $r$ divides $q^i-1$. First assume $i$ is even and write $x = (I_e,\L_1^{a_1}, \ldots, \L_t^{a_t})$, where $e = \dim C_V(x)$. If $e=0$ then by computing $|x^{G_0}\cap H|$ and $|x^{G_0}|$ we deduce that $|C_{\O}(x)| = 1$ and ${\rm fpr}(x) \leqs (q^{i/2}+2)^{-1}$ as required. Now assume $e>0$, so we may view $C_V(x)$ as a nondegenerate orthogonal space of type $\e'$ and we get
\[
{\rm fpr}(x) = \frac{q^{e/2}(q^{e/2}+\e')}{q^{n/2}(q^{n/2}+\e)} \leqs q^{-i/2}\left(\frac{q^{(n-i)/2}+1}{q^{n/2}-1}\right)
\]
since $e \leqs n-i$. One can now check that this bound gives ${\rm fpr}(x) \leqs (q^{i/2}+2)^{-1}$ unless $i=q=2$. Here $r=3$ and we quickly reduce to the case where $x = (\L, I_{n-2})$. If $\e=+$ then $\e' = -$ and we obtain ${\rm fpr}(x) \leqs 1/4$. However, if $\e=-$ then $\e' = +$ and we compute
\[
{\rm fpr}(x) = \frac{1}{4}+\frac{3}{4(2^{n/2}-1)}.
\]
The latter case is recorded in part (iii) of Proposition \ref{p:symp_sub}. 

Finally, suppose $i$ is odd. As above, if $e=0$ then $|C_{\O}(x)|=1$ and we deduce that ${\rm fpr}(x) \leqs q^{-i}$. Similarly, for $e>0$ we get 
\[
{\rm fpr}(x) = \frac{q^{e/2}(q^{e/2}+\e)}{q^{n/2}(q^{n/2}+\e)} \leqs q^{-i}\left(\frac{q^{n/2-i}+1}{q^{n/2}+1}\right) \leqs q^{-i}
\]
and the result follows.
\end{proof}

\subsection{Odd dimensional orthogonal groups}\label{ss:orth1}

We begin our analysis of subspace actions of orthogonal groups by handling the groups with socle $G_0 = \O_n(q)$, where $nq$ is odd. Note that we may assume $n \geqs 7$ since $\O_3(q) \cong {\rm L}_2(q)$ and $\O_5(q) \cong {\rm PSp}_4(q)$. 

Our main result is the following. Note that in part (i), $x \in {\rm SO}_{n}(q)$ is an involution of type $(-I_{n-1},I_1)$ with a plus-type $(-1)$-eigenspace. Whereas in part (ii), the $(-1)$-eigenspace of $x$ is a minus-type space.

\begin{prop}\label{p:orth1_sub}
Let $G \leqs {\rm Sym}(\O)$ be a finite almost simple primitive permutation group with point stabilizer $H$ and socle $G_0 = \O_n(q)$, where $n \geqs 7$ and $q$ is odd. Assume $H$ is a subspace subgroup of $G$. If $x \in G$ has prime order $r$, then either ${\rm fpr}(x) \leqs (r+1)^{-1}$, or one of the following holds:
\begin{itemize}\addtolength{\itemsep}{0.2\baselineskip}
\item[{\rm (i)}] $H = P_1$, $r=2$, $q=3$, $x = (-I_{n-1}, I_1)^{+}$ and 
\[
{\rm fpr}(x) = \frac{1}{3} + \frac{2}{3(3^{(n-1)/2}+1)}.
\]
\item[{\rm (ii)}] $H = N_1^{-}$, $r=2$, $q=3$, $x = (-I_{n-1}, I_1)^{-}$ and 
\[
{\rm fpr}(x)= \frac{1}{3}+\frac{2(3^{(n-3)/2}+1)}{3^{(n-1)/2}(3^{(n-1)/2}-1)}.
\]
\end{itemize}
\end{prop}

\begin{lem}\label{l:orth1_sub1}
The conclusion to Proposition \ref{p:orth1_sub} holds if $H = P_1$. 
\end{lem}

\begin{proof}
As usual, we identify $\O$ with the set of $1$-dimensional totally singular subspaces of $V$, noting that 
\[
|\O| = \frac{|{\rm SO}_{n}(q)|}{q^{n-2}|{\rm SO}_{n-2}(q)||{\rm GL}_{1}(q)|} = \frac{q^{n-1}-1}{q-1}.
\]
Let $x \in G$ be an element of prime order $r$. By arguing in the usual fashion, we may assume $x \in {\rm SO}_{n}(q)$ is semisimple or unipotent.

First assume $r=p$, in which case $|C_{\O}(x)|$ is the number of totally singular $1$-spaces in $C_V(x)$. Let us also recall that each $J_i$ block in the Jordan form of $x$ on $V$ has even multiplicity if $i$ is even. If $x = (J_2^2,J_{1}^{n-4})$ then $C_V(x) = U \oplus W$, where $U$ is a totally singular $2$-space and $W$ is a nondegenerate $(n-4)$-space, and we compute  
\[
|C_{\O}(x)| = \frac{q^2-1}{q-1} + q^2\left(\frac{q^{n-5}-1}{q-1}\right) = \frac{q^{n-3}-1}{q-1}.
\] 
Similarly, if $x = (J_3,J_1^{n-3})$ then $C_V(x) = \la u \ra \oplus W$, where $u$ is totally singular and $W$ is nondegenerate of dimension $n-3$. Here $C_{\O}(x)$ comprises $\la u \ra$ and every $1$-space of the form $\la \l u + w \ra$, where $\la w \ra \subseteq W$ is totally singular and $\l \in \mathbb{F}_q$. Therefore, $|C_{\O}(x)|$ is  maximal when $W$ is a plus-type space and in this situation we get   
\[
|C_{\O}(x)| = 1 + q \left(\frac{(q^{(n-5)/2}+1)(q^{(n-3)/2}-1)}{q-1}\right).
\]
In both cases, it is easy to check that ${\rm fpr}(x) \leqs (q+1)^{-1}$. If $x$ is any other unipotent element of order $p$, then $\dim C_V(x) \leqs n-4$ and the bound $|C_{\O}(x)| \leqs (q^{n-4}-1)/(q-1)$ is sufficient.

Now assume $r\ne p$. If $r=2$ then $|C_{\O}(x)|$ is maximal when $x \in {\rm SO}_{n}(q)$ is an involution of the form $(-I_{n-1},I_1)$ with a plus-type $(-1)$-eigenspace. Here
\[
|C_{\O}(x)| = \frac{(q^{(n-3)/2}+1)(q^{(n-1)/2}-1)}{q-1}
\]
and thus
\[
{\rm fpr}(x) = \frac{(q^{(n-3)/2}+1)(q^{(n-1)/2}-1)}{q^{n-1}-1}.
\]
For $q \geqs 5$ it is easy to check that this is at most $1/3$ and the result follows. On the other hand, if $q=3$ then
\[
{\rm fpr}(x) = \frac{1}{3} + \frac{2}{3(3^{(n-1)/2}+1)}
\]
and this special case is recorded in part (i) of Proposition \ref{p:orth1_sub}. Note that if $q=3$ and $x$ is any other involution, then $|C_{\O}(x)|$ is maximal when $x = (-I_{n-1},I_1)$ has a minus-type $(-1)$-eigenspace, in which case  
\[
|C_{\O}(x)| = \frac{1}{2}(3^{(n-3)/2}-1)(3^{(n-1)/2}+1)
\]
and we conclude that ${\rm fpr}(x) \leqs 1/3$.

To complete the proof, we may assume $x$ is semisimple and $r$ is odd. Let $i \geqs 1$ be minimal such that $r$ divides $q^i-1$. Note that if $i \geqs 2$ then $C_{\O}(x)$ 
comprises the totally singular $1$-spaces in the nondegenerate space $C_V(x)$. First assume $i$ is even. By the previous observation, it follows that $|C_{\O}(x)|$ is maximal when $x = (\L, I_{n-i})$, in which case $|C_{\O}(x)| = (q^{n-i-1}-1)/(q-1)$ and it is easy to check that ${\rm fpr}(x) \leqs (q^{i/2}+2)^{-1}$. Similarly, if $i \geqs 3$ is odd then $|C_{\O}(x)| \leqs (q^{n-2i-1}-1)/(q-1)$ and we obtain ${\rm fpr}(x) \leqs q^{-i}$. Finally, if $i=1$ then $|C_{\O}(x)| = 2+(q^{n-3}-1)/(q-1)$ (maximal when $x = (\L,\L^{-1},I_{n-2})$) and once again the desired bound holds.
\end{proof}

\begin{lem}\label{l:orth1_sub2}
The conclusion to Proposition \ref{p:orth1_sub} holds if $H = P_m$ with $2 \leqs m < n/2$. 
\end{lem}

\begin{proof}
Here $\O$ is the set of totally singular $m$-spaces in $V$ and we have 
\[
\frac{1}{2}q^{m(2n-3m-1)/2} < |\O| = \frac{|{\rm SO}_{n}(q)|}{q^{m(2n-3m-1)/2}|{\rm SO}_{n-2m}(q)||{\rm GL}_{m}(q)|} < 2q^{m(2n-3m-1)/2}.
\]
Let $x \in G$ be an element of prime order $r$. As usual, we may assume $x \in {\rm SO}_{n}(q)$ is semisimple or unipotent. By \cite[Proposition 3.15]{GK} we have 
\begin{equation}\label{e:gk3}
{\rm fpr}(x) < 2q^{-(n-3)/2} + q^{-(n-1)/2} + q^{-m}. 
\end{equation}

If $r = p$ then the bound in \eqref{e:gk3} is sufficient unless $(n,q) = (7,3)$, in which case a {\sc Magma} computation yields ${\rm fpr}(x) \leqs 37/280$. Similarly, if $r=2$ and $p$ is odd, then the bound in \eqref{e:gk3} is sufficient unless $(n,m,q) = (7,2,3)$, where we obtain ${\rm fpr}(x) \leqs 1/7$ by a straightforward {\sc Magma} calculation. 

Finally, suppose $r \ne p$ and $r$ is odd. Let $i \geqs 1$ be minimal such that $r$ divides $q^i-1$. If $i \leqs 2$ then it is easy to check that the bound in \eqref{e:gk3} is always sufficient, so we may assume $i \geqs 3$. Suppose $i \geqs 4$ is even. If $m \geqs i$ then $n \geqs 2i+1$ and \eqref{e:gk3} implies that
${\rm fpr}(x) < 2q^{-i}+2q^{1-i}$, which in turn is less than $(q^{i/2}+1)^{-1}$ unless $(n,q,i) = (9,3,4)$. Here $r=5$ and one checks that the previous bound yields ${\rm fpr}(x) \leqs 1/6$. On the other hand, if $m<i$ then $|C_{\O}(x)|$ is the number of totally singular $m$-spaces in $C_V(x)$, whence
\[
|C_{\O}(x)| < 2q^{m(2n-3m-1)/2}\cdot q^{-mi}
\] 
(maximal when $x = (\L,I_{n-i})$) and thus ${\rm fpr}(x) <4q^{-mi} \leqs (q^{i/2}+2)^{-1}$. 

Similar reasoning applies when $i \geqs 3$ is odd. First note that $r$ divides $t = (q^i-1)/(q-1)$. Now, if $m<i$ then $|C_{\O}(x)|$ is maximal when $x = (\L,\L^{-1},I_{n-2i})$, in which case
\[
|C_{\O}(x)| < 2q^{m(2n-3m-1)/2} \cdot q^{-2mi}
\]
and we get ${\rm fpr}(x) < 4q^{-2mi} \leqs (t+1)^{-1}$ as required. Now assume $m \geqs i$. 
If $n=2i+1$ then $m=i$, $x = (\L,\L^{-1},I_1)$ and $|C_{\O}(x)|=2$. Similarly, if $n = 2i+3$ then $x = (\L,\L^{-1},I_3)$ and either $m=i$ and $|C_{\O}(x)|=2$, or $m=i+1$ and $|C_{\O}(x)|=2(q+1)$. In each of these cases, the desired bound holds. Finally, if $n \geqs 2i+5$ then the bound in \eqref{e:gk3} is sufficient.
\end{proof}

\begin{lem}\label{l:orth1_sub3}
The conclusion to Proposition \ref{p:orth1_sub} holds if $H = N_1^{\e}$. 
\end{lem}

\begin{proof}
First identify $\O$ with the set of nondegenerate $1$-spaces $U$ such that $U^{\perp}$ has type $\e$ (in this situation, we will refer to $U$ itself as an $\e$-type $1$-space). Note that 
\[
|\O| = \frac{|{\rm SO}_{n}(q)|}{2|{\rm SO}_{n-1}^{\e}(q)|} = \frac{1}{2}q^{(n-1)/2}(q^{(n-1)/2}+\e).
\]
Let $x \in G$ be an element of prime order $r$. We may assume $x \in {\rm SO}_{n}(q)$ is semisimple or unipotent.

First assume $r=p$ and note that $|C_{\O}(x)|$ coincides with the total number of $\e$-type $1$-spaces in $C_V(x)$. In particular, if $\dim C_V(x) \leqs n-3$ then $|C_{\O}(x)|$ is at most $(q^{n-3}-1)/(q-1)$ and it is easy to check that ${\rm fpr}(x) \leqs (q+1)^{-1}$. Therefore, we may assume $x=(J_2^2,J_1^{n-4})$ or $(J_3,J_1^{n-3})$. In the former case, we compute
\[
|x^G \cap H| = \frac{|{\rm SO}_{n-1}^{\e}(q)|}{q^{2n-9}|{\rm Sp}_{2}(q)||{\rm SO}_{n-5}^{\e}(q)|},\;\; 
|x^G| = \frac{|{\rm SO}_{n}(q)|}{q^{2n-7}|{\rm Sp}_{2}(q)||{\rm SO}_{n-4}(q)|} 
\]
and we obtain ${\rm fpr}(x) \leqs (q+1)^{-1}$. A very similar calculation establishes the same bound when $x = (J_3,J_1^{n-3})$.

Now assume $r \ne p$. First we handle the case $r=2$. Let $x = (-I_{n-1},I_1)$ be an involution with a minus-type $(-1)$-eigenspace. If $\e=+$ then
\[
|x^G\cap H| = \frac{|{\rm SO}_{n-1}^{+}(q)|}{2|{\rm SO}_{n-2}(q)|},\;\; |x^G| = \frac{|{\rm SO}_{n}(q)|}{2|{\rm SO}_{n-1}^{-}(q)|}
\]
and we compute ${\rm fpr}(x) = q^{-1} \leqs 1/3$. Similarly, if $\e=-$ then 
\[
{\rm fpr}(x) = \frac{2+q^{(n-3)/2}(q^{(n-1)/2}+1)}{q^{(n-1)/2}(q^{(n-1)/2}-1)},
\]
which is at most $1/3$ if $q \geqs 5$. However, if $q=3$ then one checks that ${\rm fpr}(x)>1/3$ and so this case is recorded in part (ii) of Proposition \ref{p:orth1_sub}. Now if $x$ is any other involution, then $|C_{\O}(x)|$ is maximal when $x = (-I_{n-1},I_1)$ with a plus-type $(-1)$-eigenspace and it is straightforward to check that ${\rm fpr}(x)  \leqs 1/3$. 

Finally, let us assume $x$ is semisimple and $r$ is odd. Let $i \geqs 1$ be minimal such that $r$ divides $q^i-1$ and note that $|C_{\O}(x)|$ is maximal when $x = (\L,I_{n-i})$ (for $i$ even) or $x = (\L,\L^{-1},I_{n-2i})$ (for $i$ odd). If we fix such an element $x$, then $|C_{\O}(x)|$ is the number of $\e'$-type $1$-spaces in $C_V(x)$, where $\e' = \e$ if $i$ is odd, otherwise $\e'=-\e$. For example, if $i$ is even and $\e=+$, then $x$ preserves an orthogonal decomposition $V = U \perp W$, with $x$ acting irreducibly on the minus-type $i$-space $U$. It follows that $|C_{\O}(x)|$ is the number of nondegenerate $1$-spaces in $W = C_V(x)$ with a minus-type orthogonal complement in $W$, whence $\e'=-$. By setting $j = i$ if $i$ is even and $j = 2i$ if $i$ is odd, we deduce that   
\[
|C_{\O}(x)| \leqs \frac{1}{2}q^{(n-j-1)/2}(q^{(n-j-1)/2}+\e').
\]
For $i$ even, it is easy to check that this bound implies that ${\rm fpr}(x) \leqs (q^{i/2}+2)^{-1}$. Similarly, if $i$ is odd then we obtain ${\rm fpr}(x) \leqs q^{-i}$.
\end{proof}

\begin{lem}\label{l:orth1_sub4}
The conclusion to Proposition \ref{p:orth1_sub} holds if $H = N_m^{\eta}$ with $2 \leqs  m < n/2$. 
\end{lem}

\begin{proof}
Recall that if $m$ is even, then we may identify $\O$ with the set of nondegenerate $m$-spaces of type $\eta \in \{+,-\}$. Similarly, if $m$ is odd then we take $\O$ to be the set of nondegenerate $m$-spaces $U$ such that $U^{\perp}$ has type $\eta$. If $m$ is even, then  
\[
\frac{1}{4}\left(\frac{q}{q+1}\right)q^{m(n-m)} < |\O| = \frac{|{\rm SO}_{n}(q)|}{2|{\rm SO}_{m}^{\eta}(q)||{\rm SO}_{n-m}(q)|} < q^{m(n-m)}
\]
and it is easy to check that the same upper and lower bounds on $|\O|$ are also valid when $m$ is odd. Let $x \in G$ be an element of prime order $r$. As usual, we may assume $x \in {\rm SO}_{n}(q)$ is semisimple or unipotent. By \cite[Proposition 3.16]{GK} we note that 
\begin{equation}\label{e:gk4}
{\rm fpr}(x) < 2q^{-(n-3)/2}+q^{-(n-1)/2}+q^{-m} +q^{-(n-m-\a)/2},
\end{equation}
where $\a = 1-\eta$ if $m$ is odd, otherwise $\a=1$.

First assume $r=p$. If $n \geqs 9$ then the bound in \eqref{e:gk4} is sufficient, so let us assume $n=7$. If $m=2$, or if $m=3$ and $\eta=+$, then one checks that \eqref{e:gk4} is effective if $q \geqs 7$, while the cases $q \in \{3,5\}$ can be checked using {\sc Magma}. Now assume $m=3$ and $\eta=-$. Here we find that $|C_{\O}(x)|$ is maximal when $x = (J_3,J_1^4)$ and we compute
\begin{align*}
|x^G \cap H| & = \frac{|{\rm SO}_{4}^{-}(q)|}{2q^2} + \frac{|{\rm SO}_{3}(q)|}{q} = \frac{1}{2}(q^2-1)(q^2+3) \\
|x^G| & = \frac{|{\rm SO}_{7}(q)|}{2q^5|{\rm SO}_{4}^{-}(q)|} = \frac{1}{2}q^2(q^2-1)(q^6-1),
\end{align*}
which immediately implies that ${\rm fpr}(x) \leqs (q+1)^{-1}$.

If $r=2$ then the bound in \eqref{e:gk4} is sufficient unless $(n,q) = (7,3)$; in the latter case, we can use {\sc Magma} to show that ${\rm fpr}(x) \leqs 1/3$. 

Finally, suppose $r \ne p$ and $r$ is odd. Let $i \geqs 1$ be minimal such that $r$ divides $q^i-1$.  First assume $i$ is even. As in previous cases, we are free to assume that $x = (\L^{\ell},I_{n-\ell i})$ for some $\ell \geqs 1$. If $m<i$ then $|C_{\O}(x)|$ is the number of $\eta$-type $m$-spaces in $C_V(x)$, which is maximal when $x = (\L, I_{n-i})$ and we deduce that
\[
{\rm fpr}(x) < 4\left(\frac{q+1}{q}\right)q^{-mi} \leqs (q^{i/2}+2)^{-1}.
\]
For $m \geqs i$ one checks that the bound in \eqref{e:gk4} is sufficient unless $n=2m+1$ and $m \in \{i,i+1\}$.

Suppose $n=2m+1$ and $m=i+1$. Here $m$ is odd and there are three cases to consider. If $x = (\L,I_{m+2})$ then
\[
|x^{G_0} \cap H| = \frac{|{\rm SO}_{m}(q)|}{|{\rm GU}_{1}(q^{i/2})|} + \frac{|{\rm SO}_{m+1}^{\eta}(q)|}{|{\rm SO}_{2}^{-\eta}(q)||{\rm GU}_{1}(q^{i/2})|} < q^{\frac{1}{2}m^2}
\]
and by estimating $|x^G|$ we deduce that 
\[
{\rm fpr}(x) < 2\left(\frac{q+1}{q}\right)q^{-m^2+m+1/2} \leqs (q^{i/2}+2)^{-1}.
\]
Similarly, one can check that the same conclusion holds if $x = (\L^2,I_3)$. Finally, if $n = 7$ and $x = (\L^3,I_1)$ then 
\[
|x^{G_0} \cap H| \leqs \frac{|{\rm SO}_{3}(q)|}{|{\rm GU}_{1}(q)|} \cdot \frac{|{\rm O}_{4}^{+}(q)|}{|{\rm GU}_{2}(q)|} < 2q^4
\]
and the result follows since $|x^G|>\frac{1}{2}(q+1)^{-1}q^{13}$.

Now assume $n=2m+1$ and $m=i$, in which case $x = (\L,I_{m+1})$ or $(\L^2,I_1)$. If $x = (\L^2,I_1)$ then
\[
|x^{G_0} \cap H| \leqs \frac{|{\rm O}_{m}^{-}(q)|}{|{\rm GU}_{1}(q^{m/2})|} \cdot \frac{|{\rm SO}_{m+1}(q)|}{|{\rm GU}_{1}(q^{m/2})|} < 2q^{m(m-1)}
\]
and we deduce that 
\[
{\rm fpr}(x) < 4\left(\frac{q+1}{q}\right)q^{-m^2} \leqs (q^{i/2}+2)^{-1}.
\]
One can check that the same conclusion holds when $x = (\L,I_{m+1})$.

A very similar argument applies when $i$ is odd. If $m<2i$ then $|C_{\O}(x)|$ is maximal when $x = (\L,\L^{-1}, I_{n-2i})$ and we deduce that 
\[
{\rm fpr}(x) < 4 \left(\frac{q+1}{q}\right)q^{-2mi} \leqs q^{-i}.
\]
For $m \geqs 2i$, one can check that the bound in \eqref{e:gk4} is sufficient unless $n=2m+1$ and $m \in \{2i,2i+1\}$. The analysis of these remaining cases is essentially identical to the argument given above in the analogous cases with $i$ even. We leave the reader to check the details.
\end{proof}

\subsection{Even dimensional orthogonal groups}\label{ss:orth2}

To complete the proof of Theorem \ref{t:subspace}, it remains to handle the subspace actions of the even dimensional orthogonal groups. 

Our main result is Proposition \ref{p:orth2_sub} below. Note that in part (v), we write $x = (-I_{n-1},I_{1})^{\delta}$ to denote a semisimple involution such that $C_V(x) = \la v \ra$ has discriminant $\delta \in \{\square, \boxtimes\}$ (that is, if $Q$ is the defining quadratic form on $V$, then $Q(v) \in \mathbb{F}_{q}^{\times}$ is a square if $\delta = \square$ and a nonsquare if $\delta = \boxtimes$). Also note that in part (ii), both $(-I_{n-1},I_{1})^{\square}$ and $(-I_{n-1},I_{1})^{\boxtimes}$ have the same number of fixed points on $\O$, so there is no need to specify a label.

\begin{prop}\label{p:orth2_sub}
Let $G \leqs {\rm Sym}(\O)$ be a finite almost simple primitive permutation group with point stabilizer $H$ and socle $G_0 = {\rm P\O}_n^{\e}(q)$ with $n \geqs 8$ even. Assume $H$ is a subspace subgroup of $G$. If $x \in G$ has prime order $r$, then either ${\rm fpr}(x) \leqs (r+1)^{-1}$, or one of the following holds:
\begin{itemize}\addtolength{\itemsep}{0.2\baselineskip}
\item[{\rm (i)}] $H = P_1$, $r=q=2$, $x = (J_2,J_1^{n-2})$ and 
\[
{\rm fpr}(x) = \frac{1}{3}+\frac{2^{n-2}-\e 2^{n/2-1}-2}{3(2^{n/2-1}+\e)(2^{n/2}-\e)}.
\]
\item[{\rm (ii)}] $H = P_1$, $\e=-$, $r=2$, $q=3$, $x = (-I_{n-1},I_1)$ and 
\[
{\rm fpr}(x) = \frac{1}{3}+\frac{2}{3(3^{n/2}+1)}.
\]
\item[{\rm (iii)}] $H = P_1$, $\e=-$, $r=3$, $q=2$, $x = (\omega,\omega^{-1},I_{n-2})$  and 
\[
{\rm fpr}(x) = \frac{1}{4}+\frac{3}{4(2^{n/2}+1)}.
\]
\item[{\rm (iv)}] $H=P_2$, $(\e,n,q)=(+,8,4)$, $r=5$, $x = (\omega, \omega^{-1},I_6)$ or $(\omega I_4, \omega^{-1}I_4)$, and ${\rm fpr}(x) = 1/5$.
\item[{\rm (v)}] $H = N_1$, $r=2$, $q=3$, $x = (-I_{n-1},I_{1})^{\delta}$ and 
\[
{\rm fpr}(x) = \left\{\begin{array}{ll}
\frac{1}{3}+\frac{4}{3(3^{n/2}-1)} & \mbox{if $\e=+$} \\
\frac{1}{3}+\frac{2(3^{n/2-2
}+1)}{3^{n/2-1}(3^{n/2}+1)} & \mbox{if $\e=-$,}
\end{array}\right.
\]
where $\delta = \boxtimes$ if $\e=+$, otherwise $\delta = \square$.

\item[{\rm (vi)}] $H = N_1$, $r=q=2$, $x = (J_2,J_1^{n-2})$ and 
\[
{\rm fpr}(x) = \frac{1}{3}+\frac{2^{n/2-1}+\e}{3(2^{n/2}-\e)}.
\] 
\item[{\rm (vii)}] $H = N_1$, $\e=+$, $r=3$, $q=2$, $x = (\omega, \omega^{-1}, I_{n-2})$ and 
\[
{\rm fpr}(x) = \frac{1}{4}+\frac{3}{4(2^{n/2}-1)}.
\]   
\end{itemize}
\end{prop}

The proof of Proposition \ref{p:orth2_sub} will be given in the sequence of lemmas presented below. In the proofs, we will sometimes write $\ell^{\eta}$ to denote a nondegenerate $\ell$-space of type $\eta$. For example, if $\e=-$ then $V = U \perp W = 2^{+} \perp (n-2)^{-}$ denotes an orthogonal decomposition of $V$ into nondegenerate spaces, where $U$ is a plus-type $2$-space and $W$ is a minus-type space of dimension $n-2$.

\begin{lem}\label{l:orth2_sub1}
The conclusion to Proposition \ref{p:orth2_sub} holds if $H = P_1$. 
\end{lem}

\begin{proof}
Identify $\O$ with the set of totally singular $1$-dimensional subspaces of $V$ and note that 
\[
|\O| = \frac{|{\rm SO}_{n}^{\e}(q)|}{q^{n-2}|{\rm SO}_{n-2}^{\e}(q)||{\rm GL}_{1}(q)|} = \frac{(q^{n/2-1}+\e)(q^{n/2}-\e)}{q-1}.
\]
Let $x \in G$ be an element of prime order $r$. If $G_0 = {\rm P\O}_{8}^{+}(q)$ then  the maximality of $H$ implies that $G$ does not contain any triality automorphisms (see \cite{K} or \cite[Table 8.50]{BHR}, for example). Since field and graph-field automorphisms can be handled in the usual manner, we may assume $x \in {\rm PGO}_{n}^{\e}(q)$ is semisimple or unipotent.

First assume $r=p$, in which case $|C_{\O}(x)|$ is equal to the number of totally singular $1$-spaces in $C_V(x)$. In particular, if $\dim C_V(x) \leqs n-3$ then 
$|C_{\O}(x)| \leqs (q^{n-3}-1)/(q-1)$ and we quickly deduce that ${\rm fpr}(x) \leqs (q+1)^{-1}$. Therefore, we may assume $\dim C_V(x) \geqs n-2$ and we will consider the cases $p=2$ and $p$ odd separately.

Suppose $p=2$ and note that the condition $\dim C_V(x) \geqs n-2$ implies that $x$ is of type $b_1$, $a_2$ or $c_2$ in the notation of \cite{AS}. First assume $x = (J_2,J_1^{n-2})$ is a $b_1$-type involution, which fixes an orthogonal decomposition $U \perp W = 2^{+} \perp (n-2)^{\e}$ of the natural module. Here $C_U(x) = \la u \ra$ is nonsingular and we may assume $Q(u) = 1$. Then $C_{\O}(x)$ comprises the set of totally singular $1$-spaces in $W$, together with the spaces $\la u + w\ra$, where $\la w \ra \subseteq W$ is nonsingular and $Q(w) = 1$. It follows that $|C_{\O}(x)|$ coincides with the total number of $1$-dimensional subspaces of $W$ and thus
\begin{equation}\label{e:oo2}
{\rm fpr}(x) = \frac{q^{n/2-1}-\e}{q^{n/2}-\e}.
\end{equation}
If $q \geqs 4$ then it is easy to check that ${\rm fpr}(x) \leqs 1/3$. However, it $q=2$ then ${\rm fpr}(x) > 1/3$ and this case is recorded in part (i) of Proposition \ref{p:orth2_sub}. 

Now assume $x= (J_2^2,J_1^{n-4})$ is of type $a_2$ or $c_2$. We claim that ${\rm fpr}(x) \leqs 1/3$. If $x= a_2$ then $x$ fixes an orthogonal decomposition $U \perp W = 4^{+} \perp (n-4)^{\e}$ with $C_U(x) = \la u_1,u_2\ra$ totally singular. It follows that $C_{\O}(x)$ comprises the set of $1$-dimensional subspaces of $C_U(x)$, together with every $1$-space of the form $\la u + w \ra$, where $u \in C_U(x)$ and $\la w \ra \subseteq W$ is totally singular. Therefore,
\[
|C_{\O}(x)| = q^2\left(\frac{(q^{n/2-3}+\e)(q^{n/2-2}-\e)}{q-1}\right)+q+1.
\]
Similarly, if $x = c_2$ then $x$ fixes a decomposition $U_1 \perp U_2 \perp W = 2^{+} \perp 2^{-} \perp (n-4)^{-\e}$ with $C_{U_i}(x) = \la u_i \ra$ and $Q(u_i) = 1$. Here $C_{\O}(x)$ comprises the spaces $\la \l(u_1+u_2)+w \ra$ and $\la u+w'\ra$, where $\l \in \mathbb{F}_q$ and $w,w' \in W$, $u \in \la u_1,u_2 \ra$ are vectors such that $Q(w) = 0$ and $Q(w')=Q(u)=1$. Let $\a$ be the number of totally singular $1$-spaces in $W$. Then there are $1+q\a$ spaces of the form $\la \l(u_1+u_2)+w \ra$ and there are $q((q^{n-4}-1)/(q-1)-\a)$ spaces of the form $\la u+w'\ra$, whence 
\[
|C_{\O}(x)| = q\left(\frac{q^{n-4}-1}{q-1}\right)+1. 
\]
In both cases, it is easy to check that ${\rm fpr}(x) \leqs 1/3$.

Next assume $r=p>2$. Since we are free to assume $\dim C_V(x) \geqs n-2$, it follows that $x = (J_2^2,J_1^{n-4})$ or $(J_3,J_1^{n-3})$ because every even size Jordan block must occur with an even multiplicity. First assume $x = (J_2^2,J_1^{n-4})$, which fixes an orthogonal decomposition $U \perp W = 4^{+} \perp (n-4)^{\e}$ with $C_U(x) = \la u_1,u_2\ra$ totally singular. Then $C_{\O}(x)$ comprises every $1$-dimensional subspace of $\la u_1, u_2\ra$, together with the $1$-spaces $\la u+w\ra$, where $u \in \la u_1,u_2\ra$ and $\la w \ra \subseteq W$ is totally singular. Therefore, if $\a$ denotes the number of totally singular $1$-spaces in $W$, then $|C_{\O}(x)| = \a q^2 + q+1$ and it is easy to check that ${\rm fpr}(x) \leqs (q+1)^{-1}$. Now assume $x = (J_3,J_1^{n-3})$, fixing a decomposition $U \perp W = 3 \perp (n-3)$ with $C_U(x) = \la u \ra$. Here $C_{\O}(x)$ comprises $\la u \ra$ and $\la \l u+w \ra$, where $\l \in \mathbb{F}_q$ and $\la w \ra$ is a totally singular $1$-space in $W$. Therefore, $|C_{\O}(x)| = q(q^{n-4}-1)/(q-1)+1$ and once again the desired bound holds.

For the remainder, we may assume $r \ne p$. Suppose $r=2$ and note that the bound in Theorem \ref{t:ls91} is sufficient if $q \geqs 5$, so we may assume $q=3$. Here $|C_{\O}(x)|$ is maximal when $x$ is of the form $(-I_{n-1},I_1)$; there are two such conjugacy classes and elements in both classes have the same number of fixed points. We get $|C_{\O}(x)| = (q^{n-2}-1)/(q-1)$ and thus \eqref{e:oo2} holds. One can now check that ${\rm fpr}(x) \leqs 1/3$ unless $(\e,q) = (-,3)$, which is the case recorded in part (ii) of Proposition \ref{p:orth2_sub}. Now, if $(\e,q) = (-,3)$ and $x$ is some other involution, then $|C_{\O}(x)|$ is maximal when $x = (-I_{n-2},I_{2})$ with a plus-type $(-1)$-eigenspace and we compute
\[
|C_{\O}(x)| = \frac{1}{2}(3^{n/2-2}+1)(3^{n/2-1}-1),
\]
which coincides with the number of totally singular $1$-spaces in the $(-1)$-eigenspace of $x$. Once again, it is easy to check that ${\rm fpr}(x) \leqs 1/3$.

To complete the proof, suppose $r \ne p$ and $r$ is odd. As usual, let $i \geqs 1$ be minimal such that $r$ divides $q^i-1$. Note that if $i \geqs 2$ then $|C_{\O}(x)|$ is the number of totally singular $1$-spaces in $C_V(x)$.

First assume $i$ is even. By the previous observation, $|C_{\O}(x)|$ is maximal when $x = (\L,I_{n-i})$, in which case 
\[
|C_{\O}(x)| \leqs \frac{(q^{(n-i)/2-1}-\e)(q^{(n-i)/2}+\e)}{q-1}
\]
and one can check that this bound yields ${\rm fpr}(x) \leqs (q^{i/2}+2)^{-1}$ unless $\e = -$ and $q=i=2$. Here $r=3$ and we get
\[
{\rm fpr}(x) = \frac{1}{4} + \frac{3}{4(2^{n/2}+1)}.
\]
This special case is recorded in part (iii) of Proposition \ref{p:orth2_sub}. If $\e=-$,  $q=i=2$ and $x$ is any other element of order $3$, then $|C_{\O}(x)|$ is maximal when $x = (\L^2, I_{n-4})$. Here $|C_{\O}(x)| = (2^{n/2-3}-1)(2^{n/2-2}+1)$ is the number of totally singular $1$-spaces in a nondegenerate minus-type space of dimension $n-4$ and we quickly deduce that ${\rm fpr}(x) \leqs 1/4$.

Now assume $i$ is odd. Here $|C_{\O}(x)|$ is maximal when $x = (\L,\L^{-1},I_{n-2i})$, in which case $|C_{\O}(x)| = \a+\b$, where
\[
\a = \frac{(q^{(n-2i)/2-1}+\e)(q^{(n-2i)/2}-\e)}{q-1}
\]
is the number of totally singular $1$-spaces in $C_V(x)$ and we set $\b=2$ if $i=1$, otherwise $\b=0$. It is straightforward to check that ${\rm fpr}(x) \leqs q^{-i}$ and the result follows.
\end{proof}

\begin{lem}\label{l:orth2_sub2}
The conclusion to Proposition \ref{p:orth2_sub} holds if $H = P_m$ with $2 \leqs m \leqs n/2$. 
\end{lem}

\begin{proof}
For $m < n/2$ we identify $\O$ with the set of totally singular $m$-spaces in $V$. On the other hand, if $m=n/2$ then $\e=+$ and we identify $\O$ with one of the two $G_0$-orbits on the set of totally singular $m$-spaces (the actions of $G$ on each orbit are permutation isomorphic, so it does not matter which orbit we choose). Note that
\[
|\O| = \frac{|{\rm SO}_{n}^{\e}(q)|}{q^{m(2n-3m-1)/2}|{\rm SO}_{n-2m}^{\e}(q)||{\rm GL}_{m}(q)|} 
\]  
if $m<n/2$, whereas
\[
|\O| = \frac{|{\rm O}_{n}^{+}(q)|}{2q^{n(n-2)/8}|{\rm GL}_{n/2}(q)|} 
\]
if $m=n/2$. In particular, we have
\begin{equation}\label{e:bd4}
\frac{1}{2}\left(\frac{q}{q+1}\right)q^{m(2n-3m-1)/2} < |\O| < 4q^{m(2n-3m-1)/2}
\end{equation}
for all $m$. Note that if $G_0 = {\rm P\O}_{8}^{+}(q)$ then we may assume $m \in \{2,3\}$ since the action of $G$ on each orbit of totally singular $4$-spaces is permutation isomorphic to its action on the set of totally singular $1$-spaces, which we handled in the previous lemma.

Let $x \in G$ be an element of prime order $r$. As usual, we may assume $x$ is not a field or graph-field automorphism. If
\begin{equation}\label{e:cases}
(n,q) \in \{ (8,2), (8,3), (8,4), (10,2), (12,2)\}
\end{equation}
then the desired result can be checked directly using {\sc Magma}, so for the remainder we will exclude these cases from the analysis. By \cite[Proposition 3.15]{GK} we have
\begin{equation}\label{e:gk5}
{\rm fpr}(x) < 2q^{-(n-2-2\a)/2} + q^{-(n-2\a)/2} + q^{-m},
\end{equation}
where $\a=1$ if $\e=+$, otherwise $\a=0$. This bound immediately implies that ${\rm fpr}(x) \leqs 1/3$ and so we may assume $r$ is odd. In addition, we note that ${\rm fpr}(x) \leqs 1/4$ if $n = 8$ and $q \geqs 4$, which handles the special case where $G_0 = {\rm P\O}_{8}^{+}(q)$ and $x$ is a triality graph automorphism. In particular, for the remainder we may assume $x \in {\rm PGO}_{n}^{\e}(q)$ is semisimple or unipotent.

If $r=p>2$ then the bound in \eqref{e:gk5} implies that ${\rm fpr}(x) \leqs (q+1)^{-1}$, so we may assume $r \ne p$ and $r$ is odd. Let $i \geqs 1$ be minimal such that $r$ divides $q^i-1$. There are two cases to consider, according to the parity of $i$.

First assume $i$ is even. If $m<i$ then $|C_{\O}(x)|$ is maximal when $x = (\L,I_{n-i})$, in which case $|C_{\O}(x)|$ is the number of totally singular $m$-spaces in the $(-\e)$-type space $C_V(x)$. By applying the bounds in \eqref{e:bd4}, we deduce that
\[
{\rm fpr}(x) < 8\left(\frac{q+1}{q}\right)q^{-mi} \leqs (q^{i/2}+2)^{-1}
\]
and the result follows. If $m \geqs i+1$ then one can check that the bound in \eqref{e:gk5} is sufficient (recall that we may exclude the cases in \eqref{e:cases}). Therefore, we may assume $m=i$. If $n \geqs 2m+4$ then \eqref{e:gk5} is sufficient unless $m=q=2$, in which case $r=3$ and $x = (\L^{\ell},I_{n-2\ell})$ for some $\ell \geqs 1$. In addition, we may assume $n \geqs 14$ (see \eqref{e:cases}). Let $\b$ be the number of totally singular $2$-spaces in $C_V(x)$. If $\ell=1$ then $|C_{\O}(x)|=\b$ and we compute  
\[
{\rm fpr}(x) \leqs \frac{(2^{n/2-3}-\e)(2^{n/2-2}-\e)}{(2^{n/2-1}-\e)(2^{n/2}-\e)} \leqs \frac{1}{4}
\]
as required. Now assume $\ell \geqs 2$. By arguing as in the proof of Lemma \ref{l:sub_useful}(ii) we deduce that $|C_{\O}(x)| \leqs \b + (2^{2\ell}-1)/3$ and once again we conclude that ${\rm fpr}(x) \leqs 1/4$.

To complete the analysis of the case $i$ even, we may assume $m=i \geqs 4$ and $n \in \{2m,2m+2\}$. Suppose $n=2m$, in which case $x$ is of the form $(\L,I_m)$, $(\L^2)$ or $(\L_1,\L_2)$ (the latter two possibilities only occur when $\e=+$). Here $x$ has fixed points on $\O$ if and only if $x = (\L^2)$, so let us assume $x$ is of this form. We may embed $\hat{x}$ as a scalar in a subgroup ${\rm GU}_{2}(q^{m/2}) < {\rm O}_{2m}^{+}(q)$ and we note that there is a bijective correspondence between the set of totally singular $m$-spaces in $V$ fixed by $x$ and the set of totally singular $1$-spaces in the natural module for ${\rm GU}_{2}(q^{m/2})$. This implies that $|C_{\O}(x)| = q^{m/2}+1$ and the bound ${\rm fpr}(x) \leqs (q^{m/2}+2)^{-1}$ quickly follows. Similarly, if $n=2m+2$ then we may assume $x = (\L^2,I_2)$ and the result follows since $|C_{\O}(x)| = q^{m/2}+1$. 

Now assume $i$ is odd. If $m<i$ then it is plain to see that $|C_{\O}(x)|$ is maximal when $x = (\L,\L^{-1},I_{n-2i})$, in which case $C_{\O}(x)$ is the set of totally singular $m$-spaces in $C_V(x)$ and by applying the bounds in \eqref{e:bd4} we deduce that 
\[
{\rm fpr}(x) < 8\left(\frac{q+1}{q}\right)q^{-2mi} \leqs q^{-i}
\]
and the result follows. Now assume $m \geqs i$. If $i=1$ then $q \geqs 4$ and one checks that the bound in \eqref{e:gk5} implies that ${\rm fpr}(x) \leqs q^{-1}$. Now assume $i \geqs 3$ and note that $r$ divides $t = (q^i-1)/(q-1)$. If $n<4i$ then $x = (\L,\L^{-1},I_{n-2i})$ is the only possibility and we get $|C_{\O}(x)| = 2\b$, where $\b$ is the number of totally singular $(m-i)$-spaces in $C_V(x)$, which is a nondegenerate $(n-2i)$-space of type $\e$. Therefore,
\[
{\rm fpr}(x) < 8\left(\frac{q+1}{q}\right)q^{-i(2n-2m-i-1)/2}
\]
and we deduce that ${\rm fpr}(x) \leqs (t+1)^{-1}$. 

To complete the proof, suppose $i \geqs 3$ is odd, $m \geqs i$ and $n \geqs 4i$. If $m \geqs i+1$ then the bound in \eqref{e:gk5} is sufficient unless $(n,m,q,i) = (14,4,2,3)$, in which case  $r=7$ and $x = (\L,\L^{-1},I_8)$ or $((\L,\L^{-1})^2,I_2)$. With the aid of {\sc Magma}, it is easy to check that ${\rm fpr}(x) \leqs 1/8$ in both cases. 

Finally, suppose $m=i$. In this case, one checks that \eqref{e:gk5} is sufficient unless $q=2$ and $r = 2^i-1$ is a Mersenne prime. Here we proceed as in the final paragraph in the proof of Lemma \ref{l:symp_sub2}. In particular, we first observe that $|C_{\O}(x)|$ is maximal when $x$ is of the form $((\L,\L^{-1})^{\ell},I_{n-2m\ell})$ for some $\ell \geqs 1$ and we compute 
\[
|C_{\O}(x)| \leqs \gamma + 2\left(\frac{2^{m\ell}-1}{2^m-1}\right),
\]
where $\gamma$ is the number of totally singular $m$-spaces in $C_V(x)$. It is now routine to check that this yields ${\rm fpr}(x) \leqs 2^{-m}$ as required. 
\end{proof}

If $G_0 = {\rm P\O}_{8}^{+}(q)$ and $G$ contains triality graph or graph-field automorphisms, then $G$ has a maximal parabolic subgroup $H$ of type $P_{1,3,4}$. This special case is handled in the following lemma.

\begin{lem}\label{l:orth2_triality}
The conclusion to Proposition \ref{p:orth2_sub} holds if $G_0 = {\rm P\O}_{8}^{+}(q)$ and $H = P_{1,3,4}$.  
\end{lem}

\begin{proof}
Let $x \in G$ be an element of prime order $r$ and note that 
\[
|\O| = \frac{|{\rm O}_{8}^{+}(q)|}{2q^{11}|{\rm GL}_{2}(q)||{\rm GL}_{1}(q)|^2} = (q^4+q^2+1)(q^2+1)^2(q+1)^3.
\]
The usual argument applies if $x$ is a field or graph-field automorphism. Next assume $x$ is a triality graph automorphism, so $r=3$. If $q \geqs 7$ then the bound in Theorem \ref{t:ls91} is sufficient, so we may assume $q \leqs 5$. In each of these cases we can use {\sc Magma} to show that $|C_{\O}(x)| \leqs (q^6-1)/(q-1)$, which yields ${\rm fpr}(x) \leqs 1/225$, with equality if $q=2$ and $C_{G_0}(x) = G_2(2)$. (For $q=5$ we can construct $H$ by observing that $H = N_G(P)$, where $P$ is a suitable index-five subgroup of a Sylow $5$-subgroup of $G_0$.)

For the remainder, we may assume $x \in {\rm PGO}_{8}^{+}(q)$ is either semisimple or unipotent. Since ${\rm fpr}(x) = |x^G \cap H|/|x^G|$, it follows that ${\rm fpr}(x)$ is at most the corresponding fixed point ratio for the action of $x$ on totally singular $1$-spaces. Therefore, our earlier analysis of the case $H = P_1$ (see Lemma \ref{l:orth2_sub1}) immediately implies that ${\rm fpr}(x) \leqs (r+1)^{-1}$ unless $q = 2$ and $x = b_1$ is a transvection. But here an easy {\sc Magma} computation yields ${\rm fpr}(x) = 1/15$ and the result follows.
\end{proof}

\begin{lem}\label{l:orth2_sub3}
The conclusion to Proposition \ref{p:orth2_sub} holds if $H = N_1$ and $q$ is odd. 
\end{lem}

\begin{proof}
Here we may assume $\O$ is the set of nondegenerate $1$-dimensional subspaces of $V$ with square discriminant (recall that we refer to such a subspace as a  \emph{square $1$-space}). Note that 
\[
|\O| = \frac{|{\rm SO}_{n}^{\e}(q)|}{2|{\rm SO}_{n-1}(q)|} = \frac{1}{2}q^{n/2-1}(q^{n/2}-\e).
\]  
We may assume $x \in {\rm PGO}_{n}^{\e}(q)$ has prime order $r$, noting that the maximality of $H$ implies that $G$ does not contain any triality automorphisms when $(n,\e) = (8,+)$. 

First assume $r=p$ and note that $|C_{\O}(x)|$ is the number of square $1$-spaces in $C_V(x)$. In particular, if $\dim C_V(x) \leqs n-3$ then $|C_{\O}(x)| \leqs (q^{n-3}-1)/(q-1)$ and it is easy to check that ${\rm fpr}(x) \leqs (q+1)^{-1}$. Therefore, we may assume $\dim C_V(x) = n-2$, in which case $x$ has Jordan form $(J_2^2,J_1^{n-4})$ or $(J_3,J_1^{n-3})$. If $x = (J_2^2,J_1^{n-4})$ then 
\[
|x^G \cap H| = \frac{|{\rm SO}_{n-1}(q)|}{q^{2n-9}|{\rm Sp}_{2}(q)||{\rm SO}_{n-5}(q)|},\;\; |x^G| = \frac{|{\rm SO}_{n}^{\e}(q)|}{q^{2n-7}|{\rm Sp}_{2}(q)||{\rm SO}_{n-4}^{\e}(q)|}
\]
and we deduce that 
\[
{\rm fpr}(x) = \frac{q^{(n-4)/2}-\e}{q^{n/2}-\e}.
\]
Similarly, if $x = (J_3,J_1^{n-3})$ then
\[
|x^G \cap H| \leqs \frac{|{\rm SO}_{n-1}(q)|}{q^{n-3}|{\rm SO}_{n-4}^{+}(q)|},\;\; |x^G| = \frac{|{\rm SO}_{n}^{\e}(q)|}{2q^{n-2}|{\rm SO}_{n-3}(q)|}
\]
and we get
\[
{\rm fpr}(x) \leqs \frac{2(q^{(n-4)/2}+1)}{q^{n/2}-\e}.
\]
In both cases, one can check that ${\rm fpr}(x) \leqs (q+1)^{-1}$ and the result follows.

Next assume $r=2$. Here $|C_{\O}(x)|$ is equal to the number of square $1$-spaces in the eigenspaces of $x$ on $V$. As a consequence, we see that $|C_{\O}(x)|$ is maximal when $x = (-I_{n-1},I_1)$ fixes an orthogonal decomposition $V = U \perp W$ with $W = C_V(x)$. There are two cases to consider, according to the discriminant of the defining quadratic form $Q$ restricted to $W$. First assume this discriminant is a square, so $x = (-I_{n-1},I_1)^{\square}$, $W$ is contained in $\O$ and we have $|C_{\O}(x)| = 1+\a$, where $\a$ is the number of square $1$-spaces in $U$. If $Y \subseteq U$ is such a subspace, then the orthogonal complement of $Y$ in $U$ is a nondegenerate $(n-2)$-space of type $\eta$ and we get 
\[
\a = \frac{1}{2}q^{(n-2)/2}(q^{(n-2)/2}+\eta).
\]
If $\eta = -$, then it is easy to check that ${\rm fpr}(x) \leqs 1/3$. The same conclusion holds if $\eta =+$ and $q \geqs 5$. However, if $\eta = +$ and $q=3$ then ${\rm fpr}(x) > 1/3$ and this case is recorded in part (v) of Proposition \ref{p:orth2_sub}. Moreover, one can check that $\eta=-\e$, so we only get ${\rm fpr}(x)>1/3$ when $\e=-$. Similarly, if $x = (-I_{n-1},I_1)^{\boxtimes}$ then $|C_{\O}(x)|=\a$ as above with $\eta=\e$ and we deduce that ${\rm fpr}(x) \leqs 1/3$ unless $(q,\e) = (3,+)$. Again, this special case is recorded in Proposition \ref{p:orth2_sub}. To summarize, ${\rm fpr}(x) > 1/3$ if and only if $q=3$ and $x = (-I_{n-1},I_1)^{\delta}$, where $\delta = \boxtimes$ if $\e=+$, and $\delta = \square$ if $\e=-$.

Finally, let us assume $r \ne p$ and $r$ is odd. As usual, let $i \geqs 1$ be minimal such that $r$ divides $q^i-1$ and note that $|C_{\O}(x)|$ is equal to the number of square $1$-spaces in $C_V(x)$. In particular, if $i$ is even then $|C_{\O}(x)|$ is maximal when $x = (\L,I_{n-i})$ and we compute
\[
|C_{\O}(x)| \leqs \frac{1}{2}q^{(n-i)/2-1}(q^{(n-i)/2}+\e), 
\]
which implies that ${\rm fpr}(x) \leqs (q^{i/2}+2)^{-1}$.
Similarly, if $i$ is odd then $|C_{\O}(x)|$ is maximal when $x = (\L,\L^{-1},I_{n-2i})$ and we get
\[
|C_{\O}(x)| \leqs \frac{1}{2}q^{(n-2i)/2-1}(q^{(n-2i)/2}-\e). 
\]
This gives ${\rm fpr}(x) \leqs q^{-i}$ and the result follows.
\end{proof}

\begin{lem}\label{l:orth2_sub4}
The conclusion to Proposition \ref{p:orth2_sub} holds if $H = N_1$ and $q$ is even. 
\end{lem}

\begin{proof}
Here we identify $\O$ with the set of nonsingular $1$-dimensional subspaces of $V$, whence
\[
|\O| = \frac{|\O_{n}^{\e}(q)|}{|{\rm Sp}_{n-2}(q)|} = q^{n/2-1}(q^{n/2}-\e).
\]
Note that if $G_0 = \O_8^{+}(q)$ then the maximality of $H$ implies that $G$ does not contain any triality automorphisms. Also note that if $G = {\rm O}_{n}^{\e}(q)$, then $H = 2 \times {\rm Sp}_{n-2}(q)$ is the centralizer of a $b_1$-type involution.

Let $x \in G$ be an element of prime order $r$. As usual, we may assume $x \in {\rm O}_{n}^{\e}(q)$ is semisimple or unipotent. First assume $r=2$. By arguing as in the proof of the previous lemma, we may assume $\dim C_V(x) \geqs n-2$ and thus $x$ is an involution of type $b_1$, $a_2$ or $c_2$ in the notation of \cite{AS}. In addition, we may assume $q=2$ since the bound in Theorem \ref{t:ls91} is sufficient if $q \geqs 4$. If $x = b_1$ then
\[
|x^G \cap H| = 1+\frac{|{\rm Sp}_{n-2}(2)|}{2^{n-3}|{\rm Sp}_{n-4}(2)|} = 2^{n-2},\;\; |x^G| = |\O|
\]
and we deduce that 
\[
{\rm fpr}(x) = \frac{1}{3}+\frac{2^{n/2-1}+\e}{3(2^{n/2}-\e)}.
\]
This special case is recorded in part (vi) of Proposition \ref{p:orth2_sub}. Similarly, if  $x=a_2$ then we compute
\[
|x^G \cap H| = \frac{|{\rm Sp}_{n-2}(2)|}{2^{2n-9}|{\rm Sp}_{2}(2)||{\rm Sp}_{n-6}(2)|},\;\; |x^G| = \frac{|{\rm O}_{n}^{\e}(2)|}{2^{2n-7}|{\rm Sp}_{2}(2)||{\rm O}_{n-4}^{\e}(2)|} 
\]
and this gives
\[
{\rm fpr}(x) = \frac{2^{n/2-2}-\e}{2^{n/2}-\e} \leqs \frac{1}{3}.
\]
And for $x = c_2$ we obtain ${\rm fpr}(x) = 2^{n/2-2}/(2^{n/2}-\e)$ and once again ${\rm fpr}(x) \leqs 1/3$.

To complete the proof, we may assume $r$ is odd, so $x$ is semisimple. Let $i \geqs 1$ be minimal such that $r$ divides $q^i-1$ and observe that $C_{\O}(x)$ is the set of nonsingular $1$-spaces in $C_V(x)$. Suppose $i$ is even. Then $|C_{\O}(x)|$ is maximal when $x = (\L,I_{n-i})$, in which case
\[
|C_{\O}(x)| = q^{(n-i)/2-1}(q^{(n-i)/2}+\e)
\]
and we deduce that ${\rm fpr}(x) \leqs (q^{i/2}+2)^{-1}$ unless $\e=+$ and $q=i=2$. Here $r=3$ and we get
\[
{\rm fpr}(x) = \frac{1}{4}+\frac{3}{4(2^{n/2}-1)}
\]
as recorded in part (vii) of Proposition \ref{p:orth2_sub}. Note that if $q=2$ and $x$ is any other element of order $r=3$, then $|C_{\O}(x)|$ is maximal when $x = (\L^2,I_{n-4})$ and it is easy to check that ${\rm fpr}(x) \leqs 1/4$. Finally, if $i$ is odd then $|C_{\O}(x)|$ is maximal when $x = (\L,\L^{-1},I_{n-2i})$, so 
\[
|C_{\O}(x)| \leqs q^{n/2-i-1}(q^{n/2-i}-\e)
\]
and it is straightforward to check that ${\rm fpr}(x) \leqs q^{-i}$.
\end{proof}

\begin{lem}\label{l:orth2_sub5}
The conclusion to Proposition \ref{p:orth2_sub} holds if $H = N_m^{\eta}$ with $2 \leqs m \leqs n/2$. 
\end{lem}

\begin{proof}
Here we identify $\O$ with an orbit of nondegenerate $m$-dimensional subspaces of $V$ of type $\eta$. More precisely, if $m$ is even then $\O$ is the complete set of subspaces of the given type and either $m < n/2$, or $(\e,\eta)=(-,+)$ and $m=n/2$. On the other hand, if $m$ is odd then $m<n/2$, $q$ is odd and we may assume that $\O$ is the set of nondegenerate $m$-spaces with square discriminant. In all cases, let us observe that
\[
\frac{1}{4}\left(\frac{q}{q+1}\right)q^{m(n-m)} < |\O| = \frac{|{\rm O}_{n}^{\e}(q)|}{|{\rm O}_{m}^{\eta}(q)||{\rm O}_{n-m}^{\eta'}(q)|}<2q^{m(n-m)},
\]
where $\e = \eta\eta'$.

Let $x \in G$ be an element of prime order $r$. As usual, the desired bound quickly follows if $x$ is a field or graph-field automorphism, while the maximality of $H$ implies that $G$ does not contain triality automorphisms when $(n,\e) = (8,+)$. Therefore, for the remainder we may assume $x \in {\rm PGO}_{n}^{\e}(q)$ is unipotent or semisimple. By inspecting \cite[Proposition 3.16]{GK} we deduce that
\begin{equation}\label{e:gk7}
{\rm fpr}(x) < 2q^{-(n-4)/2}+q^{-(n-2)/2}+q^{-m} + q^{-(n-m-\a)/2},
\end{equation}
where $\a=2$ if $m$ is even, otherwise $\a=1$. The groups with $(n,q)$ as in \eqref{e:cases} can be handled directly using {\sc Magma}, so we will exclude these cases for the remainder of the proof.

First assume $r=p=2$. If $q \geqs 4$ then the bound in Theorem \ref{t:ls91} is sufficient,  so we may assume $q=2$. As noted above, we may also assume that $n \geqs 14$ and one can check that the bound in \eqref{e:gk7} is effective unless $(n,m) = (14,2)$. Here $H = N_2^{-}$ (as noted in \cite{KL}, $N_2^{+}$ is non-maximal when $q \leqs 3$) and $|C_{\O}(x)|$ is maximal when $x = (J_2,J_1^{12})$ is a $b_1$-type involution, in which case
\[
|x^G \cap H| \leqs \frac{1}{2}|{\rm O}_{2}^{-}(2)| + \frac{|{\rm O}_{12}^{-}(2)|}{2|{\rm Sp}_{10}(2)|} = 2083
\]
and we deduce that ${\rm fpr}(x) \leqs 1/3$ since $|x^G| \geqs 8128$.

Next assume $r=p>2$. Here the bound in \eqref{e:gk7} is sufficient for $n \geqs 10$, so we may assume $n=8$ and $q \geqs 5$. If $m \leqs 3$ then \eqref{e:gk7} is good enough unless $(m,q) = (2,5)$. In fact, by carefully inspecting \cite[Proposition 3.16]{GK}, we may assume that $(\e,\eta) = (+,-)$ and with the aid of {\sc Magma} one checks that ${\rm fpr}(x) \leqs 1/620$ (maximal if $x$ has Jordan form $(J_3,J_1^5)$). Finally, suppose $(n,m) = (8,4)$. Here $\e=-$ and we are free to assume that $H = N_4^{+}$, which means that \cite[Proposition 3.16]{GK} yields ${\rm fpr}(x) \leqs 3q^{-2}+2q^{-4} \leqs (q+1)^{-1}$ for $q \geqs 5$.

To complete the proof, we may assume $r \ne p$. If $r=2$ then one checks that the bound in \eqref{e:gk7} is sufficient (recall that we may assume $q \geqs 5$ when $n=8$). Now assume $r$ is odd and let $i \geqs 1$ be minimal such that $r$ divides $q^i-1$. As before, we note that $|C_{\O}(x)|$ is maximal when $x$ is of the form $(\L^{\ell},I_{n-\ell i})$ (if $i$ is even) or $((\L,\L^{-1})^{\ell},I_{n-2\ell i})$ (if $i$ is odd) for some $\ell \geqs 1$.

Suppose $i$ is even. If $m<i$ then $C_{\O}(x)$ is the set of $\eta$-type $m$-spaces in $C_V(x)$, so $|C_{\O}(x)|$ is maximal when $x = (\L,I_{n-i})$ and we deduce that 
\[
{\rm fpr}(x) < 8\left(\frac{q+1}{q}\right)q^{-mi} \leqs (q^{i/2}+2)^{-1}
\]
as required. 

Now assume $m \geqs i$ (with $i$ even). If $i=2$ then one can check that \eqref{e:gk7} is sufficient unless $m=q=2$ and $n \geqs 14$. Here $\eta=-$ (since $H$ is a maximal subgroup of $G$), $r=3$ and we compute
\begin{align*}
|x^G \cap H| & \leqs 2\left(\frac{|{\rm SO}_{n-2}^{-}(q)|}{|{\rm GU}_{\ell-1}(q)||{\rm SO}_{n-2\ell}^{+}(q)|}\right) + \frac{|{\rm SO}_{n-2}^{-}(q)|}{|{\rm GU}_{\ell}(q)||{\rm SO}_{n-2\ell-2}^{+}(q)|} \\
& < 2q^{2n\ell - 3\ell^2-\ell}\left(2q^{-2n+2\ell+2}+q^{-4\ell}\right)
\end{align*}
and
\[
|x^G| > \frac{1}{2}\left(\frac{q}{q+1}\right)q^{2n\ell-3\ell^2-\ell}.
\]
The resulting upper bound on ${\rm fpr}(x)$ is sufficient unless $\ell=1$. For $\ell=1$, we verify the bound ${\rm fpr}(x) \leqs 1/4$ by working with the precise values for $|x^G \cap H|$ and $|x^G|$. 

Now assume $m \geqs i \geqs 4$ (we are continuing to assume that $i$ is even). If $m<n/2$ then \eqref{e:gk7} is sufficient unless $(n,m,q,i) = (10,4,3,4)$. Here $r=5$ and $x = (\L,I_6)$ or $(\L^2,I_2)$. For $x = (\L,I_6)$ we get 
\[
|x^G \cap H| \leqs \frac{|{\rm O}_{4}^{-}(3)|}{|{\rm GU}_{1}(9)|} + \frac{|{\rm SO}_{6}^{-}(3)|}{|{\rm GU}_{1}(9)||{\rm SO}_{2}^{+}(3)|} < 2\cdot 3^{12}
\]
and the bound $|x^G|>\frac{1}{8}3^{29}$ is clearly sufficient. The case $x = (\L^2,I_2)$ is similar. Finally, suppose $n=2m$, so $\e=-$ and we can slightly strengthen the bound in \eqref{e:gk7} by replacing the term $q^{-(n-2)/2}$ by $q^{-n/2}$ (see \cite[Proposition 3.16]{GK}). Working with this modified bound, we can now reduce the problem to the cases where $(m,q,i) = (6,3,6)$ or $(5,3,4)$, or $m=i=4$. In the latter case, $r$ divides $q^2+1$ and $x = (\L,I_4)$ since $\e=-$. In particular, we calculate
\[
{\rm fpr}(x) = \frac{2}{q^8(q^4+q^2+1)(q^4+1)} \leqs (q^2+2)^{-1}
\]
and the result follows. The two other special cases that we need to consider can be handled in a similar fashion.

Finally, let us assume $i$ is odd. If $m < 2i$ then $|C_{\O}(x)|$ is maximal when $x = (\L,\L^{-1},I_{n-2i})$ and we deduce that
\[
{\rm fpr}(x) < 8\left(\frac{q+1}{q}\right)q^{-2mi} \leqs q^{-i}.
\]
On the other hand, if $m \geqs 2i$ then one can check that the bound in \eqref{e:gk7} is sufficient, noting that $r$ divides $(q^i-1)/(q-1)$ if $i \geqs 3$.
\end{proof}

\vs

This completes the proof of Proposition \ref{p:orth2_sub}, which in turn completes the proof of Theorem \ref{t:subspace} and our analysis of subspace actions of classical groups.

\section{Product type groups}\label{s:product}

In this section, we complete the proof of Theorem \ref{t:main} by handling the product-type primitive groups. Here we have $G \leqs L \wr S_k$ in its product action on $\O = \Gamma^k$, where $k \geqs 2$ and $L \leqs {\rm Sym}(\Gamma)$ is a primitive group of diagonal or almost simple type. In addition, if $T$ denotes the socle of $L$, then $T^k$ is the socle of $G$ and the subgroup of $S_k$ induced by the conjugation action of $G$ on the $k$ factors of $T^k$ is transitive. There are two cases to consider.

\begin{prop}\label{p:pt1}
If $L$ is a diagonal type group, then ${\rm fpr}(x) \leqs (r+1)^{-1}$ for all $x \in G$ of prime order $r$.
\end{prop}

\begin{proof}
Write $x = (x_1, \ldots, x_k)\pi \in G$, where $x_i \in L$ and $\pi \in S_k$. If $\pi = 1$ then $x_i^r = 1$ for all $i$ and at least one $x_i$ is nontrivial, whence 
\[
{\rm fpr}(x) = \prod_i {\rm fpr}(x_i,\Gamma) \leqs (r+1)^{-1}
\]
by Proposition \ref{p:diagonal}. Now assume $\pi \ne 1$, in which case $\pi$ has cycle-shape $(r^h,1^{k-hr})$ for some $h \geqs 1$. In this situation, a straightforward computation with the product action shows that 
$|C_{\O}(x)| \leqs |\Gamma|^{k - h(r-1)}$ and thus
\[
{\rm fpr}(x) \leqs |\Gamma|^{-h(r-1)} \leqs |\Gamma|^{1-r} \leqs (r+1)^{-1}
\]
since $|\Gamma| \geqs 60$.
\end{proof}

\begin{prop}\label{p:pt2}
Suppose $L \leqs {\rm Sym}(\Gamma)$ is almost simple with point stabilizer $J$ and let 
\[
x = (x_1, \ldots, x_k)\pi \in G
\]
be an element of prime order $r$. Then ${\rm fpr}(x) > (r+1)^{-1}$ only if $\pi=1$ and one of the following holds (up to permutation isomorphism):
\begin{itemize}\addtolength{\itemsep}{0.2\baselineskip}
\item[{\rm (i)}] $L=S_n$ or $A_n$ acting on $\ell$-element subsets of $\{1, \ldots, n\}$ with $1 \leqs \ell < n/2$.
\item[{\rm (ii)}] $x$ is conjugate to $(x_1,1, \ldots, 1)$ and $(L,J,x_1)$ is one of the cases in parts (b)-(d) of Theorem \ref{t:main}(i).
\end{itemize}
\end{prop}

\begin{proof}
Suppose ${\rm fpr}(x) > (r+1)^{-1}$. By arguing as in the proof of the previous proposition, noting that $|\Gamma| \geqs 5$, we see that $\pi=1$ and thus ${\rm fpr}(x) = \prod_i {\rm fpr}(x_i,\Gamma)$. In addition, we may assume $L$ is neither $S_n$ nor $A_n$ acting on $\ell$-sets. If $x_1$ is nontrivial, then either ${\rm fpr}(x_1,\Gamma) \leqs (r+1)^{-1}$ or $(L,J,x_1)$ is one of the special cases arising in parts (b)-(d) of Theorem \ref{t:main}(i). By Corollary \ref{c:main}, which is an easy consequence of Theorem \ref{t:main} for almost simple groups (see below), we see that 
\[
{\rm fpr}(x_1,\Gamma) \cdot {\rm fpr}(x_2,\Gamma) \leqs (r+1)^{-1}
\]
if $x_1$ and $x_2$ are both nontrivial. Therefore, we conclude that $x$ is conjugate to $(x_1, 1, \ldots, 1)$ and the proof is complete.
\end{proof} 

\vs

This completes the proof of Theorem \ref{t:main}. Finally, let us prove Corollary \ref{c:main}.

\begin{proof}[Proof of Corollary \ref{c:main}]
Let $G \leqs {\rm Sym}(\O)$ be a finite primitive group and let $x \in G$ be an element of prime order $r$. By Theorem \ref{t:main}, we may assume that $G \leqs L \wr S_k$ acts on $\O = \Gamma^k$ with the product action, where $L \leqs {\rm Sym}(\Gamma)$ is almost simple and $k \geqs 1$ (note that the desired conclusion clearly holds if $G$ is an affine group since $r+1 <r^2$). In addition, we can exclude the special case where $L$ is permutation isomorphic to $S_n$ or $A_n$ acting on $\ell$-element subsets of $\{1, \ldots, n\}$. If $k =1$ then $G$ is almost simple and the desired conclusion follows by inspecting the special cases that arise in parts (b)-(d) of Theorem \ref{t:main}(i). Finally, by combining this observation with Proposition \ref{p:pt2}, we conclude that ${\rm fpr}(x) \leqs (r+1)^{-1/2}$ when $k \geqs 2$.
\end{proof}

\section{Minimal index}\label{s:appl}

Let $G \leqs {\rm Sym}(\O)$ be a primitive permutation group of degree $m$ with point stabilizer $H$. Recall that 
\[
{\rm Ind}(G) = \min\{ {\rm ind}(x) \,:\, 1 \ne x \in G\}
\]
is the \emph{minimal index} of $G$, where 
\[
{\rm ind}(x) = m - {\rm orb}(x) = m - \frac{1}{|x|}\sum_{y \in \la x \ra}|C_{\O}(y)|
\]
is the \emph{index} of $x$ and ${\rm orb}(x)$ is the number of orbits of $x$ on $\O$. If we wish to specify the set $\O$, we will write ${\rm Ind}(G,\O)$ and ${\rm ind}(x,\O)$. 

Observe that if $x \in G$ has order $r$, then ${\rm orb}(x) \geqs m/r$ and thus ${\rm ind}(x) \leqs m(1-1/r)$. In particular, if $x$ is an involution then ${\rm ind}(x) \leqs m/2$, with equality if and only if $x$ acts fixed point freely on $\O$. Consequently, if $|G|$ is even then ${\rm Ind}(G) \leqs m/2$, with equality only if $|H|$ is odd (of course, if $|H|$ is even then $G$ contains involutions with fixed points).

In this final section we prove Theorems \ref{t:mind} and \ref{t:mind2}. We will also establish Theorem \ref{t:odd} in the special case where $|G|$ is odd. We begin with the following easy lemma.

\begin{lem}\label{l:prime}
If ${\rm Ind}(G) = {\rm ind}(x)$ then $x$ has prime order.
\end{lem}

\begin{proof}
Suppose otherwise, say $|x|=pq$ with $p$ a prime and $q>1$. Then $z = x^q$ has order $p$ and each $\la x \ra$-orbit is a union of $\la z \ra$-orbits. Since ${\rm Ind}(G) = {\rm ind}(x)$, it follows that every $\la x \ra$-orbit is a $\la z \ra$-orbit and vice versa. But if $\{\a,\a^z, \ldots, \a^{z^{p-1}}\}$ is an orbit of $\la z \ra$ of length $p$, then $\la x^p \ra$ is the stabilizer of $\a$ in $\la x \ra$ and we deduce that ${\rm orb}(x^p) > {\rm orb}(x)$, which is a contradiction. 
\end{proof}

Next we handle two important special cases.

\begin{prop}\label{p:mind1}
Let $G = S_n$ or $A_n$ acting on $\ell$-element subsets of $\{1, \ldots, n\}$, where $n \geqs 5$ and $1 \leqs \ell < n/2$. Then
\[
{\rm Ind}(G) = \left\{\begin{array}{ll}
\binom{n-2}{\ell-1} & \mbox{if $G = S_n$} \\
\frac{1}{2}\left(\binom{n}{\ell} - \binom{n-4}{\ell} - 2\binom{n-4}{\ell-2}-\binom{n-4}{\ell-4}\right) & 
\mbox{if $G = A_n$.}
\end{array}\right.
\]
Moreover, ${\rm Ind}(G) = {\rm ind}(x)$ if and only if
\begin{itemize}\addtolength{\itemsep}{0.2\baselineskip}
\item[{\rm (i)}] $G = S_n$ and $x$ is a transposition;
\item[{\rm (ii)}] $G = A_n$ and $x$ is a double transposition; or
\item[{\rm (iii)}] $G = A_n$, $\ell=1$ and $x$ is a $3$-cycle.
\end{itemize}
\end{prop}

\begin{proof}
If $1 \ne y \in G$ has odd order, then ${\rm ind}(y)$ is minimal when $y$ is a $3$-cycle, in which case 
\[
{\rm ind}(y) = \frac{2}{3}\left(\binom{n}{\ell}-\binom{n-3}{\ell}-\binom{n-3}{\ell-3}\right).
\]
Now let $x \in G$ be an involution and observe that ${\rm ind}(x)$ is minimal when $x$ is a transposition (for $G = S_n$) or a double transposition (for $G = A_n$). If $x$ is a transposition, then 
\[
{\rm ind}(x) = \binom{n-2}{\ell-1} < {\rm ind}(y).
\]
Similarly, if $x$ is a double transposition then
\[
{\rm ind}(x) = \frac{1}{2}\left(\binom{n}{\ell} - \binom{n-4}{\ell} - 2\binom{n-4}{\ell-2}-\binom{n-4}{\ell-4}\right)
\]
and this is strictly less than ${\rm ind}(y)$ when $\ell \geqs 2$. However, if $\ell=1$ then ${\rm ind}(x) = {\rm ind}(y) = 2$. The result follows. 
\end{proof}

\begin{prop}\label{p:mind2}
Let $G \leqs {\rm Sym}(\O)$ be an almost simple primitive classical group of degree $m$ in a subspace action with socle $G_0$ and point stabilizer $H$, where $(G,H)$ is one of the cases appearing in Table \ref{tab:class}. Then the following hold:
\begin{itemize}\addtolength{\itemsep}{0.2\baselineskip}
\item[{\rm (i)}] ${\rm Ind}(G) = {\rm ind}(x)$ only if $|x| \in \{2,3\}$.
\item[{\rm (ii)}] ${\rm Ind}(G) = {\rm ind}(x)$ for some involution $x \in G$.
\item[{\rm (iii)}] ${\rm Ind}(G) = {\rm ind}(x)$ for some element $x \in G$ of order $3$ if and only if $G = {\rm L}_2(8){:}3$ and $H =P_1$, in which case $m=9$ and ${\rm Ind}(G) = 4$.
\item[{\rm (iv)}] ${\rm Ind}(G)< m/4$ if and only if $(G,H,m,{\rm Ind}(G))$ is one of the cases in Table \ref{tab:subc}. In each of these cases, ${\rm Ind}(G) \geqs 3m/14$.
\end{itemize}
\end{prop}

\renewcommand{\arraystretch}{1.2}
{\small \begin{table}
\[
\begin{array}{lllll} \hline
G & H & m & {\rm Ind}(G) & {\rm Conditions} \\ \hline
{\rm U}_4(2).2 & P_2 & 27 & 6 \\
{\rm Sp}_n(2) & {\rm O}_{n}^{-}(2) & 2^{n/2-1}(2^{n/2}-1) & 2^{n/2-2}(2^{n/2-1}-1) & n \geqs 6 \\
{\rm O}_{n}^{-}(2) & P_1 & (2^{n/2-1}-1)(2^{n/2}+1) & 2^{n/2-2}(2^{n/2-1}-1) & n \geqs 8 \\
{\rm O}_{n}^{+}(2) & N_1 & 2^{n/2-1}(2^{n/2}-1) & 2^{n/2-2}(2^{n/2-1}-1) & n \geqs 8 \\ \hline
\end{array}
\]
\caption{The subspace actions in part (iv) of Proposition \ref{p:mind2}}
\label{tab:subc}
\end{table}}
\renewcommand{\arraystretch}{1}

\begin{proof}
First observe that $|H|$ is even (for example, see \cite[Theorem 2]{LS911}) and thus ${\rm Ind}(G)<m/2$ as noted above. Let $x \in G$ be an element such that ${\rm Ind}(G) = {\rm ind}(x)$, so $x$ has prime order $r$ by Lemma \ref{l:prime}. Seeking a contradiction, suppose $r \geqs 5$. 

If ${\rm fpr}(y) \leqs r^{-1}$ for every element $y \in G$ of order $r$, then 
\[
{\rm ind}(x) \geqs \left(1-\frac{1}{r}\right)^2m>\frac{m}{2} > {\rm Ind}(G)
\]
and we have reached a contradiction. Therefore, we may assume ${\rm fpr}(y)>r^{-1}$ for some $y \in G$ of order $r$. Then Corollary \ref{c:main0} implies that $G_0 = {\rm L}_2(q)$, $H = P_1$ and $r = q-1 \geqs 7$ is a Mersenne prime. Here $m=q+1$ and $|C_{\O}(x)|=2$, so ${\rm ind}(x) = q-2 > m/2$ and once again this is a contradiction. This proves part (i) and so for the remainder we may assume $r \in \{2,3\}$.

For $r \in \{2,3\}$, set 
\[
f_r = \max\{{\rm fpr}(x)\,:\, |x|=r\},\;\; m_r = \min\{ {\rm ind}(x) \,:\, |x|=r\}
\]
and note that ${\rm Ind}(G) = \min\{m_2,m_3\}$. If $f_3 \leqs 1/4$ then ${\rm ind}(x) \geqs m/2>{\rm Ind}(G)$ for every $x \in G$ of order $3$, in which case ${\rm Ind}(G) = m_2$ (so (ii) holds) and $(G,H)$ does not arise as a special case in part (iii). In addition, if $f_2 \leqs 1/3$ then ${\rm ind}(x) \geqs m/3$ for every involution $x$ and thus $(G,H)$ does not appear in part (iv). Therefore, to complete the proof of the proposition we may assume $f_2>1/3$ or $f_3>1/4$, in which case $(G,H,x)$ is one of the special cases in Table \ref{tab:class} with $|x|=2$ or $3$. We now inspect each of these cases in turn, computing $m_2$, and also $m_3$ if $f_3>1/4$. 

\vs

\noindent \emph{Case 1. $G_0 = {\rm L}_n(q)$.}

\vs

Here $H = P_1$, $m = (q^n-1)/(q-1)$ and we may assume $q \in \{2,3,4\}$ or $(n,q) = (2,8)$ since we are only interested in the cases where $f_2>1/3$ or $f_3>1/4$. If $(n,q) = (2,8)$ then $G = {\rm L}_2(8){:}3$, $m=9$ and ${\rm Ind}(G) = 4$, with ${\rm ind}(x)=4$ if and only if $x$ is an involution or a field automorphism of order $3$ (in particular, this is the special case recorded in part (iii)). If $q=2$ then $f_3 \leqs 1/4$ and ${\rm Ind}(G) = m_2 = {\rm ind}(x)$ with $x = (J_2,J_1^{n-2})$. Here ${\rm fpr}(x)$ is recorded in Table \ref{tab:class} and we compute ${\rm Ind}(G) = 2^{n-2}>m/4$. Similarly, if $q=3$ then $m_3 = 2.3^{n-2}$ and 
\[
m_2 = \left\{\begin{array}{ll}
\frac{1}{2}(3^{n-1}-1) & \mbox{if $(-I_{n-1},I_1) \in G$} \\
2(3^{n-2}-1) & \mbox{otherwise,}
\end{array}\right.
\]
whence ${\rm Ind}(G) = m_2 > m/4$. Finally, suppose $q=4$. Here $f_3>1/4$ if and only if $G$ contains an element of order $3$ of the form $x = (\omega,I_{n-1})$ (modulo scalars), in which case $m_3 = {\rm ind}(x) =  2(4^{n-1}-1)/3$. If $x \in G$ is an involution, then ${\rm ind}(x)$ is minimal when $x = (J_2,J_1^{n-2})$ is a transvection and we compute $m_2 = 2^{2n-3}$. We conclude that ${\rm Ind}(G) = m_2>m/4$.

\vs

\noindent \emph{Case 2. $G_0 = {\rm U}_n(q)$, $n \geqs 3$.}

\vs

First assume $H=P_1$, so $n \geqs 5$ is odd, $q=2$ and $m = (2^n+1)(2^{n-1}-1)/3$. Here $f_3>1/4$ if and only if $G$ contains an element $x = (\omega, I_{n-1})$ of order $3$, in which case $m_3 = {\rm ind}(x) = 2^{n-1}(2^{n-1}-1)/3$. If $x \in G$ is an involution, then ${\rm ind}(x)$ is minimal when $x = (J_2,J_1^{n-2})$. By inspecting the proof of Lemma \ref{l:psu_sub1}, we see that $|C_{\O}(x)| =4(2^{n-2}+1)(2^{n-3}-1)/3+1$ and thus $m_2 = {\rm ind}(x) = 2^{2n-4}$. In particular, ${\rm Ind}(G) = m_2 >m/4$.

Now suppose $H = P_2$, so $n=4$ and $q \in \{2,3\}$. Here it is straightforward to check that ${\rm Ind}(G) = m_2<m_3$. More precisely, if $q=2$ then $m=27$ and either $G = G_0$ and ${\rm Ind}(G) = 10>m/4$, or $G = G_0.2$ and ${\rm Ind}(G) = 6<m/4$ (so the latter case is recorded in part (iv)). Similarly, if $q=3$ then $m=112$ and ${\rm Ind}(G) \geqs 36 > m/4$.

Finally suppose $H = N_1$, so $n \geqs 4$ is even, $q=2$ and $m = 2^{n-1}(2^n-1)/3$. If $r=3$ then ${\rm ind}(x)$ is minimal when $x = (\omega,I_{n-1})$ and we compute ${\rm ind}(x) = 2(2^{2n-3}-2^{n-2}-1)/3$. Similarly, if $r=2$ then ${\rm ind}(x)$ is minimal when $x = (J_2,J_1^{n-2})$ and by inspecting the proof of Lemma \ref{l:psu_sub3} we deduce that ${\rm ind}(x) = 2^{2n-4}$. Therefore, ${\rm Ind}(G) = 2^{2n-4}>m/4$ and ${\rm Ind}(G) = {\rm ind}(x)$ if and only if $x = (J_2,J_1^{n-2})$. 

\vs

\noindent \emph{Case 3. $G_0 = {\rm PSp}_n(q)$, $n \geqs 4$.}

\vs

First assume $H = P_1$, so $m = (q^n-1)/(q-1)$ and we may assume $q \in \{2,3\}$. If $q=2$ then ${\rm ind}(x)$ is minimal when $x = (J_2,J_1^{n-2})$ and we compute ${\rm ind}(x) = 2^{n-2}$. Therefore, ${\rm Ind}(G) = 2^{n-2}>m/4$, with ${\rm Ind}(G) = {\rm ind}(x)$ if and only if $x = (J_2,J_1^{n-2})$. Now assume $q=3$. If $r=3$, then ${\rm ind}(x)$ is minimal when $x = (J_2,J_1^{n-2})$ and we calculate ${\rm ind}(x) = 2.3^{n-2}$. For $r=2$, we find that ${\rm ind}(x)$ is minimal when $x = (-I_2,I_{n-2})$. Here the proof of Lemma \ref{l:symp_sub1} gives $|C_{\O}(x)| = (3^{n-2}-1)/2+4$ and we deduce that ${\rm ind}(x) = 2(3^{n-2}-1)$. Therefore, ${\rm Ind}(G) = 2(3^{n-2}-1)>m/4$ and ${\rm Ind}(G) = {\rm ind}(x)$ only if $x$ is an involution.

Next suppose $n \geqs 6$, $q=2$ and $H = {\rm O}_{n}^{\e}(2)$, so $m = 2^{n/2-1}(2^{n/2}+\e)$. Here $f_3>1/4$ if and only if $\e=-$, in which case $m_3 = 2^{n/2-1}(2^{n/2-1}-1)$ (with ${\rm ind}(x)$ minimal if $x = (\L,I_{n-2})$). If $r = 2$ then ${\rm ind}(x)$ is minimal when $x = (J_2,J_1^{n-2})$ and we compute ${\rm ind}(x) = 2^{n/2-2}(2^{n/2-1}+\e)$. Therefore, ${\rm Ind}(G) = 2^{n/2-2}(2^{n/2-1}+\e)$, which is less than $m/4$ if and only if $\e=-$ (note that in this situation we have ${\rm Ind}(G) \geqs 3m/14$, with equality if $n=6$). We also deduce that ${\rm Ind}(G) = {\rm ind}(x)$ only if $x$ is an involution.

\vs

\noindent \emph{Case 4. $G_0 = \O_n(q)$, $n \geqs 7$ odd, $q$ odd.} 

\vs

Here $q=3$, $H = P_1$ or $N_1^{-}$, and $f_3 \leqs 1/4$, so we may assume $r=2$. First assume $H = P_1$, so $m = (3^{n-1}-1)/2$. If $G = {\rm SO}_n(3)$ then ${\rm ind}(x)$ is minimal when $x=(-I_{n-1},I_1)^{+}$, in which case ${\rm ind}(x) = 3^{(n-3)/2}(3^{(n-1)/2}-1)/2$. Similarly, if $G = \O_n(3)$ and $(-I_{n-1},I_1)^{+} \not\in G$, then ${\rm ind}(x)$ is minimal when $x = (-I_{n-1},I_1)^{-}$ and we compute ${\rm ind}(x) = 3^{(n-3)/2}(3^{(n-1)/2}+1)/2$. In conclusion, if $H = P_1$ then 
${\rm Ind}(G) = 3^{(n-3)/2}(3^{(n-1)/2}-\delta)/2$, where $\delta = 1$ if $(-I_{n-1},I_1)^{+} \in G$, otherwise $\delta = -1$. In particular, ${\rm Ind}(G)>m/4$ and ${\rm Ind}(G) = {\rm ind}(x)$ only if $x$ is an involution.

Now assume $H = N_1^{-}$, so $m = 3^{(n-1)/2}(3^{(n-1)/2}-1)$. As in the previous case, we only need to consider involutions of type $(-I_{n-1},I_1)^{\e}$. If $G = {\rm SO}_n(3)$, then ${\rm ind}(x)$ is minimal when $x = (-I_{n-1},I_1)^{-}$ and we get ${\rm ind}(x) = (3^{n-2}-2.3^{(n-3)/2}-1)/2$. And if $G = \O_n(3)$ does not contain $(-I_{n-1},I_1)^{-}$, then ${\rm ind}(x)$ is minimal when $x = (-I_{n-1},I_1)^{+}$ and we calculate that ${\rm fpr}(x) = 1/3$, which yields ${\rm ind}(x) = 3^{(n-3)/2}(3^{(n-1)/2}-1)/2$. Therefore, 
\[
{\rm Ind}(G) = \left\{\begin{array}{ll}
(3^{n-2}-2.3^{(n-3)/2}-1)/2 & \mbox{if $(-I_{n-1},I_1)^{-} \in G$} \\
3^{(n-3)/2}(3^{(n-1)/2}-1)/2 & \mbox{otherwise.}
\end{array}\right.
\]
Once again we conclude that ${\rm Ind}(G)>m/4$ and ${\rm Ind}(G) = {\rm ind}(x)$ only if $x$ is an involution.

\vs

\noindent \emph{Case 5. $G_0 = {\rm P\O}_n^{\e}(q)$, $n \geqs 8$ even.}

\vs

First assume $H = P_1$, so $m = (q^{n/2-1}+\e)(q^{n/2}-\e)/(q-1)$ and $q \in \{2,3\}$. Suppose $q=3$, in which case $\e=-$, $f_3 \leqs 1/4$ and we may assume $G$ contains a reflection $x = (-I_{n-1},I_1)$ (otherwise $f_2 \leqs 1/3$). Then we compute $m_2 = {\rm ind}(x) = 3^{n/2-1}(3^{n/2-1}-1)/2>m/4$. 

Now assume $q=2$. If $r=3$, then we may assume $\e=-$, in which case ${\rm ind}(x)$ is minimal when $x = (\L,I_{n-2})$. Here we compute ${\rm ind}(x) = 2^{n/2-1}(2^{n/2-1}-1)$. Now assume $r=2$. If $G = {\rm O}_n^{\e}(2)$ then ${\rm ind}(x)$ is minimal when $x = (J_2,J_1^{n-2})$ is a $b_1$-type involution and we calculate ${\rm ind}(x) =  2^{n/2-2}(2^{n/2-1}+\e)$. Therefore, if $G = {\rm O}_n^{\e}(2)$ then ${\rm Ind}(G) = 2^{n/2-2}(2^{n/2-1}+\e)$, which is less than $m/4$ if and only if $\e=-$ (note that if $\e=-$ then it is easy to check that ${\rm Ind}(G)>3m/14$). Now suppose $G = \O_n^{\e}(2)$. Here ${\rm ind}(x)$ is minimal when $x=a_2$ or $c_2$. If $x=a_2$ then $|C_{\O}(x)|$ is given in the proof of Lemma \ref{l:orth2_sub1} and we compute ${\rm ind}(x) = 3.2^{n-4}$. Similarly, if $x = c_2$ then ${\rm ind}(x) = 2^{n/2-2}(3.2^{n/2-2}+\e)$. Since $3.2^{n-4} < 2^{n/2-1}(2^{n/2-1}-1)$, we conclude that if $G = \O_n^{\e}(2)$ then 
\[
{\rm Ind}(G) = \left\{\begin{array}{ll}
3.2^{n-4} & \mbox{if $\e=+$} \\
2^{n/2-2}(3.2^{n/2-2}-1) & \mbox{if $\e=-$}
\end{array}\right.
\]
is greater than $m/4$ and ${\rm Ind}(G) = {\rm ind}(x)$ only if $x$ is an involution. 

To complete the analysis of orthogonal groups, we may assume $H = N_1$ and $q \in \{2,3\}$. Note that $m = q^{n/2-1}(q^{n/2}-\e)$, where $d = (2,q-1)$. 

First assume $q=2$. If $r=3$ then we may assume $\e=+$, $x = (\L,I_{n-2})$ and we compute ${\rm ind}(x) = 2^{n/2-1}(2^{n/2-1}-1)$. Now suppose $r=2$. If $G = {\rm O}_n^{\e}(2)$ then ${\rm ind}(x)$ is minimal when $x = b_1$, in which case ${\rm ind}(x) = 2^{n/2-2}(2^{n/2-1}-\e)$. So in this situation, we have ${\rm Ind}(G) = 2^{n/2-2}(2^{n/2-1}-\e)$, which is less than $m/4$ if and only if $\e=+$ (here one checks that ${\rm Ind}(G) \geqs 3m/14$). Now assume $G = \O_n^{\e}(2)$. Here ${\rm ind}(x)$ is minimal when $x = a_2$ or $c_2$. We calculated ${\rm fpr}(x)$ in the proof of Lemma \ref{l:orth2_sub4} and we deduce that ${\rm ind}(x) = 3.2^{n-4}$ if $x=a_2$ and ${\rm ind}(x) = 2^{n/2-2}(3.2^{n/2-2}-\e)$ if $x = c_2$. Now $3.2^{n-4}<2^{n/2-1}(2^{n/2-1}-1)$ and thus
\[
{\rm Ind}(G) = \left\{\begin{array}{ll}
3.2^{n-4} & \mbox{if $\e=-$} \\
2^{n/2-2}(3.2^{n/2-2}-1) & \mbox{if $\e=+$}
\end{array}\right.
\]
and ${\rm Ind}(G) = {\rm ind}(x)$ only if $x$ is an involution. In addition, ${\rm Ind}(G)>m/4$. 

Finally, suppose $q=3$. We may assume $r=2$ since $f_3 \leqs 1/4$. First assume $\e=+$ and note that we may assume $x = (-I_{n-1},I_1)^{\boxtimes} \in G$ (otherwise $f_2 \leqs 1/3$). Then ${\rm Ind}(G) = {\rm ind}(x) = 3^{n/2-1}(3^{n/2-1}-1)/2$. Similarly, if $\e=-$ then we may assume $G$ contains $x = (-I_{n-1},I_1)^{\square}$ and we deduce that ${\rm Ind}(G) = (3^{n-2}-1)/2$. In both cases, ${\rm Ind}(G)>m/4$ and the result follows.
\end{proof}

With these two special cases in hand, we are now ready to prove Theorem \ref{t:mind}.

\begin{proof}[Proof of Theorem \ref{t:mind}]
Assume $|G|$ is even and recall that ${\rm Ind}(G) \leqs m/2$, with equality only if $|H|$ is odd. Let $x \in G$ be an element of order $r$ such that ${\rm Ind}(G) = {\rm ind}(x)$. By Lemma \ref{l:prime}, $r$ is a prime. Seeking a contradiction, suppose $r \geqs 5$. By arguing as in the proof of Proposition \ref{p:mind2}, it follows that ${\rm fpr}(y)>r^{-1}$ for some $y \in G$ of order $r$, so by applying Theorem \ref{t:main} we deduce that 
\begin{itemize}\addtolength{\itemsep}{0.2\baselineskip}
\item[{\rm (a)}] $G$ is almost simple; or
\item[{\rm (b)}] $G \leqs L \wr S_k$ is a product type primitive group with its product action on $\O = \Gamma^k$, where $k \geqs2$ and $L \leqs {\rm Sym}(\Gamma)$ is almost simple.
\end{itemize}

If $G$ is almost simple then the possibilities for $(G,H)$ are described in Corollary \ref{c:main0} and the result follows via Propositions \ref{p:mind1} and \ref{p:mind2}. 

Now assume (b) holds and let $y = (y_1, \ldots, y_k)\pi \in G$ be an element of prime order $r \geqs 5$ with ${\rm fpr}(y)>r^{-1}$. Then Theorem \ref{t:main} (also see Remark \ref{r:main}(c)) implies that $\pi=1$ and either $L=S_n$ or $A_n$ acting on $\ell$-element subsets of $\{1, \ldots, n\}$, or $y = (y_1, 1, \ldots, 1)$ up to conjugacy. In the latter case, we have ${\rm fpr}(y) = {\rm fpr}(y_1,\Gamma)$, 
\[
{\rm ind}(y) = |\Gamma|^{k-1} \cdot {\rm ind}(y_1,\Gamma)
\]
and ${\rm Ind}(L,\Gamma) < {\rm ind}(y_1,\Gamma)$ by the result for almost simple groups handled in case (a). Similarly, if $L = S_n$ or $A_n$ acting on $\ell$-sets, then Proposition \ref{p:mind1} implies that ${\rm Ind}(G) = {\rm ind}(x)$ only if $x$ has order $2$ or $3$. 

We have now proved that ${\rm Ind}(G) = {\rm ind}(x)$ only if $|x| \in \{2,3\}$. To complete the proof of Theorem \ref{t:mind}, let us assume we have equality for some element $x$ of order $3$. If ${\rm fpr}(x) \leqs 1/4$ then ${\rm ind}(x) \geqs m/2$, so we must have ${\rm fpr}(x)=1/4$, ${\rm Ind}(G) = m/2$ and every involution in $G$ acts fixed point freely on $\O$ (otherwise there would be an involution $y \in G$ with ${\rm ind}(y)<m/2$). In particular, this forces $|H|$ to be odd and we have ${\rm Ind}(G)={\rm ind}(y)$ for every involution $y \in G$.

So to complete the argument, we may assume ${\rm fpr}(x)>1/4$, which implies that $(G,H,x)$ is one of the cases arising in Theorem \ref{t:main}. 

If $G$ is almost simple then by applying Propositions \ref{p:mind1} and \ref{p:mind2} we deduce that one of the following holds (up to permutation isomorphism):
\begin{itemize}\addtolength{\itemsep}{0.2\baselineskip}
\item[{\rm (i)}] $G = A_m$ in its natural action of degree $m$ and ${\rm Ind}(G) = {\rm ind}(x) = 2$ with $x$ a $3$-cycle. 
\item[{\rm (ii)}] $G = {\rm L}_2(8){:}3$, $H =P_1$, $m=9$ and ${\rm Ind}(G) = {\rm ind}(x) = 4$ with $x$ a field automorphism of order $3$.
\end{itemize}

Next assume $G = V{:}H$ is an affine group with socle $(C_p)^d$ and point stabilizer $H \leqs {\rm GL}_d(3)$ as in part (ii) of Theorem \ref{t:main}. Here $p=3$, $x$ is conjugate to a transvection in $H$ and ${\rm fpr}(x) = 1/3$, so ${\rm ind}(x) = 4m/9$. Recall that we are assuming $|G|$ is even, so $H$ contains involutions. We claim that $H$ contains an involution $y$ with $\dim C_V(y) \geqs d-2$. In particular, if $H$ contains a reflection $y=(-I_1, I_{d-1})$ then ${\rm fpr}(y) = 1/3$ and ${\rm ind}(y) =m/3<{\rm ind}(x)$, whence ${\rm Ind}(G) = m/3$ and $G$ does not contain an element of order $3$ with ${\rm Ind}(G) = {\rm ind}(x)$. On the other hand, if $H$ does not contain such an element, then 
${\rm ind}(y)$ is minimal when $y \in H$ is an involution with $\dim C_V(y) = d-2$, in which case ${\rm ind}(y) = {\rm Ind}(G) = 4m/9$. This is the case recorded in part (ii)(b) of Theorem \ref{t:mind}.

Therefore, it remains to justify the claim. To do this, we apply a theorem of McLaughlin \cite{McL}. Recall that we are assuming $H$ contains a transvection, so we can consider the normal subgroup $H_0$ generated by the transvections in $H$. Since $H$ acts irreducibly on $V = (\mathbb{F}_3)^d$, it follows that $H_0$ is semisimple, preserving a direct sum decomposition $V = V_1 \oplus \cdots \oplus V_t$. Moreover, $H_0$ acts on $V$ as a direct product $H_1 \times \cdots \times H_t$, where each $H_i \leqs {\rm GL}(V_i)$ is either ${\rm SL}(V_i)$ or ${\rm Sp}(V_i)$. Therefore, $H_0$ contains involutions of the form $(-I_2,I_{d-2})$ and the claim follows.  

Finally, let us turn to the product type groups in part (iii) of Theorem \ref{t:main}. Here  $G \leqs L \wr S_k$ acts on $\O = \Gamma^k$ with its product action, where $k \geqs 2$ and $L \leqs {\rm Sym}(\Gamma)$ is one of the almost simple primitive groups in part (i) of Theorem \ref{t:main}. Recall that we are assuming there exists an element $x \in G$ of order $3$ with ${\rm fpr}(x)>1/4$ and we note that $x$ is of the form $(x_1, \ldots, x_k) \in L^k$ (see Remark \ref{r:main}(c)). Therefore, up to permutation isomorphism, $L$ is either $S_n$ or $A_n$ acting on $\ell$-element subsets of $\{1, \ldots, n\}$, or $L$ is a classical group in a subspace action as in Table \ref{tab:class}. 

First assume $L = S_n$ or $A_n$ acting on $\ell$-sets, where $1 \leqs \ell < n/2$. If $\ell \geqs 2$ then Proposition \ref{p:mind1} implies that ${\rm Ind}(G) = {\rm ind}(x)$ only if $x$ is an involution, so we may assume $\ell=1$. Now, if $x \in G$ has order $3$, then ${\rm ind}(x)$ is minimal when $x = (x_1,1, \ldots,1) \in (A_n)^k$ and $x_1 \in A_n$ is a $3$-cycle, in which case we compute ${\rm ind}(x) = 2m/n$. If $y \in G \cap (S_n)^k$ is of the form $(y_1, \ldots, y_t, 1, \ldots, 1)$ and each $y_i$ is a transposition, then 
\[
{\rm ind}(y) = \frac{m}{2}\left(1 - \left(1-\frac{2}{n}\right)^t\right)
\]
and we deduce that ${\rm ind}(y) < 2m/n$ if and only if $t \in \{1,2\}$, or if $t=3$ and $n=5$. This gives the conclusion recorded in part (d) of Theorem \ref{t:mind}(ii). 

Finally, suppose $L$ is a classical group in a subspace action. By applying part (iii) of Proposition \ref{p:mind2} we deduce that $G = L \wr P$ is the only possibility, where $L = {\rm L}_2(8){:}3$, $\Gamma$ is the set of $1$-dimensional subspaces of the natural module for ${\rm L}_2(8)$ and $P \leqs S_k$ is transitive. Here $m = 9^k$, ${\rm Ind}(G) = 4.9^{k-1}$ and ${\rm Ind}(G) = {\rm ind}(x)$ if and only if $x$ is conjugate to $(x_1,1, \ldots, 1)$, where $x_1 \in L$ is either an involution or a field automorphism of ${\rm L}_2(8)$ of order $3$. 
\end{proof}

Next we prove Theorem \ref{t:mind2}.

\begin{proof}[Proof of Theorem \ref{t:mind2}]
As before, $|G|$ is even and thus ${\rm Ind}(G) \leqs m/2$, with equality only if $|H|$ is odd. By Theorem \ref{t:mind}, there exists an involution $x \in G$ with ${\rm Ind}(G) = {\rm ind}(x)$. If ${\rm fpr}(x) \leqs 1/2$, then ${\rm ind}(x) \geqs m/4$, so we may assume ${\rm fpr}(x)>1/2$ and thus $(G,H,x)$ is one of the special cases appearing in the statement of Theorem \ref{t:main}. If $G$ is almost simple, then by applying Corollary \ref{c:main0} we reduce to the case where $G$ is a classical group in a subspace action and the result follows via Proposition \ref{p:mind2}(iv). 

Finally, we may assume $G \leqs L \wr S_k$ is a product type primitive group with its product action on $\O = \Gamma^k$, where $k \geqs 2$ and $L \leqs {\rm Sym}(\Gamma)$ is an almost simple primitive group with point stabilizer $J$. In addition, we may assume $L$ is not $S_n$ or $A_n$ acting on $\ell$-element subsets of $\{1, \ldots, n\}$ (since this is covered by case (ii) in Theorem \ref{t:mind2}). Therefore, our involution $x$ with ${\rm fpr}(x)>1/2$ must be conjugate to $(x_1, 1, \ldots, 1)$ in $L^k$ and thus 
\[
{\rm Ind}(G) = {\rm ind}(x) = |\Gamma|^{k-1} \cdot {\rm ind}(x_1,\Gamma).
\]
It follows that
\[
{\rm Ind}(L,\Gamma) = {\rm ind}(x_1,\Gamma)<\frac{1}{4}|\Gamma|
\]
and thus $(L,J,|\Gamma|,{\rm Ind}(L,\Gamma))$ is one of the cases in Table \ref{tab:subc}. 

Let $T$ denote the socle of $L$. If $L={\rm Sp}_n(2)$ with $n \geqs 6$ then $L=T$ and thus $G = L \wr P$ for some transitive group $P \leqs S_k$. In the three remaining cases in Table \ref{tab:subc}, we have $|L:T|=2$, but in each case every involution $y \in L$ with ${\rm Ind}(L,\Gamma) = {\rm ind}(y,\Gamma)$ is contained in $L \setminus T$ (for $T = {\rm U}_4(2)$, $y$ is an involutory graph automorphism with $C_{T}(y) = {\rm Sp}_4(2)$, while $y$ is a $b_1$-type involution for the cases with $T = \O_n^{\e}(2)$). Therefore, $G$ must contain $L^k$ and thus $G = L \wr P$ for some transitive group $P \leqs S_k$. 
\end{proof}

Finally, let us consider the minimal index of a primitive group $G \leqs {\rm Sym}(\O)$ of odd order. Here $G$ is solvable by the Feit-Thompson theorem, so $G = V{:}H$ is an affine group with socle $V = (C_p)^d$ and point stabilizer $H \leqs {\rm GL}_d(p)$ for some odd prime $p$ and positive integer $d$.

\begin{thm}\label{t:odd}
Let $G = V{:}H$ be a primitive permutation group of odd order with socle $V = (C_p)^d$ and point stabilizer $H$. Then
\[
\min\left\{ m\left(1-\frac{3}{2r+1}\right), m\left(1-\frac{1}{p}\right)^2 \right\} \leqs {\rm Ind}(G) \leqs m\left(1-\frac{1}{r}\right)
\]
where $r$ is the smallest prime divisor of $|G|$. 
\end{thm}

\begin{proof}
As previously noted, if $x \in G$ has order $r$, then ${\rm orb}(x) \geqs m/r$ and thus 
\[
{\rm Ind}(G) \leqs {\rm ind}(x) \leqs m\left(1-\frac{1}{r}\right).
\]
Next let $x \in G$ be an element of prime order $s$. We may assume ${\rm fpr}(x)>0$ (otherwise ${\rm ind}(x) = m(1-1/s) \geqs m(1-1/r)$), so by replacing $x$ by a suitable conjugate we may assume $x \in H$. Then ${\rm fpr}(x) = p^{e-d}$, where $e = \dim C_V(x)$. 

Suppose $s=p$. If $e=d-1$ then $x$ is a transvection, ${\rm fpr}(x) = p^{-1}$ and we compute 
\[
{\rm ind}(x) = m\left(1-\frac{1}{p}\right)^2.
\]
On the other hand, if $e<d-1$ then ${\rm fpr}(x) \leqs p^{-2}$ and thus 
\[
{\rm ind}(x) \geqs m\left(1 - \frac{p^2+p-1}{p^3}\right) > m\left(1-\frac{1}{p}\right)^2.
\]
Now suppose $s \ne p$. Then $s \leqs (p^{d-e}-1)/2$ since $p$ is odd and thus ${\rm fpr}(x) \leqs 1/(2s+1)$. In turn, this implies that
\[
{\rm ind}(x) \geqs m\left(1-\frac{1}{s}\left(1+\frac{s-1}{2s+1}\right)\right) = m\left(1-\frac{3}{2s+1}\right) \geqs m\left(1-\frac{3}{2r+1}\right)
\]
and the result follows.
\end{proof}

\vspace{2mm}

\noindent \textbf{Acknowledgements.} We thank an anonymous referee for their careful reading of the paper and helpful suggestions. Burness thanks the Department of Mathematics at the University of Padua for their generous hospitality during a research visit in autumn 2021. Guralnick was partially supported by the NSF grant DMS-1901595 and a Simons Foundation Fellowship 609771.

\end{document}